\documentclass{amsart}

\usepackage{amsmath,amsfonts,amssymb}

\def\R{{\mathbb R}}
\def\N{{\mathbb N}}

\def\Z{{\mathbb Z}}
\def\ind{{1\!\!1}}
\def\E{\mathbb{E}}
\def\P{\mathbb{P}}

\def\cal{\mathcal}

\def\supp{{\rm{supp}}}
\def\endsup{M}

\def\theequation{\arabic{section}.\arabic{equation}}

\newcounter{bean}
\newcommand{\benuma}{\setlength{\labelwidth}{.25in}
\begin{list}%
{(\alph{bean})}{\usecounter{bean}}}
\newcommand{\eenuma}{\end{list}}

\newcommand{\beginsec}{\setcounter{equation}{0}}

\newtheorem{theorem}{Theorem}[section]
\newtheorem{remark}[theorem]{Remark}
\newtheorem{ex}[theorem]{Example}
\newtheorem{lemma}[theorem]{Lemma}
\newtheorem{cor}[theorem]{Corollary}
\newtheorem{defn}[theorem]{Definition}

\newtheorem{ass}{Assumption}
\newtheorem{prop}[theorem]{Proposition}

\newcommand{\sm}{\setminus}

\newcommand{\noi}{\noindent }
\newcommand{\ba}{\begin{array}}
\newcommand{\ea}{\end{array}}
\newcommand{\bea}{\begin{eqnarray}}
\newcommand{\eea}{\end{eqnarray}}
\newcommand{\beas}{\begin{eqnarray*}}
\newcommand{\eeas}{\end{eqnarray*}}
\newcommand{\be}{\begin{equation}}
\newcommand{\ee}{\end{equation}}
\newcommand{\bt}{\begin{theorem}}
\newcommand{\et}{\end{theorem}}
\newcommand{\bc}{\begin{center}}
\newcommand{\ec}{\end{center}}
\newcommand{\ben}{\begin{enumerate}}
\newcommand{\een}{\end{enumerate}}
\newcommand{\lan}{\langle}
\newcommand{\ran}{\rangle}
\newcommand{\ei}{\end{itemize}}
\newcommand{\rear}{\renewcommand{\arraystretch}}

\newcommand{\nrm}[1]{\left\Vert #1 \right\Vert}

\newcommand{\ds}{\displaystyle}

\newcommand{\ve}{\varepsilon}

\def\zeroset{\cal Z}

\newcommand{\inv}{\mbox{inv}}

\newcommand{\ra}{\rightarrow}

\def\vconv{\stackrel{v}{\rightarrow}}
\def\wconv{\stackrel{w}{\rightarrow}}
\def\swconv{\Rightarrow}

\def\mb{{\cal M}_{\leq B} (\R_+)}

\def\eem{{\cal M}(E)}
\def\eemf{{\cal M}_F(E)}
\def\eemb{{\cal M}_{\leq B} (E)}
\def\eemsp{{\cal M}_{\leq 1} (E)}

\def\eemone{{\cal M}_1(E)}

\def\mm{{\cal M}[0,M)}
\def\mmf{{\cal M}_F[0,M)}
\def\mmsp{{\cal M}_{\leq 1} [0,M)}
\def\mmb{{\cal M}_{\leq B} [0,M)}

\def\dspaceh{{\cal D}_{\cal H}[0,\infty)}
\def\dTspaceh{{\cal D}_{\cal H}[0,T]}

\def\incspace{{\cal I}_{\R_+}[0,\infty)}

\def\ecb{{\cal C}_b(E)}
\def\ecc{{\cal C}_c (E)}
\def\eocd{{\cal C}^1(E)}
\def\eocdc{{\cal C}^1_c(E)}
\def\eocdb{{\cal C}^1_b(E)}

\def\cb{{\cal C}_b(\R_+)}
\def\newccp{{\cal C}_c(\R_+^2)}
\def\cc{{\cal C}_c (\R_+)}

\def\newocdcp{{\cal C}^1_c(\R_+^2)}
\def\mnewocdcp{{\cal C}^1_{c,M} (\R_+^2)}
\def\mnewccp{{\cal C}_{c,M} (\R_+^2)}

\def\ocdb{{\cal C}^1_b(\R_+)}

\def\newcb{{\cal C}_b(\R^2)}

\def\cbm{{\cal C}_b[0,M)}
\def\newccpm{{\cal C}_c([0,M)\times\R_+)}
\def\ccm{{\cal C}_c [0,M)}

\def\newocdcpm{{\cal C}^1_c([0,M)\times\R_+)}
\def\ocdcm{{\cal C}^1_c[0,M)}

\def\ocdbm{{\cal C}^1_b[0,M)}
\def\newcbpm{{\cal C}_b([0,M)\times\R_+)}

\def\tnewf{\tilde{\varphi}}
\def\tnewft{\tnewf(\cdot,t)}

\def\tnewfs{\tnewf(\cdot,s)}
\def\dxtnewfs{\tnewf_x(\cdot,s)}
\def\dttnewfs{\tnewf_s(\cdot,s)}

\def\newf{\varphi}
\def\newft{\varphi(\cdot,t)}

\def\newfs{\varphi(\cdot,s)}
\def\dxnewfs{\newf_x(\cdot,s)}
\def\dtnewfs{\newf_s(\cdot,s)}

\def\f1{{\bf 1}}
\def\zerof{{\bf 0}}

\def\newchi{f}
\def\fren{\overline{R}_E^{(N)}}
\def\ren{R_E^{(N)}}
\def\idlen{I^{(N)}}
\def\tdn{\tilde{d}^{(N)}}

\def\wn{w^{(N)}}
\def\vn{v^{(N)}}

\def\fv{\overline{v}}
\def\fw{\overline{w}}
\def\fvn{\overline{v}^{(N)}}
\def\fwn{\fw^{(N)}}

\def\qn{Q^{(N)}}
\def\dn{D^{(N)}}
\def\kn{K^{(N)}}
\def\kninv{\sigma^{(N)}}
\def\fkn{\overline{K}^{(N)}}

\def\en{E^{(N)}}

\def\nun{\nu^{(N)}}
\def\agen{a^{(N)}}

\def\cI{{\cal I}}

\def\measn{\nu^{(N)}}

\def\fqn{\overline{Q}^{(N)}}

\def\fe{\overline{E}}
\def\fen{\overline{E}^{(N)}}
\def\fmeasn{\overline{\nu}^{(N)}}
\def\fmeas{\overline{\nu}}
\def\fmeass{\overline{\nu}_*}
\def\flam{\overline{\lambda}}
\def\fdn{\overline{D}^{(N)}}

\def\resaren{R_{E}^{(N)}}

\def\fidlen{\overline{I}^{(N)}}
\def\xn{X^{(N)}}

\def\ttimen{\tau^{(N)}}
\def\fxn{\overline{X}^{(N)}}
\def\fx{\overline{X}}
\def\fidle{\overline{I}}
\def\fnewmeasn{\tilde{\nu}^{(N)}}
\def\newfk{\overline{Z}}

\def\fbare{\overline{Q}}
\def\fdcomp{\overline{A}}

\def\small{\delta}

\def\fxp{\overline{X}^{(2)}}
\def\fkp{\overline{K}^{(2)}}

\def\fmeasp{\fmeas^{(2)}}

\def\fxo{\overline{X}^{(1)}}
\def\fko{\overline{K}^{(1)}}

\def\fmeaso{\fmeas^{(1)}}

\def\fd{\overline{D}}
\def\fk{\overline{K}}

\def\fsmallk{\overline{\kappa}}

\def\newspace{{\cal S}_0}

\def\lloc{{\cal L}^1_{loc}}
\def\newfmeas{\overline{\mu}}

\def\are{q}
\def\bare{Q}
\def\baren{Q^{(N)}}
\def\fbaren{\overline{Q}^{(N)}}
\def\newdn{D_*^{(N)}}
\def\remn{R^{(N), n}}

\def\ftime{\theta}
\def\stime{\zeta}

\def\fmartn{\overline{M}^{(N)}}
\def\martn{M^{(N)}}
\def\dcompn{A^{(N)}}
\def\fdcompn{\overline{A}^{(N)}}

\def\F{\mathbb{F}}

\def\ttime{\tau}

\def\suli{\sum\limits}

\title[Fluid Limits of Many-Server Queues]{Law of Large Numbers Limits for Many-Server Queues}

\author{Haya Kaspi}
\address{Department of Industrial Engineering and Management \\
Technion, Haifa, Israel} \email{iehaya@techunix.technion.ac.il}

\author{Kavita Ramanan }
\address{Department of Mathematical Sciences \\
Carnegie Mellon University \\
Pittsburgh, PA 15213, USA }
\email{kramanan@math.cmu.edu }
\thanks{Partially supported
by the National Science Foundation under Grants DMS-0406191,
DMI-0323668-0000000965, DMS-0405343 and the Technion Fund for the
Promotion of Research  under Grant 2003941}


\begin{document}

\date{July 30, 2007}

\begin{abstract}
This work considers a many-server queueing system in which
customers with i.i.d., generally distributed service times enter service in the
order of arrival.  
The dynamics of the system is represented in terms of
  a  process that describes
the total number of customers in the system, as well as
a measure-valued process that keeps track of the ages of
customers in service.
Under mild assumptions on the service time distribution, as the
number of servers goes to infinity,
 a  law of large numbers (or fluid) limit is established for this pair of processes.
The  limit is characterised as the unique solution to
a coupled pair of integral equations,
which admits a fairly explicit representation.
As a corollary, the fluid limits of several other functionals of interest, such as the
waiting time, are also obtained.
Furthermore, in the time-homogeneous setting,
 the fluid limit is shown to converge to its equilibrium.
Along the way, some results of independent interest are obtained,
including a continuous mapping result and a maximality property
of the fluid limit.
A motivation for studying these systems is that they arise as models of
computer data systems and call centers.
\end{abstract}

\date{July 30, 2007}

\maketitle

\bigskip
\hrule
\tableofcontents
\hrule

\section{Introduction}
\label{sec-intro}

\subsection{Background and Motivation.}
\label{subs-back}
The main objective of this paper is to obtain functional strong laws of large numbers
 or ``fluid'' approximations to various functionals of the G/GI/N queue,
in the limit as the number of servers tends to infinity.
In fact,  a more general setting is considered that allows
for possibly time-inhomogeneous arrivals.
In order to obtain a Markovian description of the dynamics,  the state
includes a non-negative integer-valued process that represents the  total number of customers in system, as well as  a measure-valued
process that keeps track of the ages of customers in service.
The fluid limit obtained is for this pair of processes and thus
contains more information than just
the limit of the scaled number of customers in system.
In particular, it  also yields a description of the
fluid limits of several other functionals of interest, including the waiting time.
A fairly explicit representation for the fluid limit is obtained, which is then
used to study
 the convergence, as $t \ra \infty$, of the fluid limit to equilibrium
in the time-homogeneous setting. These results are obtained under
mild assumptions on the service distribution, such as the
existence of a density, which are satisfied by most distributions
that arise in applications. While we expect that these conditions
can be relaxed,
the representation of the fluid limit is likely to be more involved in that setting.
Thus,
for ease of exposition, we have restricted ourselves to this generality.

Multiserver queueing systems arise in many applications, and have generally proved to be
more difficult to analyse than single server queues.
Thus it is natural to resort to an  asymptotic analysis in order to gain insight into the behavior
of these systems.
In this work we consider the asymptotic regime introduced by Halfin and Whitt in \cite{halwhi},
in which the number of servers and, proportionally, the arrival rate of customers, tend to infinity.
For the case of Poisson arrivals and exponentially distributed service times,
a central limit theorem for the number of customers in system was obtained in \cite{halwhi}.
For networks of multi-server queues with (possibly time-varying) Poisson arrivals and
exponential services, fluid and diffusion limits for the total number of customers in system  were obtained by
Mandelbaum, Massey and Reiman in \cite{manmasrei}, and these results were later extended to
include the queue length and virtual waiting time processes in \cite{manmasreisto}.
All of these results  were obtained under the assumption of exponential service times.

This work is to a large extent motivated by the fact that
G/GI/N queues arise as models of large-scale
telephone call centers, for which the limiting regime considered here admits the natural interpretation of the
scaling up of the number of servers in response to an analogous scaling up of the arrival rate of
customers (see \cite{brownetal} for a survey of the applications of multi-server models to call centers).
Recent statistical evidence from real call centers presented
in Brown et.\ al.\ \cite{brownetal}
suggests that, in many cases, it may be more appropriate
to model the service times as being   non-exponentially
distributed and, in particular,  according to the lognormal distribution.
This emphasises the need to obtain fluid limits when
the service times are   generally distributed.
With this as motivation,  a deterministic fluid approximation for a G/GI/N queue, with
abandonments  that are possibly  generally distributed,
was proposed by Whitt in \cite{whifluid06}.
However, a general functional strong law of large numbers justifying the fluid approximation
was not obtained in \cite{whifluid06}
(instead, convergence was established for
a discrete-time version of the model, allowing for time-dependent and state-dependent
arrivals).
In this paper, we establish a functional strong law of large numbers limit
 allowing for time-dependent arrivals, but in the absence of abandonments,
and provide an intuitive  characterisation of the limit.
Concurrently with this work, fluid and central limit theorems
for just the number in system were established in the  work of
Reed \cite{reed07} using a clever comparison with a
 $G/GI/\infty$ system.
Here, we take a different approach, involving a measure-valued representation
that leads to a different representation of the fluid limit of the number
in system and also yields the fluid limits of several other functionals
of interest.

One of the challenges in going from exponential to non-exponential
service distributions is that a Markovian description of the
dynamics leads to an infinite-dimensional state. The
measure-valued representation and martingale methods adopted in
this paper provide a convenient framework for the asymptotic
analysis of multi-server queues (see \cite{whi-martsurvey07} for a
recent survey on the use of martingale methods for establishing
heavy-traffic limits of multi-server queues with exponentially
distributed service times). Indeed, we  believe that the framework
developed here can be extended in many ways, for example,  to
include abandonments. The characterisation of the pre-limit
obtained here is also used to
 establish central limit theorems in \cite{kasram2}.
In the context of single-server queueing networks,   recent works that have
used measure-valued processes to study fluid limits
include  \cite{decmoy06}, \cite{dowpuhgro06}, \cite{gropuhwil02}, \cite{grorobzwa07} and
 \cite{growil07}.

The outline of the paper is as follows.
A precise mathematical description of the model,
including the basic  assumptions, is provided in Section \ref{sec-mode}.
 Section \ref{sec-flconj} introduces the fluid equations
and contains a summary of the main results.
  Uniqueness of solutions
to the fluid equations  is established in
 Section \ref{sec-flsol} and  the functional strong law of large numbers limit is proved in
Section \ref{sec-tight}.
Finally,  the large-time or equilibrium behavior
of the fluid limit is described  in Section \ref{sec-longtime}.
In the remainder of this section, we introduce some common notation
used in the paper.

\subsection{Notation and Terminology}
\label{subs-not}

The following notation will be used throughout the paper.
 $\Z_+$ is the set of non-negative integers, $\R$ is set of real numbers and
$\R_+$ the set of non-negative real numbers.
For $a, b \in \R$, $a \vee b$  denotes  the maximum of $a$ and $b$,
$a \wedge b$  the minimum of $a$ and $b$ and the short-hand $a^+$ is used for $a \vee 0$.
Given $A \subset \R$ and $a \in \R$,  $A - a$ equals the set
$\{x \in \R: x + a \in A\}$ and  $\ind_B$ denotes the indicator function of the set $B$
(that is, $\ind_B (x) = 1$ if $x \in B$ and $\ind_B(x) = 0$ otherwise).

\subsubsection{Function and Measure Spaces}
\label{subsub-funmeas}

Given any metric space $E$, $\ecb$ and $\ecc$  are
 the space of bounded,  continuous functions and, respectively,
 the space of continuous real-valued functions with compact support defined on
$E$, while
 $\eocd$ is the space of real-valued,
once  continuously differentiable functions on $E$, and  $\eocdc$
is the subspace of functions in $\eocd$ that have compact support.
The subspace of functions in $\eocd$ that, together with their
first derivatives, are bounded, will be denoted by $\eocdb$. For
$M\le\infty$, let
 ${\cal L}^1([0,M))$ and $\lloc([0,M))$ represent, respectively, the spaces of
integrable and locally integrable functions on $[0,M)$, where for
$M<\infty$ a locally integrable function $f$ on $[0,M)$ satisfies
$\int_{[0,a]}f(x)dx<\infty$ for all $a<M$ . The constant functions
$f \equiv 1$ and $f \equiv 0$ will be represented by the symbols
$\f1$ and $\zerof$, respectively. Given any c\`{a}dl\`{a}g,
real-valued function $\newf$ defined on $E$, we define $\nrm{\newf
}_T \doteq \sup_{s \in [0,T]} |\newf (s)|$ for every $T < \infty$,
and let $\nrm{\newf}_\infty \doteq \sup_{s \in [0,\infty)}
|\newf(s)|$, which could possibly take the value $\infty$.
In addition,
the support of a function $\newf$ is denoted by $\supp(\newf)$.

The space of  Radon measures on a metric space $E$, endowed with the Borel $\sigma$-algebra,
 is denoted by  $\eem$, while
  $\eemf$, $\eemone$ and $\eemsp$ are respectively, the subspaces of finite,
 probability and sub-probability measures in $\eem$.  Also, given $B < \infty$,
$\eemb \subset \eemf$ denotes the space of measures $\mu$ in $\eemf$ such  that $|\mu(E)| \leq B$.
Recall that a Radon measure is one that assigns finite measure to every relatively
compact subset of $\R_+$.
The space $\eem$ is equipped with the vague topology,  i.e.,  a sequence of measures
$\{\mu_n\}$ in $\eem$ is said to converge to $\mu$ in the vague topology (denoted
$\mu_n \vconv \mu$) if and only if for every $\newf \in \ecc$,
\be
\label{w-limit}
\int_{E} \newf(x) \,  \mu_n(dx)  \ra \int_E \newf(x) \, \mu (dx)
 \quad \quad \mbox{ as } n \ra \infty.
\ee
By identifying a Radon measure $\mu \in \eem$ with the mapping on $\ecc$ defined by
\[ \newf \mapsto \int_{E} \newf(x) \, \mu(dx), \]
one can equivalently define a Radon measure on $E$ as a linear mapping
from $\ecc$ into $R$ such that for every compact ${\cal K} \subset E$, there exists
$L_{{\cal K}} < \infty$ such that
\[ \left| \int_{E} \newf(x) \, \mu(dx) \right| \leq L_{{\cal K}} \nrm{\newf}_\infty \quad \quad \forall
\newf \in \ecc  \mbox{ with } \supp (\newf ) \subset {\cal K}. \]
 On  $\eemf$, we will also consider the
weak topology, i.e., a sequence $\{\mu_n\}$ in $\eemf$ is said to 
converge weakly to $\mu$  (denoted $\mu_n \wconv \mu$) if and only
if (\ref{w-limit}) holds for every $\newf \in \ecb$. As is
well-known, $\eem$ and $\eemf$, endowed with the vague and weak
topologies, respectively,
 are Polish spaces.
The symbol $\delta_x$ will be used to denote the measure with unit mass at the point $x$ and,
by some abuse of notation, we will use $\zerof$ to denote the identically zero
Radon measure on $E$.  When $E$ is an interval, say $[0,M)$, for notational conciseness,
we will often write ${\cal M}[0,M)$ instead of ${\cal M}([0,M))$.

We will mostly be interested in the case when
 $E = [0,M)$ and  $E = [0,M) \times \R_+$, for some $M \in (0,\infty]$.
To distinguish these cases, we will usually use $f$ to denote generic  functions
on $[0,M)$  and $\newf$ to denote generic
 functions on $[0,M) \times \R_+$.  By some abuse of notation,
given $f$ on $[0,M)$, we will sometimes
also treat it as a function on $[0,M) \times \R_+$ that is constant in the second
variable.  For any
Borel measurable function $f: [0,M) \ra \R$ that is integrable
with respect to $\xi \in  \mm$, we  often use the short-hand notation
\[ \lan f, \xi \ran \doteq \int_{[0,M)} f(x) \, \xi(dx).
\] 
Also, for ease of notation, given $\xi \in \mm$ and an interval $(a,b) \subset [0,M)$, we will
use $\xi(a,b)$ and $\xi(a)$ to denote $\xi((a,b))$ and $\xi(\{a\})$, respectively.

\subsubsection{Measure-valued Stochastic Processes}

Given a Polish space ${\cal H}$, we denote by $\dTspaceh$
(respectively, $\dspaceh$) the space of ${\cal H}$-valued,
c\`{a}dl\`{a}g functions on $[0,T]$ (respectively, $[0,\infty)$),
and we endow this space with the usual Skorokhod $J_1$-topology
\cite{parbook}. Then $\dTspaceh$ and $\dspaceh$ are also Polish
spaces (see \cite{parbook}). In this work, we will be interested
in ${\cal H}$-valued stochastic processes, where ${\cal H} = \mmf$
for some $M\le\infty$ .
These are random elements that are defined on a probability space $(\Omega, {\cal F}, \P)$ and  take values in  $\dspaceh$,
equipped with the Borel $\sigma$-algebra (generated by open sets under the
Skorokhod $J_1$-topology).
A sequence $\{X_n\}$ of c\`{a}dl\`{a}g, ${\cal H}$-valued  processes, with $X_n$ defined on the probability space $(\Omega_n, {\cal F}_n, \P_n)$,
 is said to converge in distribution
to a c\`{a}dl\`{a}g ${\cal H}$-valued process $X$ defined on  $(\Omega, {\cal F}, \P)$ if, for every bounded, continuous functional
$F:\dspaceh \ra \R$, we have
\[ \lim_{n \ra \infty}  \E_n\left[ F(X_n) \right] = \E \left[ F(X)\right],
\]
where $\E_n$ and $\E$ are the expectation operators with respect to the probability measures
$\P_n$ and $\P$, respectively.
Convergence in distribution of $X_n$ to $X$ will be denoted by $X_n \swconv X$.

\beginsec

\section{Model Dynamics and  Basic Assumptions}
\label{sec-mode}

In Section \ref{subs-modyn} we describe our basic model and state our main assumptions,
and in Section \ref{subs-aux} we introduce some auxiliary processes that
are useful for the study of the dynamics of the model.

\subsection{Description of the Model}
 \label{subs-modyn}

Consider a system with $N$ servers, where arriving
customers are served in a non-idling, First-Come-First-Serve (FCFS) manner, i.e., a newly
arriving customer immediately enters service if there are any idle servers, or,
if all servers are busy, then the customer joins the back of
the queue, and the customer at the head of the queue (if one is present) enters
 service as soon as a server becomes free.
Our results are not sensitive to the exact mechanism used to assign an
arriving customer to an idle server, as long as the non-idling condition
is satisfied (see Remark \ref{rem-assign}). 
 Let $\en$ denote the
cumulative arrival process, with $\en (t)$ representing the total number
of customers that arrive into the system in the time
interval $[0,t]$, and let the service requirements be given by
the i.i.d.\ sequence $\{v_i, i = -N +1, -N+2, \ldots, 0, 1, \ldots\}$, with
common cumulative distribution function $G$.
For $i \in \N$, $v_i$ represents the service requirement of the
$i$th customer to enter service after time $0$, while
$\{v_{i}, i = -N+1, \ldots, 0\}$ represents the service requirements
of customers already in service at time $0$ (if such customers are present),
 ordered according to their arrival times (prior to zero).
When $\en$ is a renewal process, this is simply a  GI/G/N queueuing system.

Consider the process $R_E^{(N)}$ defined, for
$s>0$, by
\be
\label{def-ren}
 R_E^{(N)}(s) \doteq \inf\{t>0:E^{(N)}(s+t)>E^{(N)}(s)\}.
\ee
Note that by our assumption on the finiteness of $E^{(N)}(t)$ for
$t$ finite, $R_E^{(N)}(s)>0$ for $s  > 0$.
The following mild assumptions will be imposed throughout, without explicit mention.
\begin{itemize}
\item
The cumulative arrival process $\en$ is independent of the sequence of service requirements
$\{v_j, j = -N+1, \ldots, \}$;
\item
The process $R_E^{(N)}$ is Markovian with respect to its own natural
filtration (this holds, for example, when $E^{(N)}$ is a renewal process or an
inhomogeneous Poisson process);
 \item
 $G$ has density $g$.
\item
Without loss of generality, we can (and will) assume that the mean service requirement
is $1$:
\be
\label{def-mean1}
\int_{[0,\infty)} \left( 1 - G(x) \right) \, dx = \int_{[0,\infty)} x g(x) \, dx = 1.
\ee
\end{itemize}

The first two assumptions above are very general,  allowing for a large class of
arrival processes.  Note that the third assumption implies, in particular, that
$G(0+) = 0$.  The existence of a density is assumed for convenience, and
is satisfied by a large class of distributions of interest in applications.
The relaxation of this assumption would lead to a more complicated and somewhat less
intuitive representation for the fluid limit and thus, for ease of exposition, we have restricted
ourselves to this generality.
Define the hazard rate
\be
\label{def-h}
h(x) \doteq \dfrac{g(x)}{1 - G(x)}  \quad \quad \mbox{ for } x \in [0,\endsup),
\ee
where $\endsup \doteq \sup \{x \in [0,\infty): G(x) < 1 \}$.
Observe that $h$ is automatically locally integrable on $[0,M)$ since for every
$0 \leq a \leq b < \endsup$,
\[ \int_a^b h(x) \, dx = \ln (1 - G(a)) -\ln (1 - G(b)) < \infty. \]
When additional assumptions on $h$ are needed, they will be mentioned
explicitly in the statements of the results.

The $N$-server model described above can be represented in many ways
(see, for example, representations for GI/G/N queueing systems in
\cite[Chapter XII]{asmbook}).
For our purposes, we will find it convenient to encode the state of the system
in the processes $(\ren,\xn, \measn)$, where
$\ren$ is the residual inter-arrival time process defined in (\ref{def-ren}),
$\xn(t)$ represents the
 number of customers in the system at time $t$ (including those
in service and those in the queue, waiting to enter service) and
$\nun (t)$ is the discrete non-negative Borel measure on $[0,M)$
that has a unit mass at the age of each of the customers in
service at time $t$. Here, the age $a_j^{(N)}$  of  customer $j$
is (for every realisation) the piecewise linear function on
$[0,\infty)$ that is defined to be $0$ till the customer enters
service, then increases linearly while the customer is in service
(representing the amount of time elapsed since  entering service)
and is then constant (equal to the total service requirement)
after the customer departs. Hence, the total number of customers
in service at time $t$ is given by $\lan \f1, \measn_t \ran =
\measn_t ([0,M))$, which is bounded above by $N$ and so $\measn_t
\in {\cal M}_F ([0,M))$ for every $t \in [0,\infty)$. Our results
will be independent of the  particular rule used to assign
customers to stations, but for technical purposes we will  find it
convenient to also introduce the additional ``station process''
$s^{(N)} \doteq (s_j^{(N)}, j \in  \{-N + 1,  \ldots, 0\} \cup
\N)$. For each $t \in [0,\infty)$, if customer $j$ has already
entered service by time $t$,  then
 $s_j^{(N)} (t)$ is equal to the index $i \in \{1, \ldots, N\}$ of the station  at which
customer $j$ receives/received service and  $s_j^{(N)} (t) \doteq
0$ otherwise.
 Finally,
for $t \in [0,\infty)$, let $\tilde{{\cal F}}_t^{(N)}$ be the
$\sigma$-algbera generated by $(\ren(s),  a_j^{(N)}(s), s_j^{(N)}(s), 
j \in \{-N+1, \ldots, 0\}\cup \N, s \in [0,t]\}$ and let
$\{{\cal F}_t^{(N)}\}$ denote the associated completed,
right-continuous filtration. As elaborated in the next section, it
is not hard to see that $(\en,\xn,\measn)$ is ${\cal
F}_t^{(N)}$-adapted (also see Appendix \ref{ap-markov} for an
explicit construction of these processes that shows, in
particular, that they are well-defined). In fact, although we do
not prove or use this property in this paper, it is not hard to
show that $\{(\ren (t), \xn (t), \measn_t), \{{\cal
F}_t^{(N)}\}\}$ is a Markov process.

\subsection{Some auxiliary processes}
\label{subs-aux}

We now introduce the following auxiliary processes that will be useful for the study of the evolution
of the system:
\begin{itemize}
\item
the cumulative departure process $\dn$, where $\dn(t)$ is the cumulative
number of customers that have departed the system in the interval $[0,t]$;
\item
the process $\kn$, where $\kn(t)$ represents the cumulative number of customers that
have entered service in the interval $[0,t]$;
\item
 the age process
$(\agen_j, j = -\measn_0 (\R_+) + 1, \ldots, 0, 1, \ldots )$ where $\agen_j (t)$ is the time that has
elapsed since the $j$th customer entered service, with the convention that $\agen_j(t)$ is set to zero if customer
$j$  has not yet entered service by time $t$ and remains constant once the customer departs;
\item
the idle servers process $\idlen$, where $\idlen(t)$ is the number of idle servers at time $t$.
\end{itemize}

\noi
Simple mass balances show that
\be
\label{def-dn}
\dn  \doteq \xn (0) - \xn + \en
\ee
and
\be
\label{def-kn}
\kn \doteq \lan \f1, \measn \ran - \lan \f1, \measn_0 \ran + \dn.
\ee
Due to the FCFS nature of the service,
observe that $\kn(t)$ is also the highest index of any customer that has entered service.
The age process of  all customers in service increases linearly and so, given the  service requirements of
customers,  the evolution of the age process can be described explicitly as follows: for  $j \leq  \kn (t)$,
\be
\label{def-agejn}
\agen_j (t) =  \left\{
\ba{rl}
\agen_j(s) + t - s  & \mbox{ if }  \agen_j (t) < v_j, \\
v_j & \mbox{ otherwise, }
\ea
\right.
\ee
and
 $\agen_j(t) \doteq 0$ for all $j > \kn (t)$.
The measure $\measn$ can be written in the form \be
\label{def-nun} \measn_t  = \sum\limits_{j = -\lan \f1, \measn_0
\ran   + 1}^{\kn(t)} \delta_{\agen_j(t)} \ind_{\{\agen_j (t) <
v_j\}}. \ee Here $\delta_x$ represents the Dirac mass at the point
$x$. Now, at any time $t$,  the age process of any customer has a
right-derivative that is positive (and equal to one) if and only
if the customer is in service, and has a left-derivative that is
positive and a right-derivative that is zero if and only if it has
just departed. Thus $\dn$ is clearly $\{{\cal
F}_t^{(N)}\}$-adapted and, since $\measn$ can be written
explicitly in terms of the age process as \be \label{def-newnun}
\measn_t \doteq \measn (t)  \doteq \sum\limits_{j = -\lan \f1,
\measn_0 \ran   + 1}^{\kn(t)} \delta_{\agen_j(t)}
\ind_{\left\{\frac{d\agen_j}{dt} (t+) > 0\right\}}, \ee
 $\measn$ is also $\{{\cal F}^{(N)}_t\}$-adapted.
Furthermore,
 since $\lan \f1, \measn_t \ran$ represents the number of customers in service at time $t$, the idle servers process is given
explicitly by the relation
\be
\label{def-idlen}
 \idlen \doteq N - \lan \f1, \measn  \ran.
\ee
The non-idling condition then takes the form
\be
\label{def-nonidling}
\idlen = [N - \xn]^+.
\ee
Indeed, note that this shows that when $\xn (t) < N$,  so that $\idlen (t) > 0$,
 the number in system is equal to the
number in service, and so there is no queue, while if $\xn (t) > N$, so that there is a positive queue,
$\idlen (t) = 0$, indicating that there are no idle servers.
>From the above discussion, it follows immediately that the processes $\xn, \measn, \dn, \kn$ and $\idlen$ are all
$\{{\cal F}_t^{(N)}\}$-adapted.
For an explicit construction of these processes, see
Appendix \ref{ap-markov} and, in particular, Remark \ref{rem-aux}.

\begin{remark}
\label{rem-nuinitm}
{\em If $M < \infty$, then for every $N$, we will always assume that
$\measn_0$ has support in $[0,M)$.   From
(\ref{def-nun}), this automatically implies that
$\measn_t$ also has support in $[0,M)$ for every $t \in [0,\infty)$ and,
moreover, that $\measn \in {\cal D}_{\mmsp}[0,\infty)$. }
\end{remark}

\beginsec

\section{Main Results}
\label{sec-flconj}

We now summarise our main results.
First, in Section \ref{subs-fleqs}, we introduce the so-called
fluid equations, which provide a continuous analog of the discrete model
introduced in Section \ref{sec-mode}.
In Section \ref{subs-mainres} we present our main results, which in
particular show that the fluid equations uniquely characterise the
strong law of large numbers limit of the multi-server system, as the number of
servers goes to infinity.
Lastly, in Section \ref{subs-funcflimit}, we show how our results
 can be used to obtain fluid limits of various other functionals of interest.
This, in particular, illustrates the usefulness of adopting a
measure-valued representation for the state.

\subsection{Fluid  Equations}
\label{subs-fleqs}

Consider the following scaled versions of the basic processes
describing the model. For $N \in \N$, the scaled state descriptor
$(\fren,\fxn,\fmeasn)$ is given by
\be \label{fl-scaling} 
\fren (t) \doteq \ren (t)   \quad  \fxn (t)
\doteq \dfrac{\xn(t)}{N} \quad \fmeasn (t) (B) \doteq
\dfrac{\measn (t) (B)}{N} \ee
 for $t \in [0,\infty)$ and any  Borel subset $B$ of $\R_+$, and observe that
$\fmeasn(t)$ is a sub-probability measure on $[0,\infty)$ for every
$t \in [0,\infty)$.
Analogously, define
\be
\label{fl-scaling2}
\fen  \doteq \dfrac{\en}{N}; \quad\fdn \doteq \dfrac{\dn}{N};  \quad \fkn  \doteq \dfrac{\kn}{N}; \quad \fidlen \doteq \dfrac{\idlen}{N}.
\ee

Let $\incspace$ to be the subset of non-decreasing functions
$f \in {\cal D}_{\R_+}[0,\infty)$ with $f(0) = 0$.  Recall that
$M=\sup \{x \in [0,\infty): G(x) < 1\}$, and define
\be
\label{def-newspace}
\newspace \doteq \left\{(f,x,\mu) \in \incspace \times \R_+ \times \mmsp:
1 - \lan \f1, \mu \ran = [1-x]^+ \right\}.
\ee
$\newspace$ serves as the space of possible input data for the fluid equations.
We make the following convergence assumptions on the primitives of the scaled
sequence of systems.


\begin{ass} {\bf (Initial conditions)}
\label{ass-init}
There exists an $\newspace$-valued random variable $(\fe, \fx (0), \fmeas_0)$
 such that,
as $N \ra \infty$, the following limits hold almost surely:
\begin{enumerate}
\item
$\fen \ra \fe$  \quad in ${\cal D}_{\R_+}[0,\infty)$,
where $\E [\fe(t)]<\infty$ for all $t \in (0, \infty)$;
\item
$\fxn(0) \ra \fx (0)$ \quad in $\R_+$, where $\E[X(0)] < \infty$
\item
 $\fmeasn_0 \ra \fmeas_0$ \quad in $\mmsp$.
\end{enumerate}
\end{ass}

\begin{remark}
\label{rem-assinit}
{\em By the (generalized)  dominated convergence theorem,  conditions (1) and (2) of Assumption \ref{ass-init}
imply that the convergence of $\fen$ to $\fe$ and of $\fxn (0)$ to $\fx (0)$ also holds in
expectation, so that in particular, we have for every $t \in [0,\infty)$, for
all sufficiently large $N$,
$\E[ \fxn (0) + \fen (t) ] < \infty$.
}
\end{remark}

\begin{remark}
\label{rem-skorep}
{\em
If the limits in Assumption \ref{ass-init} hold only in distribution rather than
almost surely, then using the Skorokhod representation theorem in the standard way,
it can be shown that all the stochastic process convergence results in the paper continue to hold. }
\end{remark}

Our  goal is  to identify
 the limit in distribution of  the quantities
$(\fxn, \fmeasn)$, as $N \ra \infty$.
In this section, we first introduce the so-called fluid
equations and provide some intuition as to why the
 limit of any sequence $(\fxn,\fmeasn)$ should be expected
to be a solution to these equations.
In Section \ref{sec-tight}, we provide a rigorous proof of this fact.

\begin{defn} {\bf (Fluid Equations)}
\label{def-fleqns} {\em The c\`{a}dl\`{a}g function $(\fx,
\fmeas)$ defined on $[0,\infty)$ and taking values in $\R_+ \times
\mmsp$ is said to solve the}  fluid equations {\em associated with
$(\fe, \fx (0), \fmeas_0) \in \newspace$
 and  $h \in \lloc [0,M)$ if and only if 
for every} $t \in [0,\infty)$, 
\be
\label{cond-radon}
 \int_0^t  \lan h, \fmeas_s \ran \, ds  < \infty,
\ee
{\em and  the following relations are satisfied: for every $\newf \in \newocdcpm$, }
\begin{eqnarray}
\label{eq-ftmeas}
\ds \lan \newft, \fmeas_t \ran  & = & \ds \lan  \newf(\cdot,0), \fmeas_0 \ran +
\int_0^t \lan \dtnewfs, \fmeas_s \ran \, ds +
\int_0^t \lan \dxnewfs, \fmeas_s  \ran \, ds  \\
 & & \nonumber
\quad  - \ds \int_0^t \lan  h(\cdot) \newfs,  \fmeas_s \ran \, ds+ \int_0^t \newf (0,s) \, d\fk (s),  \\
\label{eq-fx}
\fx (t) &  = & \fx (0) + \fe(t) - \ds \int_0^t \lan  h, \fmeas_s \ran \, ds
\end{eqnarray}
{\em and}
\be
\label{eq-fnonidling}
 1 - \lan \f1, \fmeas_t \ran = [1 - \fx(t)]^+,
\ee
{\em where }
\be
\label{eq-fk}
\fk(t) = \lan \f1, \fmeas_t \ran - \lan \f1, \fmeas_0 \ran + \int_0^t \lan h, \fmeas_s \ran \, ds.
\ee
\end{defn}

We now  provide an informal, intuitive explanation for the form of the fluid equations.
Suppose $(\fx, \fmeas)$ solves the fluid equations associated with some
$(\fe, \fx, \fmeas_0) \in \newspace$.
Then, roughly speaking, for $x \in \R_+$,
 $\fmeas_s (dx)$ represents the amount of mass
(or limiting fraction of customers) whose age lies in the range
$[x,x+dx)$ at time $s$.
Since  $h$ is the hazard rate, $h(x)$
represents the fraction of mass with age $x$ that would depart
from the system at any time.
Thus   $\lan h, \fmeas_s \ran$
represents the departure rate of mass from the fluid system at
time $s$, and
 the process $\fd$ that satisfies
\be
\label{def-fd}
 \fd (t) \doteq  \int_0^t \lan  h, \fmeas_s \ran \, ds \quad \quad \mbox{ for } t \in [0,\infty)
\ee
represents the cumulative amount of departures from the fluid system.
Since $\fe$  is the limiting cumulative arrival rate of mass
into the fluid system in the interval $[0,t]$, a simple mass balance
yields the relation (\ref{eq-fx}) (in analogy with the
pre-limit equation (\ref{def-dn})).
Likewise, (\ref{eq-fnonidling}) and (\ref{eq-fk}) are the fluid  versions of the non-idling condition
(\ref{def-nonidling}) and the mass balance relation (\ref{def-kn}), respectively.

Next, note that the fluid equation (\ref{eq-ftmeas}) implies, in particular,
that for  $f \in \ocdcm$,
\be
\label{eq-fmeas}
\ds \lan f, \fmeas_t \ran  =  \ds \lan  f, \fmeas_0 \ran +
\int_0^t \lan f^\prime, \fmeas_s  \ran \, ds - \ds \int_0^t \lan  f h,  \fmeas_s \ran \, ds + \ds f(0) \fk (t).
\ee
 The difference
$\lan f, \fmeas_t \ran - \lan f, \fmeas_0\ran$ is caused by three different phenomena -- evolution of
the mass in the system, departures and arrivals -- which are
represented by the last three terms on the right-hand side of   (\ref{eq-fmeas}).
Specifically, the  second term on the right-hand side represents the change in $\fmeas$ on account of the
fact that the ages of all customers in service increase at a constant rate $1$,
the third term represents the change due to departures of customers in service,
and  the last term on the right-hand side of (\ref{eq-fmeas}) accounts
for new customers entering service.
Here $\fk(t)$ represents the cumulative amount of mass that
has entered service in the fluid system, and is multiplied by
$f(0)$  because any arriving customer by definition has age $0$ at the time of
entry.

\subsection{Summary of Main Results}
\label{subs-mainres}

Our first result concerns uniqueness of solutions
to the fluid equations, which is established at the end of
Section \ref{subs-cont}.

\begin{theorem}
\label{th-fluniq}
Given any $(\fe, \fx(0), \fmeas_0) \in \newspace$ and $h \in {\cal C}[0,M)$,
 there exists at most one solution $(\fx,\fmeas)$ to the associated fluid equations
(\ref{cond-radon})--(\ref{eq-fnonidling}).
Moreover,
the solution $(\fx,\fmeas)$  satisfies, for every $f \in \cb$,
\be
\label{final-fleqs}
\ds \int_{[0,M)} f(x) \,  \fmeas_t (dx)    =  \ds \int_{[0,M)} f(x+t) \dfrac{1 - G(x+t)}{1 - G(x)} \fmeas_0 (dx)
+ \int_0^t f(t-s) (1 - G(t-s)) d \fk (s),
\ee
where $\fk$ is given by the relation (\ref{eq-fk}).
 Moreover, if $\fe$ is absolutely continuous with $\flam \doteq  d\fe/dt$,
 then
$\fk$ is also absolutely continuous and $\fsmallk \doteq d\fk/dt$ satisfies for a.e.\ $t \in [0,\infty)$,
\be
\label{def-fsmallk}
 \fsmallk (t) \doteq \left\{
\ba{rl}
\flam (t)  & \mbox{ if } \fx(t) < 1 \\
\flam (t) \wedge \lan h, \fmeas_t \ran  & \mbox{ if } \fx (t) = 1 \\
\lan h, \fmeas_t \ran & \mbox{ if } \fx (t) > 1.
\ea
\right.
\ee
Furthermore, if $\fmeas_0$ and $\fe$ are  absolutely continuous, then
$\fmeas_s$ is  absolutely continuous for every $s \in [0,\infty)$.
\end{theorem}

It is also possible to consider the case when the residual service
times of the customers already in system is distributed according to another
distribution $\tilde{G}$.  Indeed, our proofs  show that in this case
the relations (\ref{final-fleqs}) and (\ref{def-fsmallk}) continue to hold, with
 $G$ in the first integral on the right-hand side of (\ref{final-fleqs}) replaced
by $\tilde{G}$.
The uniqueness result can, in fact, be shown to
hold for any $h \in \lloc [0,M)$  for which
(\ref{cond-radon}) holds (see Remark \ref{rem-gencont}), but
we defer the proof of this to a subsequent paper.

Our next main result shows that  a
solution to the fluid equations exists and provides an explicit
characterisation of the functional law of large numbers limit of the $N$-server system, as
$N \ra \infty$.

\begin{theorem}
\label{th-conv} Suppose Assumption \ref{ass-init} is satisfied for some
$(\fe, \fx (0), \fmeas_0) \in
\newspace$.
 Then
the sequence $\{(\fxn, \fmeasn)\}$ is relatively compact.  Moreover,
if $h$ is continuous on $[0,M)$ then as $N \ra \infty$,  $(\fxn, \fmeasn)$ converges in
distribution to the unique solution $(\fx, \fmeas)$ of the associated fluid
equations.
On the other hand, if
  $\fe$ and  $\fmeas_0$ are absolutely continuous, and the set of discontinuities
of $h$ in $[0,M)$ is closed and has zero Lebesgue measure,
then every limit point $(\fx, \fmeas)$ of $\{(\fxn, \fmeasn)\}$ satisfies
the fluid equations.
\end{theorem}

\noi
The proof of the theorem is given at the end of Section \ref{subs-pf3}.
The key steps in the proof involve showing tightness of the
sequence $\{(\fxn, \fmeasn)\}$, which is
carried out in Section \ref{subs-tight}, characterising the limit points
of the sequence, which is done in Section \ref{subs-pf3},
and establishing uniqueness of solutions to the fluid equations, which
is the subject of Section \ref{sec-flsol}.

The last  main result
describes the long-time behaviour of the unique solution $(\fx, \fmeas)$ to the
fluid equations in the time-homogeneous setting.
Define $\fmeass$ to be the measure on $[0, M)$ that is absolutely continuous
with respect to Lebesgue measure, and has density $1 - G(x)$: for any Borel set $A \subset [0, M)$,
\be
\label{def-meass}
 \fmeass (A) = \int_{A} (1  - G(x)) \, dx.
\ee

\begin{theorem}
\label{th-larget}
Given  $h \in \lloc [0,M)$, $\flam \in
(0,\infty)$ and $(\flam \f1,\fx(0), \fmeas_0) \in
\newspace$, suppose there exists a unique solution
$(\fx, \fmeas)$  to the associated fluid equations.
Then the following two properties are satisfied.
\begin{enumerate}
\item
If $\fx (0) = 0$ then as $t \ra \infty$,
$\fx(t) = \lan \f1, \fmeas_t \ran$ converges monotonically up to $1$ and, if in addition $\flam \leq 1$, then
for every  non-negative function $f \in \cbm$, $\lan f, \fmeas_t\ran$ converges monotonically up to
$\flam \lan f, \fmeass \ran$.
\item
Suppose the density $g$  of the service distribution is directly Riemann integrable and
$\flam \geq 1$.
Then, given any $(\flam \f1, \fx(0), \fmeas_0) \in \newspace$,
for any $f \in \cbm$
\be
\label{eq-ltime}
 \lim_{t \ra \infty} \lan f, \fmeas_t \ran  = \lan f, \fmeass \ran = \int_{[0,\infty)} f(x) \, (1 - G(x)) \, dx.
\ee
In other words,  $\fmeas_t$ converges weakly to $\fmeass$ as $t \ra \infty$.  \\
\end{enumerate}
\end{theorem}

A precise definition of direct Riemann integrability can be found in
\cite[p.\ 154]{asmbook}.  It is satisfied by most distributions of interest.
The proof of Theorem \ref{th-larget} is presented in Section \ref{sec-longtime}.
For the case $\flam < 1$, property (1) of Theorem \ref{th-larget} was stated as
Theorem 7.3 of \cite{whifluid06} without proof.

\begin{remark}
{\em Our main theorems hold for the majority of
distributions that arise in practice, including the exponential,
lognormal, phase type and uniform distributions and, when $\fmeas_0$ and $\fe$ are absolutely
continuous, the Weibull and Pareto distributions.  It does not, however,  cover
the deterministic distribution.}
\end{remark}

\subsection{Fluid Limits of Other Functionals}
\label{subs-funcflimit}

In the last section, we  identified the fluid limit of
the scaled number of customers in system.
In fact, the fluid limit contains a lot more information.
For instance, as a direct consequence of the continuous mapping theorem, Theorem \ref{th-conv} also identifies
the limit, as $N \ra \infty$,  of  the scaled queue length process $\overline{Z}^{(N)}  = Z^{(N)}/N$,
which is the normalised number of customers waiting in queue (and not in service) at any time: we have
\[ \overline{Z}^{(N)} \doteq \fen - \fkn  \Rightarrow \bar{Z} \doteq \fe - \fk.  \]
Below, we identify the fluid limits of other functionals of interest.

\subsubsection{Waiting and Sojourn Time}
\label{subsub-wait}

The  waiting time functional is of particular interest in the context of call centers,
where service targets are often specified in terms of the proportion of calls
that experience a wait of  less than some given level (see, for example, \cite{brownetal}).

Given a non-decreasing $f \in \mathcal{D}[0,\infty)$ with $L_f \doteq \sup_{s \in [0,\infty)} f(s)$,
consider the following inverse functionals that take values in the extended reals:
\be
\label{def-finv2}
\inv [f] (t) \doteq  \inf \left\{s \geq 0: f(s) \geq t \right\}
\ee
for $t \in [0,L_f]$, with the convention that $[0,L_f] = [0,\infty)$ if $L_f = \infty$,
and if $L_f < \infty$ then $\inv [f] (t) \doteq \infty$ for $t > L_f$.
Likewise, let
\be
\label{def-finv1}
  f^{-1} (t) \doteq \sup \left\{ s \geq 0: f(s) \leq t \right\}
\ee
for $t \in [0,L_f)$ and $f^{-1} (t) = \infty$ for $t \geq L_f$.
The waiting time $\wn_j$ (respectively, the sojourn time $\vn_j$) of the
$j$th customer in the $N$th system is defined to be the time elapsed between
arrival into the system and entry into service (respectively, departure from the system).
 These functionals can be written explicitly as
\be
\label{def-wntime}
\wn (j) \doteq  \inv[\kn](j) - \inv [\en] (j)
\ee
and
\be
\label{def-sojntime}
\vn (j) \doteq \inv[\dn](j) - \inv [\en] (j)
\ee
for $j \in [0,L_{\en}]$.
Also, consider the related processes defined on $[0,\infty)$ by
\[
\fwn (t)  \doteq \wn (\en (t))  \quad \quad \mbox{ and } \quad \quad \fvn \doteq \vn (\en(t))
\]
and note that for $t \in [0,\infty)$,
\[
\fwn (t) = \inv [\fkn] (\fen (t))  - \inv [\fen] (\fen(t))
\]
and similarly,
\[  \fvn = \inv [\fdn] (\fen(t)) -
\inv[\fen] (\fen(t)).
\]
Lastly,   let $\fw$ and $\fv$, respectively, be the processes  given by
\be
\label{def-fwv}
\fw(t) \doteq  \fk^{-1} (\fe(t)) - t   \quad \quad \mbox{ and } \quad \quad
\fv (t) \doteq \fd^{-1} (\fe(t)) - t
\ee
for $t \in [0,\infty)$. 
We will say a function $f \in \mathcal{D}[0,\infty)$
is uniformly strictly increasing if it is absolutely continuous and there exists $\theta > 0$ such that
$\dot{f}(t) \geq \theta$ for all $t \in [0,\infty)$.
Note that for any such function $f^{-1} (f(t)) = t$ and $f^{-1}$ is continuous on $[0,\infty)$.
We have the following fluid limit result for the waiting and sojourn times in the system.

\begin{theorem}
\label{th-wvconv}
Suppose the conditions of Theorem \ref{th-conv} hold and
$\fe$ is uniformly strictly increasing.
If, in addition,  $\fk$  is uniformly strictly increasing
then $\wn \Rightarrow \fw$ as $N \ra \infty$.
Likewise, if $\fd$ is uniformly strictly increasing, then $\vn
\Rightarrow \fv$ as $N \ra \infty$.
\end{theorem}
\begin{proof}
By Assumption \ref{ass-init} and Theorem \ref{th-conv},  it follows that
$\fen \Rightarrow \fe$, $\fkn \Rightarrow \fk$ and $\fdn \Rightarrow \fd$.  Using the
Skorokhod Representation Theorem, we can assume that the convergence  in all three
cases is almost sure.
When combined with the fact that
$\fe$, $\fk$ and $\fd$ are uniformly strictly increasing,  Lemma 4.10 of \cite{ramrei1}
shows that $\inv[f^{(N)}]\ra f^{-1}$ (almost surely, uniformly on compacts) for $f = \fe, \fk$ and $\fd$.
Now, fix $T < \infty$ and $\omega \in \Omega$ such that these limits hold and also fix some $\ve > 0$.
Moreover, let  $N_0 = N_0 (\omega) < \infty$ be such that
for all $N \geq N_0$,
\[ \sup_{s \in [0,\fen(T)]}
 \left[ \inv[f^{(N)}] (s) - f^{-1} (s) \right] \leq \ve
\]
for $f = \fe, \fk, \fd$.
Then we have
\[
\begin{array}{l}
\ds \sup_{t \in [0,T]} |\inv[\fkn](\fen(t)) - \fk^{-1} (\fe (t))| \\
\quad \quad \quad \quad \quad \quad \quad \quad \leq \ds \sup_{t \in [0,T]} |\inv[\fkn](\fen (t)) - \fk^{-1} (\fen (t))| \\
 \ds \quad \quad \quad \quad \quad \quad \quad \quad \quad \quad +
\sup_{t \in [0,T]} |\fk^{-1} (\fen(t)) - \fk^{-1} (\fe(t))|  \\
\ds \quad \quad \quad \quad \quad\quad \quad \quad  \leq \ve + \sup_{t \in [0,T]} |\fk^{-1} (\fen(t)) - \fk^{-1} (\fe(t))|.
\end{array}
\]
The continuity of $\fk^{-1}$  and the fact that a.s., $\fen \ra \fe$ u.o.c.\ as $N \ra \infty$,
 together, ensure  that a.s.\  $|\fk^{-1} (\fen) - \fk^{-1} (\fe)| \ra 0$ u.o.c. as $N \ra \infty$.
So we have
\[ \lim_{N \ra \infty}  \sup_{t \in [0,T]} |\inv[\fkn](\fen(t)) - \fk^{-1} (\fe (t))| \leq \ve. \]
Sending $\ve \ra 0$,  we infer that
$\inv[\fkn] \circ \fen \ra \fk^{-1} \circ \fe$ uniformly on $[0,T]$.
An analogous argument  shows that $\inv[\fen] \circ \fen  \ra \iota$, where $\iota:t \mapsto t$ is the
identity mapping on $[0,\infty)$ and $\inv [\fdn] \circ \fen \ra \fd^{-1} \circ \fe$.
When combined with the definitions of $\fw$ and $\fv$, the theorem follows.
\end{proof}

>From the results in Section \ref{sec-longtime}, it is easy to see that a sufficient condition
for the assumptions of Theorem \ref{th-wvconv} to hold is that
$\fe$ be piecewise linear  with slope strictly bounded away from zero.

\subsubsection{Workload Process }
\label{subsub-workload}

The workload (or unfinished work) process $U^{(N)}$ is defined to be the amount of work
in the $N$th system (including the work of customers waiting in queue and the residual service of
 customers in service):
\[ U^{(N)} (t) = \suli_{j=-\lan \f1, \measn_0 \ran+1}^{\kn(t)} \left( v_j - \agen_j (t) \right)
\ind_{\{\agen_j(t) < v_j\}}  + \suli_{j=\kn(t) + 1}^{\en(t)} v_j.
\]
  Let the scaled workload process $\overline{U}^{(N)}$ be
defined in the usual fashion.  We briefly outline below how the
results and techniques of this paper may be used to characterize
the limit $\overline{U}$ of the sequence of scaled workload
processes $\{\overline{U}^{(N)}\}$.

Let $\eta^{(N)}$ be the measure-valued process (analogous to $\measn$) that represents the residual service  times
(rather than the ages) of customers in service in the $N$th system:  for $t \in [0,\infty)$,
\[ \eta^{(N)}_t \doteq  \suli_{j=-\lan \f1, \measn_0 \ran+1}^{\kn(t)}   \delta_{v_j - \agen_j (t)} \ind_{\{\agen_j(t) < v_j\}} \]
and let $\overline{\eta}^{(N)}$ denote the corresponding scaled quantity.
Fluid equations can be derived for the limit $\overline{\eta}$ of the
sequence $\{\overline{\eta}^{(N)}\}$ in a manner similar to those derived for $\fmeas$ in this paper and,
 under mild assumptions, we believe it can be shown that, as $N \ra \infty$,
 $\overline{\eta}^{(N)} \Rightarrow \overline{\eta}$, where
for every
$f \in  \ccm$
and $t \in [0,\infty)$,
\[ \lan f, \overline{\eta}_t \ran \doteq \int_{[0,M)} \left( \int_0^\infty \dfrac{g(x+u)}{1 - G(x)} f(u) \, du \right)
\fmeas_t (dx).
\]
A completely rigorous proof of this result is beyond the scope of this paper.
However, below we provide a plausible argument to justify the above claim.
Given the age $x$ of any customer that was already in service at time $0$,
the probability that the residual service time of the customer at time
$t$ is greater than $u$ is given by
$(1-G(x+t + u))/(1 - G(x))$.  Thus the density of the residual service
time distribution at time $t$ for a customer that had age $x$ at time $0$ is
$g(x+t+\cdot)/(1-G(x))$.
Likewise, the density of the residual service distribution at time $t$ for a customer
that entered the system at time $0 < s < t$ is $g(t-s+\cdot)$.
 Moreover, given the ages of all customers in service, the residual
service times of customers in service are independent.
Therefore, by a strong law of large numbers reasoning,
one expects that the the limiting residual service measure $\overline{\eta}$  can be written
in terms of the limiting initial age measure $\fmeas_0$  and limiting cumulative entry-into-service process
$\fk$ as follows:  for $f \in \ccm$,
\[\ba{rcl}
\ds \lan f, \overline{\eta}_t \ran & =  & \ds  \int_{[0, M)} \left( \int_0^\infty \dfrac{g(x+t+u)}{1-G(x)} f(u) \, du \right)
\, \fmeas_0 (dx) \\
& & \ds + \int_0^t \left( \int_0^\infty g (t - s + u) f (u) \, du \right) \, d\fk(s).
\ea
\]
The desired result is then obtained by using the representation (\ref{eq2-fmeas})
 to  rewrite the right-hand side above as an integral with respect to $\fmeas_t$.

In turn, since the workload process admits the alternative representation
\[  U^{(N)} (t)  =  \lan x, \eta^{(N)}_t \ran + \suli_{j=\kn(t)+1}^{\en(t) + \xn(0)} v_j, \]
the convergence of $\overline{\eta}^{(N)}$ to $\overline{\eta}$ and the fact that the service times
$\{v_j\}$ are i.i.d.\ with mean $1$
should imply (under mild assumptions that justify the substitution of linear test functions $f$)
the convergence
$\overline{U}^{(N)} \Rightarrow \overline{U}$ as $N \ra \infty$, where
\[ \overline{U} (t)  \doteq  \int_{[0,M)} \left(\int_0^\infty \dfrac{u g(x+u)}{1 - G(x)} \, du \right) \fmeas_t (dx)
+ (\fx (t)  - 1 )^+.
\]

\beginsec

\section{Uniqueness of Solutions to the Fluid Equation}
\label{sec-flsol}

In this section we show that there is at most one solution to the fluid
equation for any given initial condition.
In fact, we will establish two stronger properties of the
 the fluid equation, both of which imply uniqueness.
The first is the continuity of the mapping that takes
$(\fe,\fx(0),\fmeas_0) \in \newspace$
to a corresponding solution $(\fx, \fmeas)$ of the fluid equation, which is
established in Section \ref{subs-cont}.
The second is a maximality property that is established in
Section \ref{subs-max}.
The proofs of  both continuity and maximality rely on identifying the solution to
a certain integral equation, which is carried out
in Section \ref{subs-pde}.
Existence of solutions to the fluid equation
will follow from results  established in the next section (see, in particular,
Theorem \ref{th-tight}).

\subsection{Continuity of the Fluid Equation Map}
\label{subs-cont}

We begin with Theorem \ref{th-pde}, which identifies an implicit relation
that must be satisfied by the processes $\fmeas$ and $\fk$ that satisfy the fluid equations.
This implicit relation, along with the non-idling condition, is then used to
establish continuity of the fluid solution in Theorem \ref{th-flcont}.

\begin{theorem}
\label{th-pde} Given $M, B \in (0,\infty)$, $h \in {\cal C}[0,M)$, $\fmeas_0 \in
\mmb$ and $\newfk \in \incspace$, suppose  $\fmeas \in {\cal
D}_{\mmb}[0,\infty)$ satisfies
 for every   $\newf \in {\cal
C}^1_c([0,M)\times\R_+)$ and $t \in [0,\infty)$,
\be \label{eq-pde} \ba{rcl} \ds \lan
\newft, \fmeas_t \ran  &   = & \ds \lan  \newf (\cdot,0), \fmeas_0
\ran + \int_0^t \lan \dtnewfs, \fmeas_s  \ran \, ds +
\int_0^t \lan \dxnewfs, \fmeas_s  \ran \, ds \\
& & \quad \quad - \ds \int_0^t \lan h(\cdot) \newfs,  \fmeas_s
\ran \, ds + \ds \int_0^t \newf(0,s) \, d \newfk (s). \ea \ee Then
for every  $f \in \cb$, $\fmeas$ satisfies \be \label{eq2-fmeas}
\int_{[0,M)} f (x) \, \fmeas_t (dx)  = \int_{[0,M )} f(x+t)
\dfrac{1 - G(x+t)}{1 - G(x)} \, \fmeas_0 (dx) + \int_0^t f(t-s) (1
- G(t-s)) \, d \newfk (s) \ee
 for all $t \in [0,\infty)$.
\end{theorem}

The integral equation (\ref{eq-pde}) is simply the
fluid equation (\ref{eq-ftmeas}), with $\fk$ replaced by an arbitrary, non-decreasing,
c\`{a}dl\`{a}g function $\newfk$.
Also,  (\ref{eq2-fmeas}) only depends on the values of $f$ in $[0,M)$ since
$f(u) (1 - G(u)) = 0$ for all $u \geq M$.
Moreover,  (\ref{eq2-fmeas}) explicitly characterises
the deterministic measure-valued process $\fmeas$.
Theorem \ref{th-pde}, and consequently the continuity result,
can be shown to hold more generally for $h \in \lloc [0,M)$
for which (\ref{cond-radon}) is satisfied (see Remark \ref{rem-gencont}), but this
extension is omitted from this paper in the interest of simplicity of exposition.

\begin{remark}
\label{remark-pde}
{\em
The last integral in (\ref{eq2-fmeas}) is, as usual, to be interpreted as a Riemann-Stieltjes integral.
A straightforward integration-by-parts shows that  for every }
$f \in \ocdb$ {\em and} $t \in [0,\infty)$,
{\em this   integral also admits the alternative representation  }
\be
\label{eq-ibp}
\ba{l}
\ds  \int_0^t f(t-s) (1 - G(t-s)) \, d \newfk (s) \\
 \quad \quad \quad \quad \quad =  \ds
f(0) \newfk (t)  + \int_0^t  f^\prime (t-s) (1 - G(t-s))\newfk(s) \, ds  \\
 \ds \quad \quad \quad \quad \quad \quad \quad - \int_0^t  f(t-s) g(t-s)\newfk(s) \, ds.
\ea
\ee
\end{remark}

The proof of Theorem \ref{th-pde} involves PDE techniques and is relegated to
Section \ref{subs-pde}.  As a simple corollary of Theorem \ref{th-pde}, we have the following result.

\begin{cor}
\label{cor-pde}
Let $(\fx, \fmeas)$ be a solution to the fluid equations associated with
$(\fe, \fx(0), \fmeas_0) \in \newspace$ and $h$ that is continuous on  $[0,M)$.
Then the function  $\fk$ defined by (\ref{eq-fk})  satisfies the implicit equation
\be
\label{eq2-fk}
\ba{rcl}
 \fk(t) & = & \ds  \lan \f1, \fmeas_t, \ran - \lan \f1, \fmeas_0 \ran
+ \int_{[0,M)} \dfrac{G(x+t) - G(x)}{1 - G(x)} \, \fmeas_0 (dx) \\
& & \quad  + \ds \int_0^t g(t-s) \fk(s) \, ds
\ea
\ee
 for every $t \in [0,\infty)$.
\end{cor}
\begin{proof}
We first claim that if $(\fx, \fmeas)$ solve the fluid equations, then
$\fk$ defined by (\ref{eq-fk}) must necessarily be non-decreasing (as one would
expect from the interpretation of $\fk$ as the limiting fraction of cumulative entries into service).
In order to justify the claim, fix $0 \leq s \leq t$.
If $\fx (t) > 1$ then the non-idling condition (\ref{eq-fnonidling}) implies that
 $\lan \f1, \fmeas_t \ran = 1$ which, when substituted into
(\ref{eq-fk}), shows that
\[ \fk (t) - \fk(s) = 1 - \lan \f1, \fmeas_s \ran + \int_s^t \lan h, \fmeas_s \ran \, ds
\geq  \int_s^t \lan h, \fmeas_s \ran \, ds \geq 0.
\]
On the other hand, if $\fx (t) \leq 1$ then the non-idling condition (\ref{eq-fnonidling})
shows that  $\lan \f1, \fmeas_t \ran = \fx(t)$ and
$\lan \f1, \fmeas_s \ran \leq \fx (s)$.   Hence
(\ref{eq-fk}), together with (\ref{eq-fx}) and the fact that $\fe$ is non-decreasing,
shows that
\[ \fk (t) - \fk(s)
 = \lan \f1, \fmeas_t \ran - \fx(t) - \lan \f1, \fmeas_s \ran + \fx (s) + \fe (t) - \fe (s)
\geq 0,
\]
which proves the claim.

In addition, by assumption, $\fk$ and $\fmeas$ satisfy
(\ref{eq-ftmeas}) and so Theorem \ref{th-pde} applies, with
$\newfk = \fk$. Substituting $f = \f1 \in \ocdb$ and $\newfk =
\fk$ in (\ref{eq2-fmeas}), we obtain the relation
\[\ba{rcl}
 \ds \int_0^t (1 - G(t-s)) \, d\fk(s) & = & \ds \lan \f1, \fmeas_t \ran - \int_{[0,M)}
\dfrac{1 - G(x+t)}{1 - G(x)} \, \fmeas_0 (dx) \\
& = & \ds \lan \f1, \fmeas_t \ran - \lan \f1, \fmeas_0 \ran +
\int_{[0,M)} \dfrac{G(x+t) - G(x)}{1 - G(x)} \, \fmeas_0 (dx).
\ea
\]
On the other hand, equation (\ref{eq-ibp}) with $\newfk = \fk$ and $f = \f1$
shows that
\[   \ds \int_0^t (1 - G(t-s)) \, d\fk(s) = \fk(t)  - \int_0^t g(t-s) \fk(s) \, ds.
\]
Equating the right-hand sides of the last two displays, we obtain
(\ref{eq2-fk}).
\end{proof}

As an immediate consequence of Theorem \ref{th-pde} and Corollary
\ref{cor-pde}, we obtain the following simple bound.   Given $\mu
\in \mmf$, here $|\mu|$ is used to denote the total variation
measure associated with $\mu$ and $|\mu|_{TV}$ represents the
total variation of $\mu$ on $[0,M)$.

\begin{lemma}
\label{lem-elem} For $i = 1,2$, suppose  $\fmeas_0^{(i)} \in
\mmsp$ and $\newfk^{(i)} \in \incspace$ are given, and suppose
$\fmeas^{(i)} \in {\cal D}_{\mmsp}(\R_+)$ satisfies
(\ref{eq-fmeas}) with $\fk$ replaced by $\newfk^{(i)}$. Then for
every $T < \infty$ and $f \in \ocdb$,
\be
\label{ineq-fmeas}
 \nrm{\lan f, \fmeasp_s \ran - \lan f, \fmeaso_s \ran}_T \leq
\nrm{f}_\infty |\Delta \fmeas_o|_{TV} + \left( 2 \nrm{f}_T +
\nrm{f^\prime}_T \right) \nrm{\Delta \newfk}_T,
\ee
where $\Delta \newfk \doteq \newfk^{(2)} - \newfk^{(1)}$ and $\Delta \fmeas_0
\doteq \fmeasp_0  - \fmeaso_0$.
\end{lemma}
\begin{proof}
The relations (\ref{eq2-fmeas}) and (\ref{eq-ibp})  together
imply that for $f \in \ocdb$ and $t \in [0,\infty)$,
\[
\ba{rcl}
 \lan f, \fmeasp_t \ran - \lan f, \fmeaso_t \ran  & = & \ds \int_{[0, M)} f(x+t) \dfrac{1 - G(x+t)}{1 - G(x)} \, \Delta \fmeas_0 (dx) + \ds  f(0) \Delta \newfk (t)  \\
& &  + \ds  \int_0^t f^\prime (t-s) (1 - G(t-s)) \Delta \newfk (s)\, ds \\
& & -\ds \int_0^t  f(t-s)  g(t-s)\Delta \newfk (s)  \, ds. \ea
\]
This implies that for $f \in \ocdb$
and for every $t \in [0,T]$,
\[ \left|\lan f, \fmeasp_t \ran - \lan f, \fmeaso_t \ran \right|
\leq \int_{[0,\infty)} |f(x+t)| |\Delta \fmeas_0|
(dx) + \left( |f(0)| + \nrm{f}_t + \nrm{f^\prime}_t \right) \nrm{
\Delta \newfk}_T
\]
 from  which  (\ref{ineq-fmeas}) follows.
\end{proof}

We now state the main result of this section.
Below,  $\Delta H$ denotes $H^{(2)} - H^{(1)}$ for $H = \fk, \fd, \fe, \fx$ and $\fmeas$.

\begin{theorem} {\bf (Continuity of  Solution Map)}
\label{th-flcont}
Suppose $h$ is continuous on $[0,M)$, and  for $i = 1, 2$,  let
 $(\fx^{(i)}, \fmeas^{(i)})$ be the solution to the fluid equations associated with
$(\fe^{(i)}, \fx^{(i)}(0), \fmeas_0^{(i)}) \in \newspace$ and let
$\fk^{(i)}$ and $\fd^{(i)}$ be defined as in (\ref{eq-fk}) and (\ref{def-fd}), respectively,
with $\fmeas$ replaced by $\fmeas^{(i)}$.
If $\fmeaso_0  = \fmeasp_0$
 then for every $T < \infty$,
\be
\label{ineq-cont0}
\left[ \sup_{t \in [0,T]} \Delta \fk (t) \right] \vee \left[ \sup_{t \in [0,T]} \Delta \fd (t) \right]
\leq \left[ \Delta |\fx (0)| + \sup_{t \in [0,T]} \Delta \fe (t) \right] \vee 0
\ee
and hence
\be
\label{ineq-cont}
 \nrm{\Delta \fk}_T  \vee \nrm{\Delta \fd}_T
\leq |\Delta \fx (0)| + \nrm{\Delta \fe}_T.
\ee
Moreover, for every $T < \infty$ and $f \in \ocdb$,
\be \label{ineq-cont1}
 \nrm{\lan f, \fmeasp_s \ran - \lan f, \fmeaso_s \ran}_T \leq
 \left( 2 \nrm{f}_T + \nrm{f^\prime}_T \right) \left( \Delta \fx (0) + \nrm{\Delta \fe}_T \right).
\ee
\end{theorem}
\begin{proof}
For $i = 1, 2$, fix $(\fe^{(i)}, \fx^{(i)}(0), \fmeas_0^{(i)}) \in \newspace$ and
$(\fx^{(i)}, \fmeas^{(i)}, \fk^{(i)}, \fd^{(i)})$ as in the
statement of the theorem.
Fix $T < \infty$ and define $\ve \geq 0$ by
\be
\label{def-ve1}
 \ve \doteq \left[ \Delta \fx(0) +   \sup_{t \in [0,T]} \Delta \fe (t)  \right] \vee 0.
\ee Choose $\delta > 0$ and define $\tau \doteq \tau_\delta \doteq
\inf \{ t \geq 0: \Delta \fk (t) >  \ve + \delta \}$. The
right-continuity of $\fko$ and $\fkp$ then imply that \be
\label{implication} \Delta \fk (\tau) \geq \ve + \delta. \ee
 We now show that $\tau  > T$.  Indeed, suppose $\tau \in [0,T]$ and consider the
following two cases. \\
{\em Case 1. } $ \fxo (\tau) < 1$.  In this case, the non-idling condition (\ref{eq-fnonidling})
 implies that
\[ \fxo (\tau) - \lan \f1, \fmeaso_\tau \ran  = 0 \leq \fxp (\tau) - \lan \f1, \fmeasp_\tau \ran. \]
Together with relations (\ref{eq-fk}), (\ref{eq-fx}),
 (\ref{def-ve1}) and the fact that $\Delta \fmeas_0  \equiv 0$,   this implies that
\[
 \Delta \fk (\tau) = \Delta \fe (\tau) +  \Delta \fx (0) - \lan \f1, \Delta \fmeas_0 \ran
 -  \Delta \fx (\tau) + \lan \f1, \Delta \fmeas_\tau\ran
  \leq  \ve,
\]
which contradicts (\ref{implication}).

\noi
{\em Case 2.}   $\fxo (\tau) \geq 1$.
In this case, due to the non-idling condition (\ref{eq-fnonidling}), we have
$\lan \f1, \fmeaso_\tau \ran = 1 \geq \lan \f1, \fmeasp_\tau \ran$.  Along with
Corollary \ref{cor-pde} and the fact that $\Delta \fmeas_0 = 0$, this  implies that
\[
\begin{array}{rcl}
\Delta \fk (\tau) &  = & \ds  \lan \f1, \Delta \fmeas_\tau \ran  - \lan \f1, \Delta \fmeas_0 \ran
 + \int_{[0,M)} \dfrac{G(x+\tau) - G(x)}{1- G(x)}
\, \Delta \fmeas_0 (dx) \\
& & \ds  + \int_0^\tau g(\tau-s) \Delta \fk (s) \, ds\\
& = & \ds \lan \f1,  \Delta \fmeas_\tau \ran   + \int_0^\tau g(\tau-s) \Delta \fk (s) \, ds\\
& \leq & \ds \int_0^\tau g(\tau-s) \Delta \fk (s) \, ds.
\end{array}
\]
We now assert that the right-hand side is {\em strictly} less than
$\ve + \delta$, which contradicts (\ref{implication}).
To see why the assertion holds, note that
if $g(s) = 0$ for a.e.\ $s \in [0,\tau]$, then the right-hand side of the last inequality equals zero,
which is trivially strictly less than $\ve + \delta$.
On the other hand, if $g(s) > 0$ for a set of positive Lebesgue measure in $[0,\tau]$, then the
fact that
$\Delta \fk (s) < \ve + \delta$ for all $s \in [0,\tau)$ shows once again that
\[ \Delta \fk (\tau) \leq \int_0^\tau g(\tau-s) \Delta \fk (s) \, ds < (\ve+ \delta) G(\tau) \leq (\ve +\delta).
\]

In both cases we arrive at  a contradiction, and therefore it must be that $\tau > T$, which means
 that $\Delta \fk (t) \leq  \ve + \delta$ for every $t \in [0,T]$.
Sending $\delta \downarrow 0$, we conclude that
$\Delta \fk (t) \leq  \ve$ for $t \in [0,T]$, as desired.
In turn, using the relations (\ref{eq-fk}), (\ref{def-fd}) and
Corollary \ref{cor-pde},
along with the fact that $\Delta \fmeas_0 \equiv 0$
and $g$ is non-negative,   we obtain for every $t \in [0,T]$,
\[   \Delta \fd (t)  =  \Delta \fk (t) - \lan \f1, \Delta \fmeas_t \ran =  \int_0^t g(t - s) \Delta \fk (s)  \, ds  \leq \ve G(t) \leq \ve.\]
This completes the proof of (\ref{ineq-cont0}), and
relation (\ref{ineq-cont}) follows by symmetry.
Lastly,  since for $i=1,2$, $\fmeas^{(i)}$ and $\fk^{(i)}$ satisfy the fluid equations (by assumption), inequality
(\ref{ineq-cont1})    is a direct consequence of Lemma \ref{lem-elem}
and inequality (\ref{ineq-cont}).
\end{proof}


\noi
{\bf Proof of Theorem \ref{th-fluniq}.}
Let $(\fxo,\fmeaso)$ and $(\fxp,\fmeasp)$ be two solutions
to the fluid equations corresponding to $(\fe, \fx(0), \fmeas_0) \in \newspace$.
Fix $a \in [0,\infty)$ and choose a sequence
of functions $f_n \in \ocdb$, $n \in \N$, such that $f^n \uparrow \ind_{[0,a]}$ pointwise
as $n \ra \infty$.
Then for every $t \in [0,\infty)$ and $n \in \N$,
$\lan f_n, \fmeaso_t \ran = \lan f_n, \fmeasp_t \ran$ due to
relation (\ref{ineq-cont1}).   Sending $n \ra \infty$ and
invoking the monotone convergence theorem, we conclude that
$\fmeaso_t([0,a)) = \fmeasp_t([0,a))$.  Since $a$ and $t$ are arbitrary,
it follows that $\fmeaso = \fmeasp$ and hence,
by (\ref{eq-fx}),
 that $\fxo = \fxp$, thus establishing uniqueness.
The representation (\ref{final-fleqs}) follows immediately from Theorem \ref{th-pde}.

Now,  if $\fe$ is absolutely continuous, then (\ref{eq-fx}) immediately
shows that $\fx$ is also absolutely continuous.
In turn, using (\ref{eq-fnonidling}), (\ref{eq-fk}) and the fact that
 $|[1 - a]^+ -[1 - b]^+| \leq |a - b|$,   it is easy  to see that
$\fk$ is also absolutely continuous.
If $\fx (t) < 1$, then the non-idling condition (\ref{eq-fnonidling}) and the continuity of $\fx$
 show that $\lan \f1, \fmeas_s \ran  = \fx(s)$ for all $s$ in a neighbourhood of $t$.
When combined with (\ref{eq-fx}) and (\ref{eq-fk}), this shows
that $\fsmallk(t) = \flam (t)$. On the other hand, if $\fx(t)
> 1$, then (\ref{eq-fnonidling}) and the continuity of $\fx$ show
that $\lan \f1, \fmeas_s \ran = 1$ for $s$ in a neighbourhood of
$t$. When substituted into (\ref{eq-fk}) this shows that
$\fsmallk(t) = \lan h, \fmeas_t\ran$. Lastly, since $\fx$ and
$\lan \f1, \fmeas\ran$ are absolutely continuous, $d\fx(t)/dt =
d\lan \f1, \fmeas_t\ran /dt = 0$ for a.e.\ $t$ on which $\fx(t) =
\lan \f1, \fmeas_t \ran = 1$ (see, for example, Theorem A.6.3 of
\cite{dupellbook}). When combined with (\ref{eq-fx}) and
(\ref{eq-fk}) this shows that for a.e.\ $t \in [0,\infty)$ such
that $\fx(t) = 1$, 
we have $\fsmallk(t) = \flam(t) = \lan h, \fmeas_t \ran = \flam(t) \wedge \lan h, \fmeas_t \ran$,
 which proves (\ref{def-fsmallk}). Finally, since
$\fk$ is absolutely continuous,  if $\fmeas_0$ is also absolutely
continuous then the representation (\ref{final-fleqs}) immediately
guarantees that $\fmeas_s$ is absolutely continuous for every $s
\in [0,\infty)$.
\qed \\

To close the section,
we state a simple consequence of uniqueness of
the fluid limit that will be used in Section \ref{sec-longtime}.
For this, we  require the following notation:
for any $t \in [0,\infty)$,
\[  \fe^t \doteq \fe (t + \cdot) - \fe (t)  \quad \quad  \fk^t  \doteq \fk(t + \cdot) - \fk(t) \quad \quad
\fx^t \doteq \fx (t +\cdot)  \quad \quad \fmeas^t \doteq \fmeas_{t+\cdot}. \]

\begin{lemma}
\label{lem-sgroup} Fix $h \in \lloc [0,M)$ for which there
exists a unique solution to the fluid equation associated with a
given initial condition $(\fe, \fx (0), \fmeas_0) \in \newspace$
and $h$,
 and let $\fk$ be the associated process that satisfies (\ref{eq-fk}). 
Then for any $t \in [0,\infty)$,  $(\fx^t, \fmeas^t)$ is a solution
to the fluid equation associated with the initial condition  $(\fe^t, \fx^t (0), \fmeas^t_0) \in \newspace$,
 and $\fk^t$ is the corresponding process that satisfies  (\ref{eq-fk}), with  $\fmeas$ replaced by $\fmeas^t$.
\end{lemma}

\noi
The proof  of the lemma is elementary, involving just a rewriting of the fluid equations, and is thus omitted. \\

\subsection{A Maximality Property of the Fluid Solution}
\label{subs-max}

In this section we establish a result of independent interest.
This  result is not used in the rest of the paper, and can thus be
safely skipped without loss of continuity.
Specifically, we show that the non-idling property (\ref{eq-fnonidling})
implies a  certain maximality property for solutions to the fluid equations.
In particular, this result provides an alternative   proof of uniqueness
of solutions to the fluid limit that is different from the one using
continuity of the solution map given in the last section.

Fix $h$ that is continuous on $[0,M)$ and
$(\fe, \fx(0), \fmeas_0) \in {\cal S}_0$.
Suppose that $(\fx, \fmeas)$ solve
 the corresponding fluid equations (\ref{eq-ftmeas})--(\ref{eq-fnonidling}),
and let $\fk$ and $\fd$ be the associated processes, as defined in
(\ref{eq-fk}) and (\ref{def-fd}), respectively.
Also, let $(\fx^\prime, \fmeas^\prime)$ be any process taking values in
$\R_+ \times \mmsp$ that satisfy the fluid equations,
(\ref{eq-ftmeas}) and (\ref{eq-fx}), and the relation \be
\label{max-con} \lan \f1, \fmeas_t^\prime \ran \leq \fx^\prime (t)
\quad \mbox{ for } t \in [0,\infty). \ee Here, $\fx^\prime$ and
$\fmeas^\prime$, respectively, represent the total number of
(fluid) customers in system and the distribution of ages of
(fluid) customers in service under any given feasible assignment
of customers to servers that does not necessarily satisfy the
non-idling condition (\ref{eq-fnonidling}). For
 $t \in [0,\infty)$,
let $\fk^\prime(t)$ and $\fd^\prime(t)$, respectively,
be the corresponding processes
 representing the cumulative entry into service and cumulative departures from the system,
as defined by the right-hand sides of (\ref{eq-fk}) and (\ref{def-fd}), respectively,
 but with $\fmeas$ replaced
by $\fmeas^\prime$. 
Then we have the following intuitive result that shows that the non-idling condition
(\ref{eq-fnonidling}) ensures that the cumulative entry into service and
cumulative departures from the system are maximised.

\begin{lemma}
\label{lem-max}
For every $t \in [0,\infty)$,  $\fk(t)\ge \fk^\prime (t)$ and $\fd(t)\ge \fd^\prime (t)$.
\end{lemma}
\begin{proof} 
We shall argue by contradiction to prove the lemma.
Fix $\ve>0$ and let
$$T=\inf\{t:\fk^\prime(t)\ge\fk(t)+\ve\}.$$
Suppose $T < \infty$.  Then we consider the following
two mutually exhaustive cases.

\noi
\emph{Case 1.} $\fx_T<1$.  In this case, (\ref{eq-fnonidling}) implies that
$\fx_T = \lan \f1, \fmeas_T \ran$ which, along with (\ref{eq-fx})
and (\ref{eq-fk}), shows that
\[ \fk (T) = \fx (0) - \lan \f1, \fmeas_0 \ran + \fe (T).
\]
On the other hand,  (\ref{eq-fx}),
 (\ref{eq-fk}) and (\ref{max-con}), when combined, show that for every $t \in [0,\infty)$,
\[ \fk^\prime(t) = \lan \f1, \fmeas_t^\prime \ran - \fx^\prime (t) + \fx(0) -  \lan \f1, \fmeas_0 \ran + \fe(t)
\leq \fx(0) - \lan \f1, \fmeas_0 \ran + \fe (t).
\]
The last two equations imply that $\fk^\prime(T)\leq \fk(T)$, which contradicts the definition of
$T$.  \\

\noi
 \emph{Case 2.}  $\fx_T\ge 1$.  In this case, (\ref{eq-fnonidling}) shows that
$\lan \f1, \fmeas_T \ran = 1$.
Since the pair $\fmeas$ and $\fk$, as well as the pair $\fmeas^\prime$ and $\fk^\prime$,
satisfy the fluid equation (\ref{eq-ftmeas}), Corollary \ref{cor-pde} and
 (\ref{eq-fk}) show that 
\begin{eqnarray*} \fd (T)&=& \int_0^T \lan h,\fmeas_s\ran \, ds \\
&=&\int_{[0,\infty)} \dfrac{G(x+T)-G(x)}{1-G(x)} \, \fmeas_0(dx)
+ \int_0^T g(T-s) \fk (s) ds\\
&>&\int_{[0,\infty)} \dfrac{G(x+T)-G(x)}{1-G(x)} \, \fmeas_0(dx)
+ \int_0^T g(T-s) \left( \fk^\prime(s)-\ve\right)\, ds \\
&=&\int_{[0,\infty)} \dfrac{G(x+T)-G(x)}{1-G(x)}\, \fmeas_0(dx)+\int_0^T
g(T-s)\fk^\prime(s) \, ds- \ve G(T)\\
&=& \fd^\prime (T)-\ve G(T),
\end{eqnarray*}
from which we conclude  that $\fd^\prime (T) < \fd (T) + \ve$
(To be precise, the strict inequality above holds only if there exists a subset of $[0,T]$ of
positive Lebesgue measure, on which $g$ is strictly positive -- but, as in {\em Case 2} of the
proof of Theorem \ref{th-flcont}, it is easy to see that the conclusion continues to hold even
when $g$ is a.e.\ zero on $[0,T]$).
When combined with (\ref{eq-fk}), we see that
\begin{eqnarray*}
 \fk (T) - \fk^\prime (T) & = & \lan \f1, \fmeas_T \ran - \lan \f1, \fmeas_T^\prime \ran
+ \fd(T) - \fd^\prime (T)  \\
& = & 1 - \lan \f1, \fmeas^\prime_T \ran + \fd (T) - \fd^\prime (T) \\
& \geq &  - \ve,
\end{eqnarray*}
which again contradicts the definition of $T$.

Thus we have shown that $T = \infty$ or, equivalently,
that $\fk (t) \geq \fk^\prime (t) - \ve$ for every $t \in [0,\infty)$ and $\ve > 0$.
Sending $\ve \ra 0$, we conclude that $\fk (t) \geq \fk^\prime (t)$ for $t \in [0,\infty)$.
In turn, when combined with (\ref{eq-fx}), Corollary \ref{cor-pde} and the fact that
$\fmeas_0^\prime  = \fmeas_0$, we see that for every $t \in [0,\infty)$, we have
\[ \fd (t) - \fd^\prime (t) = \int_0^t g (t-s) \left( \fk (s) - \fk^\prime(s) \right) \, ds \geq 0.
\]
which concludes the proof of the lemma.
\end{proof}

\begin{remark}
\label{rem-max} {\em A similar maximality property (in terms of a
stochastic, rather than pathwise, ordering) is satisfied by
``pre-limit'' processes describing GI/G/N queues (see, for
example, \cite[Theorem 1.2 of Chapter XII]{asmbook}). It is also
worthwhile to note the connection between Lemma \ref{lem-max} and
a minimality property  associated with the one-dimensional
reflection map that is used to characterize single-server queues.
In the latter case,  the so-called complementarity condition plays
the role of the non-idling condition here, and  ensures
 minimality of the associated constraining term (see, for instance, \cite{harbook}).
}
\end{remark}

\subsection{Analysis of the Integral Equation (\ref{eq-pde})}
\label{subs-pde}

Theorem \ref{th-pde} essentially states that  the unique (weak) solution to the
integral equation (\ref{eq-pde}) is  given by (\ref{eq2-fmeas}).
When $M = \infty$, our proof of the theorem consists of two  main steps, involving
 the following simpler integral equation for $\overline{\mu} \in {\cal D}_{\mb}[0,\infty)$:
for every $\tnewf \in \newocdcp$,
\be
\label{eq-pde2}
\ba{rcl} \ds
\lan \tnewft, \newfmeas_t \ran & = & \ds \lan \tnewf(\cdot,0),
\newfmeas_0 \ran +
\int_0^t \lan \tilde{\newf}_s (\cdot, s), \newfmeas_s \ran \, ds
+ \int_0^t \lan \tilde{\newf}_x (\cdot, s), \newfmeas_s \ran \, ds \\
& & \ds +\int_{[0,t]}  \tnewf(0,s) \, d \newfk (s)
\ea
\ee
for all $t \in [0,\infty)$.
In  Proposition \ref{prop-pde2} we first show that the simpler integral
equation (\ref{eq-pde2}) has a unique solution, and identify it explicitly.
Then, in Proposition \ref{prop-pde1}, we show that there is a one-to-one correspondence
between solutions $\fmeas$ to the integral equation (\ref{eq-pde}) and
solutions $\newfmeas$ to the simpler equation (\ref{eq-pde2}).
The proof for the case $M < \infty$ follows essentially the same line of reasoning,
although it uses the following,  slightly more involved, notation: let
\[  \mnewocdcp \doteq \left\{ \tnewf \in \newocdcp: \supp (\tnewf) \subset [0,M) \times [0,\infty) \right\} \]
and
\[ \mnewccp \doteq \left \{ \tnewf \in \newccp: \supp (\tnewf) \subset [0,M) \times [0,\infty) \right\}. \]
The reader may find it useful to first read the proofs under the supposition that $M = \infty$, in which
case $\mnewocdcp = \newocdcp$ and $\mnewccp = \newccp$.

\begin{prop}
\label{prop-pde2} For any $B, M \in (0, \infty)$, given $\newfmeas_0 \in \mb$ and $\newfk \in
\incspace$,
a measure-valued function  $\newfmeas \in {\cal
D}_{\mb}[0,\infty)$ satisfies (\ref{eq-pde2}) for every $\tnewf \in \mnewocdcp$
if and only if for every $\tnewf \in \mnewccp$,
\be
\label{eq-newfmeas}
\lan \tnewft, \newfmeas_t \ran =
\int_{[0,\infty)} \tnewf(x+t,t) \, \newfmeas_0 (dx) +
\int_0^t \tnewf (t-s,s) \, d \newfk (s)  \quad \quad \forall  t \in [0,\infty).
\ee
\end{prop}
\begin{remark}
\label{rem-f}
{\em Given any function $f \in \cc$ that has $\supp (f) \subset [0,M)$, we can
identify $f$ with a function $\tnewf \in \mnewccp$
that is only a function of $x$.  So (\ref{eq-newfmeas}) implies, in particular, that
\[ \lan f, \newfmeas_t \ran = \int_{[0,\infty)} f(x+t)\, \fmeas_0 (dx) + \int_0^t f(t-s) \, d \newfk (s)
 \quad \quad \forall  t \in [0,\infty). \]
Thus Proposition \ref{prop-pde2} shows that
if (\ref{eq-pde2}) holds for all $\tnewf \in \mnewocdcp$
then the restriction of the measure $\newfmeas_t$ to  $[0,M)$, for every $t \in [0,\infty)$, is completely
determined.
}
\end{remark}

\noi {\em Proof of Proposition \ref{prop-pde2}. } Suppose $\newfmeas =
(\newfmeas_t,t \in [0,\infty))$ satisfies (\ref{eq-newfmeas}).
  We first show that then (\ref{eq-pde2}) is satisfied for any
$\tnewf \in \mnewocdcp$ and $t = T < \infty$.
Indeed, since $\tnewf_s +
\tnewf_x$ lies in $\mnewccp$,
substituting $t =s$ and $\tnewf =
\tnewf_s + \tnewf_x$ into (\ref{eq-newfmeas}),  integrating over
$s$ in $[0,T)$ and applying Fubini's theorem, we obtain 
\[
\ba{l}
\ds \int_0^T \lan \tnewf_s (\cdot, s) + \tnewf_x (\cdot, s), \newfmeas_s \ran \, ds \\
 \quad \quad \quad  =   \ds
\int_0^T \left( \int_{[0,\infty)} \left( \tnewf_s (x+s,s) + \tnewf_x (x+s,s) \right) \newfmeas_0 (dx) \right) \, ds  \\
\quad  \quad \quad \quad \quad +\ds \int_0^T \left(\int_0^s \left( \tnewf_s(s-u,u) + \tnewf_x (s-u,u) \right) \, d\newfk(u) \right) \, ds \\
\quad \quad \quad = \ds  \int_{[0,\infty)} \left( \int_0^T \left( \tnewf_s (x+s,s) + \tnewf_x (x+s,s) \right) \, ds \right) \, \newfmeas_0 (dx) \\
\quad \quad \quad \quad \quad \ds + \int_0^T \left( \int_u^T \tnewf_s  \left( (s - u,u) + \tnewf_x (s-u,u) \right) \, ds \right) \, d \newfk (u) \\
\quad \quad \quad  =  \ds \int_{[0,\infty)} \left( \tnewf(x +T,T) - \tnewf(x,0) \right) \,
\newfmeas_0 (dx) + \int_0^T \left( \tnewf(T-u,u) - \tnewf (0,u) \right) \, d \newfk (u).
\ea
\]
Another application of (\ref{eq-newfmeas}), with $t = T$,  then yields the equality
\[  \int_0^T \lan \tnewf_s(\cdot, s) + \tnewf_x (\cdot, s), \newfmeas_s \ran \, ds
 = \lan \tnewf (\cdot, T), \newfmeas_T \ran - \lan \tnewf (\cdot,0), \newfmeas_0 \ran -
\int_0^T \tnewf(0,u) d\newfk(u),
\]
which proves that $\tnewf$ and $\newfmeas$ satisfy (\ref{eq-pde2}), thereby establishing
the ``if'' part of the proposition.

To prove the converse,  suppose $\newfmeas^{(1)}$ and $\newfmeas^{(2)}$
are two measure-valued functions that satisfy  (\ref{eq-pde2}) for all $\tnewf \in \mnewocdcp$,
define $\overline{\eta}
\doteq \newfmeas^{(1)} - \newfmeas^{(2)}$ and
let $\overline{\eta}^M = \{\overline{\eta}^M_t, t \in [0,\infty)\}$,
where $\overline{\eta}^M_t$ is the restriction of $\overline{\eta}_t$ to $[0,M)$.
In order to prove the ``only if'' part of the proposition, it suffices
to show that $\overline{\eta}^M \equiv 0$ because this implies that
the equation (\ref{eq-pde2}) (holding for all $\tnewf \in \mnewocdcp$)
  uniquely determines the family of functions
$\{ \lan \tnewf (\cdot, t), \newfmeas_t \ran, t \in [0,\infty)\}$, for $\tnewf \in \mnewccp$,
which must equal the right-hand side of (\ref{eq-newfmeas}) due to the argument in the previous paragraph.

Due to (\ref{eq-pde2}), for every $\tnewf \in \newocdcpm$, we know that for $t \in [0,\infty)$,
\be
\label{eta-pde}
  \lan \tnewf (\cdot, t), \overline{\eta}^M_t  \ran
 = \int_0^t \lan \tnewf_x (\cdot, s) + \tnewf_s (\cdot, s), \overline{\eta}^M_s \ran \, ds
\ee
and we also have $\eta_0 \equiv 0$ and $\eta (\{0\}, t ) = 0$ for all $t \in [0,\infty)$.
Now observe that $\overline{\eta}^M$ induces the following linear functional on $\newocdcpm$:
\[  \overline{\eta}^M (\cdot): \tnewf \mapsto \int_0^\infty \lan \tnewf, \overline{\eta}^M_s \ran \, ds, \]
with
\[ |\overline{\eta}^M (\tnewf)|  \leq 2 B \nrm{\tnewf}_\infty \]
since for $i = 1, 2$, $\overline{\mu}^{(1)}_s$ has total measure no greater
than $B$ for every $s \in [0,\infty)$.
Thus $\overline{\eta}^M$ can be viewed as a Radon (in fact finite)
measure on $[0,M) \times \R_+$ (see the comments
in Section \ref{subsub-funmeas}).
Sending $t \ra \infty$ in (\ref{eta-pde}), the left-hand side goes to zero since
$\tnewf$ has compact support, and so we have
\[\overline{\eta}^M (\tnewf_x + \tnewf_s) =  \int_0^\infty \lan \tnewf_x (\cdot, s) + \tnewf_s (\cdot, s), \overline{\eta}^M_s \ran \, ds
= 0 \quad \quad \forall \tnewf \in \newocdcpm \]
(which is equivalent to saying that
$\overline{\eta}^M$ is a weak solution to the
linear, homogeneous PDE  $\partial \overline{\eta}^M/\partial x + \partial \overline{\eta}^M/\partial s = 0$). 
It is straightforward to see that  $\overline{\eta}^M \equiv 0$ is
the unique Radon measure satisfying
the above relation.  Specifically, for any $\psi \in \newocdcpm$,
if we let  $\tilde{\psi}$ denote the natural extension  of $\psi$ to $\R_+^2$:
\[ \tilde{\psi} (x,t) = \left\{
\ba{rl}
 \psi(x,t)  & \mbox{ if } (x,t) \in [0,M) \times \R_+ \\
0 & \mbox{ if } x \geq M
\ea
\right.
\]
then it is easy to verify that the function
\[ \tnewf (x,t) = - \int_0^\infty  \tilde{\psi} (x+s, t+s) \, ds \quad \quad \mbox{ for } (x,t) \in [0,M) \times \R_+ \]
lies in $\newocdcpm$ and satisfies (in the classical sense)
the PDE $\tnewf_x + \tnewf_s = \psi$.  This shows that
$\overline{\eta}^M (\psi) = 0$ for all $\psi \in \newocdcpm$.
A  standard
approximation argument can then be used to show that in fact $\overline{\eta}^M (\psi) = 0$ for
all $\psi \in \newccpm$, which is equivalent to the assertion  $\overline{\eta}^M \equiv 0$.
 \qed \\

In order to  state the
 correspondence between solutions to the integral equation (\ref{eq-pde}) of interest and solutions to
the simpler integral equation (\ref{eq-pde2}),
 we will need the following notation.
Given $h \in \lloc[0,M)$,  let
 \be \label{def-psih} \psi^h (x,t)
\doteq \ds \exp (r^h(x,t)) \ee for $(x,t) \in [0,M)\times\R_+$,
where \be \label{def-r}
   r^h (x,t) \doteq  \left\{
\ba{rl}
\ds -\int_{x-t}^x h(u) \, du   & \ds \mbox{ if }  0 \leq t \leq x,  \\
\ds -\int_{0}^x h(u) \, du   & \ds \mbox{ if } 0 \leq x \leq t.
\ea \right. \ee Note that then $r^h$ and $\psi^h$ are continuous,
locally bounded functions on $[0,M)\times\R_+$.
 Hence, for every $t \in [0,\infty)$ and any non-negative Radon measure $\xi$ on $[0,M)$,
 $\psi^h (\cdot,t) \xi$ defines a non-negative Radon measure on $[0, M)$.

\begin{prop}
\label{prop-pde1} Given $B, M <\infty$,    $h \in {\cal C} [0,M)$, $\fmeas_0 \in
\mmb$ and $\newfk \in \incspace$, suppose $\psi^h$ is
the function defined in (\ref{def-psih}).
If $\fmeas \in {\cal D}_{\mmb}[0,\infty)$ satisfies (\ref{eq-pde}) for
all $\newf \in \newocdcpm$ then the measure-valued function
$\newfmeas \in {\cal D}_{\mb}[0,\infty)$ defined by
\be
\label{corr-numu}
\newfmeas_t (A) \doteq \int_{A \cap [0,M)}
\psi^h (x,t) \, \fmeas_t (dx)   \quad \quad \forall \mbox{  Borel } A \subset \R_+ \mbox{ and } t \in [0,\infty)
\ee
satisfies
equation (\ref{eq-pde2}) for all $\tnewf \in \mnewocdcp$. 
Conversely, if $\newfmeas \in {\cal D}_{\mb}[0,\infty)$ satisfies equation (\ref{eq-pde2})
for all $\tnewf \in \mnewocdcp$, then
the measure-valued function $\fmeas \in {\cal D}_{\mmb}[0,\infty)$ defined by
\be
\label{corr-munu}
\fmeas_t (A) \doteq \int_{A}
\dfrac{1}{\psi^h (x,t)} \, \newfmeas_t (dx)   \quad \quad \forall \mbox{  Borel } A \subset [0,M) \mbox{ and } t \in [0,\infty)
\ee
satisfies (\ref{eq-pde}) for all $\newf \in \newocdcpm$.
\end{prop}
\begin{proof}
 Fix $B, M$, $h$ as in the statement of the proposition and let $\psi = \psi^h$.
Since $h$ is continuous,
it can be easily verified that $\psi$ lies in ${\cal C}_c^1[0,M)$
and satisfies the partial differential equation
\be
\label{psi-pde}
 \psi_t + \psi_x   = h \psi
\ee
on $[0,M) \times \R_+$.
Now, suppose $\fmeas \in {\cal D}_{\mmb}[0,\infty)$ satisfies (\ref{eq-pde}) for
all $\newf \in \newocdcpm$.  Choose $\tnewf \in \mnewocdcp$ and define $\newf \doteq \tnewf  \psi$  on $[0,M) \times \R_+$.  Then
$\newf \in  \newocdcpm$, and so substituting the test function $\newf$ into
(\ref{eq-pde}), we obtain  
\[
\ba{rcl}
\ds \lan \tnewf(\cdot,t) \psi (\cdot, t), \fmeas_t \ran
& = & \ds  \lan \tnewf (\cdot,0), \fmeas_0 \ran + \int_0^t \lan \dttnewfs \psi(\cdot,s), \fmeas_s \ran \, ds \\
& & \quad \ds + \int_0^t \lan \dxtnewfs \psi(\cdot,s), \fmeas_s \ran \, ds  + \int_0^t  \tnewf(0,s) \psi(0,s),  d\newfk(s) \\
& & \ds + \int_0^t \lan \tnewfs\left[ \psi_s(\cdot,s) +
\psi_x(\cdot,s) - h(\cdot) \psi(\cdot,s) \right], \fmeas_s \ran \,
ds \ea
\]
Due to (\ref{psi-pde}), the last term in the above equation is identically zero.
Substituting the relation  (\ref{corr-numu}) into the above equation,
and noting that $\fmeas_0 = \newfmeas_0$ and
$\psi (0,\cdot) = \f1$,
we conclude that (\ref{eq-pde2}) holds.
Since $\tnewf \in \mnewocdcp$ is arbitrary, this
proves the first assertion.

The converse can be deduced in a similar manner and so only an outline of the proof is provided.
Suppose (\ref{eq-pde2}) is satisfied for all $\tnewf \in \mnewocdcp$,
fix  $\newf \in \newocdcpm$, and let  $\hat{\newf}$ and $\hat{\psi}^{-1}$ denote the extensions
of $\newf$ and $\hat{\psi}^{-1}$, respectively, to $\mnewocdcp$ obtained by setting the functions
 equal to zero for $x \geq M$.
Then, using (\ref{eq-pde2}) to study the action of $\newfmeas$ on the test function
$\tnewf = \psi^{-1} \hat{\newf} \in \mnewocdcp$,
the fact that $\psi^{-1}$ satisfies the PDE (\ref{psi-pde}) on $[0,M) \times \R_+$,
 but with $h$ replaced by $-h$,
and the relation (\ref{corr-munu}), one can  conclude that (\ref{eq-pde}) is satisfied
for all $\newf \in \newocdcpm$.
\end{proof}

We are now in a position  to wrap up the proof of Theorem \ref{th-pde}. \\

\noi
{\em Proof of Theorem \ref{th-pde}. }
If $\fmeas \in {\cal D}_{\mmb} [0,\infty)$
satisfies the equation (\ref{eq-pde}) for all $\newf \in \newocdcpm$, then
Proposition \ref{prop-pde1} shows that the measure-valued function
$\newfmeas \in {\cal D}_{\mb} [0,\infty)$ defined by (\ref{corr-numu})
 satisfies (\ref{eq-pde2}) for all $\tnewf \in \mnewocdcp$.
In turn,  Proposition \ref{prop-pde2} shows that then
$\newfmeas$ satisfies the explicit relation
(\ref{eq-newfmeas}) for all $\tnewf \in \mnewocdcp$.
Since $\fmeas_0 = \newfmeas_0$ and
$\fmeas$ can be recovered from $\newfmeas$ using (\ref{corr-munu}),
 it follows that
 there exists a unique solution $\fmeas$ to (\ref{eq-pde}) that satisfies,
for every $\newf \in \newccpm$ and $t \in [0,\infty)$,
\[
\ba{rcl}
\ds \lan \newft, \fmeas_t^M \ran   =     \lan \newft (1/ \psi^h(\cdot,t)), \newfmeas_t \ran
 & = & \ds \int_{[0,\infty)} \newf(x+t,t) \left( 1/\psi^h (x+t,t)\right) \fmeas_0 (dx) \\
&& \quad \ds +
\int_0^t \newf (t-s,s) (1/\psi^h (t-s,s))  \, d \newfk (s).
\ea
\]
When $h$ is the hazard rate defined in (\ref{def-h}),
elementary calculations show that
\[ \dfrac{1}{\psi^h(x,t)} =
\left\{
\ba{rl}
\ds \dfrac{1-G(x)}{1-G(x-t)} & \ds \mbox{ if } 0\leq t \leq x < M,  \\
\ds  1 - G(x)    & \ds \mbox{ if } 0\leq  x \leq t \wedge M.
\ea
\right.
\]
Substituting the last relation and $\newf (x,t) = f(x)$ for some
$f \in \ccm$ into the previous
equation, we obtain (\ref{eq2-fmeas}) for $f \in \ccm$.
Since $(1 - G(\cdot + t))(1 - G(\cdot))$ and $(1 - G(\cdot))$ are
bounded,
the relation (\ref{eq2-fmeas}) can then be established for $f \in \cb$ by
approximating $f$ by a sequence $f_n \in \ccm$, using the fact that
$f_n$ satisfies (\ref{eq2-fmeas}), then taking limits as $n \ra \infty$ and
 invoking the bounded convergence theorem to justify the interchange
of limits and integration.
\qed \\

\begin{remark}
\label{rem-gencont} {\em  The only place where the continuity of
the hazard rate $h$ is used to prove continuity of the fluid solution map is in the
representation given in Theorem \ref{th-pde}. In turn, continuity
of $h$ is used in the proof of Theorem \ref{th-pde} only in
Proposition \ref{prop-pde1}. However,  Proposition \ref{prop-pde1}
can be shown to hold for any $h \in \lloc [0,M)$ for which
(\ref{cond-radon}) holds.  The proof of the generalization
involves some rather technical PDE approximation arguments that
would involve the introduction of some more notation. In the
interest of conciseness and clarity of exposition, we omit this
generalisation from this paper.  }
\end{remark}

\beginsec

\section{Functional Law of Large Numbers Limit}
\label{sec-tight}

The main objective of this section is to show that, under
suitable assumptions, the sequence
$\{(\fxn,\fmeasn)\}$ converges to a process that solves the fluid equations.
In particular, this establishes existence of solutions to the fluid equations.
First, in Section \ref{subs-prelimit} we provide a useful description of the
evolution of the state $(\fxn, \fmeasn)$ of the $N$-server model.
Then, in Section \ref{subs-prelim}, we introduce a family of martingales that
are used in Section \ref{subs-tight} to establish tightness of the
sequence $\{(\measn, \xn), N \in \N\}$.
Finally, in Section \ref{subs-pf3}, we provide the
proof of Theorem \ref{th-conv}.

\subsection{A Characterisation of the Pre-limit Processes}
\label{subs-prelimit}

The dynamics of the $N$-server model was described in Section \ref{subs-modyn} and
certain auxiliary
processes were introduced in Section \ref{subs-aux}.
In this section, we provide
a more succinct and convenient description of the state dynamics,
which takes a form similar to that of the fluid equations.

Throughout the section, suppose  $\ren$ and initial conditions
$\xn (0) \in \R_+$ and $\measn_0 \in \mmsp$ are given,
and let $\en$, $\xn$ and $\measn$ be the associated
state processes, as described in Section \ref{subs-modyn}. For any
measurable function $\newf$ on $[0,M)\times\R_+$, consider the
sequence of processes $\{\bare^{(N)}_\newf, N \in \N\}$ taking
values in $\R_+$, given by
 \be
\label{def-baren}
   \bare^{(N)}_\newf (t) \doteq  \sum_{s \in [0,t]} \sum_{j=-\lan \f1, \measn_0 \ran  + 1}^{\kn (t)}
\ind_{\{d\agen_j (s-)/dt  >0, \agen_j (s) = v_j \}}
\newf(\agen_j (s),s), \ee where $\kn$ and $\agen_j$ are defined by
the relations (\ref{def-agejn}) and  (\ref{def-kn}).
Alternatively, we could  replace the set $\{d\agen_j (s-)/dt  >0,
\agen_j (s) = v_j\}$ in the indicator function in
(\ref{def-baren}) by the set $\{d\agen_j (s-)/dt > 0,
d\agen_j (s+)/dt = 0\}$, which immediately shows that
$\bare^{(N)}_\newf$ is ${\cal F}_t^{(N)}$-adapted. Moreover, when
$\newf$ is bounded, it is easy to see that for each $t \in
[0,\infty)$, $\E[\bare^{(N)}_\newf (t)] < \infty$  if $\E[ \xn(0)
+ \en (t)] < \infty$. Indeed,   for any $t \in [0,\infty)$,
\[ \E[|\bare^{(N)}_\newf (t)|] \leq
\nrm{\newf}_\infty  \E[\xn (0)  + \kn (t)] \leq \nrm{\newf}_\infty \E[\xn (0) + \en (t)].
\]
Also, from the relations (\ref{def-kn})--(\ref{def-nun}), it is easy to
see that $\bare^{(N)}_{\f1}$ is equal to the cumulative departure process $\dn$
defined in (\ref{def-dn}).

We now state the main result of this section.   Recall that for $s \in [0,\infty)$, $\nun_s$ represents $\nun(s)$ and
$\lan \newf(\cdot +h,s), \nun_s \ran$ is used as a short form for
$\int_{[0, M)} \newf(x+h,s) \, \nun_s (dx)$.

\begin{theorem}
\label{th-prelimit}
 The processes $(\en, \xn,
\measn)$ satisfy a.s.\ the following coupled set of equations: for
$\newf \in \newocdcpm$ and  $t \in [0,\infty)$,
\begin{eqnarray}
\label{eqn-prelimit1}
 \left\lan \newft, \nun_{t} \right\ran
& = &  \left\lan \newf(\cdot, 0), \nun_{0} \right\ran   +  \ds \int_{0}^t \left\lan \dxnewfs + \dtnewfs, \nun_s \right\ran \ ds  \\
\nonumber
& & \quad \quad \ds - \bare^{(N)}_\newf (t) + \int_0^t \newf(0,u) d\kn (u) ,  \\
\label{eqn-prelimit2}
\quad \quad \quad  \xn (t) & = &  \xn(0) + \en (t) - \bare^{(N)}_{\f1} (t) , \\
\label{comp-prelimit}
& & N - \left\lan 1, \measn_t \right\ran  = [N - \xn (t) ]^+,
\end{eqnarray}
where $\kn$ satisfies (\ref{def-kn}) and  $\bare^{(N)}_\newf$  is the process
defined in (\ref{def-baren}).
\end{theorem}

The rest of this section is devoted to the proof of this theorem.
We start with the simple observation that for any $\newf \in
\newocdcpm$, due to the right-continuity of $\measn$ we have   for
any $t \in [0,\infty)$,
 \be \label{measn-1}
\ba{l}
\ds \left \lan \newft, \measn_{t} \right\ran - \left\lan \newf(\cdot,0), \measn_0 \right\ran \\
\quad \quad \quad \quad  \ds =  \ds  \lim_{ \small \ra 0} \sum\limits_{k=0}^{\lfloor t/\small \rfloor} \left[
 \left\lan \newf (\cdot,(k+1) \small), \measn_{(k+1) \small} \right\ran -
\left\lan \newf,(\cdot,k \small),  \measn_{k \small} \right\ran
\right]. \ea \ee In order to compute the increments on the
right-hand side of the last equation,
 we observe that for $\newf \in \newocdcpm$,
 $\small > 0$ and $s \in [0,\infty)$, we can write
\rear{1.8}
 \be
\label{eqn-new}
\ba{l} \ds \left\lan \newf(\cdot,s+\small),
\nun_{s+\small} \right\ran - \lan \newf(\cdot,s), \nun_s \ran \\
\quad \quad = \ds \lan \newf(\cdot,s+\small) - \newf (\cdot,s),
\nun_{s+\small} \ran + \lan \newfs, \nun_{s+\small} \ran - \lan
\newfs, \nun_s \ran.
\ea
\ee
 We will treat the first term  and the
last two terms on the right-hand side of (\ref{eqn-new}) separately.
Summing the first term, with $s =k/n$ and $\delta = 1/n$,
over $k = 0, \ldots, \lfloor nt \rfloor$, we obtain
\be
\label{term1}
\ba{l}
\ds \suli_{k=0}^{\lfloor nt \rfloor} \left \lan \newf \left(\cdot,\frac{(k+1)}{n} \right)
- \newf \left(\cdot, \frac{k}{n} \right), \measn_{(k+1)/n} \right \ran  \\
\quad \quad \quad  \quad \quad \quad =  \ds \suli_{k=0}^{\lfloor
nt \rfloor} \dfrac{1}{n} \left\lan\dfrac{ \newf
\left(\cdot,\frac{(k+1)}{n} \right) - \newf \left(\cdot,
\frac{k}{n} \right)}{1/n}, \measn_{(k+1)/n} \right \ran. \ea \ee
Since $\newf \in \newocdcpm$, we have
\[ \sup_{k = 0, \ldots, \lfloor nt \rfloor} \left| \dfrac{\newf \left(\cdot,\frac{(k+1)}{n} \right)
- \newf \left(\cdot, \frac{k}{n} \right)}{1/n} - \newf_t \left(\cdot, \frac{k+1}{n} \right) \right|
=  O \left(\frac{1}{n}\right).
\]
Since $\measn$ has total mass no greater than $N$, and the mapping
$s  \mapsto \lan \newf_t (\cdot, s), \measn_s \ran$  is right-continuous,
 taking limits as $n \ra \infty$ in (\ref{term1}),
we obtain the corrresponding Riemann integral:
\be
\label{lim-term1}
 \lim_{n \ra \infty} \sum_{k=0}^{\lfloor nt \rfloor} \left \lan \newf \left(\cdot,\frac{(k+1)}{n} \right)
- \newf \left(\cdot, \frac{k}{n} \right), \measn_{(k+1)/n} \right \ran
= \int_0^t \lan \newf_t (\cdot, s), \measn_s \ran \, ds.
\ee

In order to simplify the last two terms on the right-hand side of (\ref{eqn-new}), for
$\newf \in \newocdcpm$,  $\small \in (0, M)$ and $s \in [0,\infty)$,
first observe that we can  write \rear{1.8} \be \label{eqn-0} \lan
\newf(\cdot,s), \nun_{s+\small} \ran = {\cal I}_1 + {\cal
 I}_2
\ee
where
 \[
{\cal I}_1  \doteq \ds \int_{[\small, M)} \newf(x,s) \,
\nun_{s+\small } (dx) \quad \quad \mbox{ and } \quad \quad
 \cI_2 \doteq   \ds  \int_{[0,\small)} \newf(x,s) \, \nun_{s+\small} (dx).
\]
We begin by rewriting ${\cal I}_1$ in terms of quantities that are known at time $s$.
For $x \geq \delta$,  customers in service with age equal to $x$
at time $s+\small$ are precisely those
 customers that were already in service at time $s$ with age equal to
 $x-\small \geq 0$ and that, in addition, did not depart the system in the interval $[s,s+\small ]$.
Since the age of a customer already in service increases linearly
with rate $1$ (see (\ref{def-agejn})), using the representation for $\measn$ given in
(\ref{def-nun}), we have
\[
\ba{rcl} \ds \cI_1 & =  & \ds
\sum_{j = - \lan \f1, \measn_0 \ran + 1}^{\kn (s+ \small )}  \newf(\agen_j (s+ \small ),s) \ind_{\{\delta \leq \agen_j (s+ \small) < v_j\}}  \\
& = & \ds \suli_{j = - \lan \f1, \measn_0 \ran + 1}^{\kn (s)}  \newf(\agen_j (s) + \small,s ) \ind_{\{\agen_j (s)+ \small  < v_j\}} \\
& = &  \ds \suli_{j = -\lan \f1, \measn_0 \ran  + 1}^{\kn (s)}  \newf(\agen_j (s) + \small,s ) \ind_{\{\agen_j (s) < v_j\}} \\
& & - \ds  \suli_{j = - \lan \f1, \measn_0 \ran + 1}^{\kn (s)}
\newf(\agen_j (s) + \small,s ) \ind_{\{\agen_j (s) < v_j \leq
\agen_j(s) + \small  \}}. \ea
\]
(Here, and in what follows below, we always assume wlog that $\delta$ is
sufficiently small so that the range of the first argument of $\newf$ falls within
$[0,M)$, ensuring that all quantities  are well-defined.)
Substituting  the definition of $\measn_s$ into the last
expression,  this can be rewritten as
\be \label{eqn-i1}
\cI_1   =  \ds \left \lan \newf(\cdot+ \small,s ), \nun_{s}
\right\ran   - \are_{\newf}^{(N)}(s,\small),
 \ee
 where for
$\newf \in \mathcal{C}^1_c([0,M)\times\R_+)$, $s \in [0,\infty)$
and $\small> 0$, we define
\be \label{def-mf} \are^{(N)}_{\newf}
(s,\small) \doteq \suli_{j = - \lan \f1, \measn_0 \ran + 1}^{\kn
(s)}
\newf(\agen_j (s) + \small,s)
 \ind_{\{\agen_j (s) < v_j \leq \agen_j(s) + \small \}}.
\ee

 We now expand ${\cal I}_2$. Since $\newf \in
\mathcal{C}^1_c([0,M)\times\R_+)$,  for every $x \in (0,\small)$
and $s\in[0,T]$ there exists $y_{x,s} \in [0,x] \subset
[0,\delta]$ such that $\newf(x,s) - \newf(0,s) = x
 \newf_x (y_{x,s},s)$, where $\newf_x$ is the partial derivative
 of $\newf$ with respect to $x$.  As a result, we can rewrite ${\cal I}_2$ as
\[
\ba{rcl} \ds {\cal I}_2 &  =  & \ds  \int_{[0,\small)} \newf(0,s)
\, \measn_{s+\small} (dx) +
 \int_{[0,\small)} (\newf(x,s) - \newf(0,s)) \, \measn_{s+\small} (dx) \\
& = & \ds
  \newf(0,s) \measn_{s+ \small} [0,\small)  + \int_{[0,\small)} x \newf_x (y_{x,s},s) \, \measn_{s + \small} (dx).
\ea
\]
For any $0 \leq r \leq s < \infty$, let $\kn (r,s)$ denote the number of customers that entered
service in the period $(r,s]$ and let
$\newdn (r, s)$ denote the number of customers that  both
entered service and departed the system in the period $(r,s]$.
Note that $\newdn (r,s)$ admits the  explicit representation
\be
\label{eq-newdn}
 \newdn (r,s) =
\suli_{j = \kn (r) + 1}^{\kn (s)}
\ind_{\{\agen_j (s) = v_j\}}. 
\ee
Also, note that $\measn_{s+ \small} [0,\small)$ is the number of customers in service
that have age less than $\delta$ at time $s + \small$.  These customers
must therefore have entered service in the interval $(s,s+\small]$ and not
yet departed by time $s + \small$.  Therefore, we can write
\[ \measn_{s+ \small} [0,\small) =
\kn (s,s+\small)  - \newdn (s,s+\small). \]
Combining the last three expressions, we obtain
\be \label{eqn-i2}
\ba{rcl}
{\cal I}_2 &  = &   \ds \newf(0,s) \kn (s,s+\small)
-
 \newf(0,s) \newdn (s, s+\small) \\
& & \quad \quad \ds  + \int_{[0,\small)} x \newf_x (y_{x,s},s) \,
\measn_{s + \small} (dx).
\ea
 \ee

Substituting  equations (\ref{eqn-i1}) and  (\ref{eqn-i2})   into
(\ref{eqn-0}),  with $s = k/n$ and $\delta = 1/n$, we obtain for
$\newf \in {\cal C}^1([0,M)\times\R_+)$, \rear{1.8}
\[
\ba{l}
\ds \left\lan \newf\left(\cdot, \frac{k}{n} \right), \measn_{(k+1)/n} \right\ran -
\left\lan \newf \left(\cdot, \frac{k}{n} \right), \measn_{k/n} \right\ran \\
\quad \quad \quad
 =   \ds \left\lan \newf\left(\cdot + \dfrac{1}{n},\dfrac{k}{n}\right) -
 \newf\left(\cdot,\frac{k}{n}\right), \measn_{k/n}\right\ran
- \are_\newf \left(\frac{k}{n}, \frac{1}{n}\right) \\
 \quad \quad \quad \quad \quad \ds
+ \newf\left(0,\frac{k}{n}\right)  \kn \left(\frac{k}{n},
 \frac{k+1}{n} \right)  - \newf\left(0,\frac{k}{n}\right) \newdn \left(\frac{k}{n},\frac{k+1}{n} \right) \\
\quad \quad \quad \quad \quad \ds
 +  \int_{\left[0,\frac{1}{n}\right)} x  \newf_x\left(y_{x,\frac{k}{n}},\frac{k}{n}\right)
 \, \measn_{(k+1)/n} (dx).
\ea
\]
Fix $t \in [0,\infty)$ and define $t^n \doteq (\lfloor nt \rfloor
+ 1)/n$.
Summing the last expression over $k = 1, \ldots, \lfloor n t \rfloor$, we obtain
\be
\label{eqn-i123}
\ba{l}
\sum\limits_{k=0}^{\lfloor n t
\rfloor} \left[ \left\lan \newf\left(\cdot, \frac{k}{n}\right), \measn_{(k+1)/n}
\right \ran -
\left\lan \newf\left(\cdot,\frac{k}{n}\right), \measn_{k/n} \right \ran \right] \\
\quad \quad \quad \quad \quad =
 \ds \sum\limits_{k=0}^{\lfloor nt
\rfloor} \dfrac{1}{n} \left\lan \dfrac{\newf(\cdot +
1/n,\frac{k}{n}) -\newf\left(\cdot,\frac{k}{n}\right)}{1/n}, \measn_{k/n} \right \ran -
\bare^{(N),1/n}_{\newf}(t)
\\ \quad \quad \quad \quad \quad \quad \quad \ds + \sum_{k=0}^{\lfloor n t \rfloor}\newf\left(0,\frac{k}{n}\right) \kn \left(\frac{k}
{n},\frac{k+1}{n}\right) -\remn_{\newf }(t), \ea \ee where, for
conciseness, for  $\newf \in \newocdcpm$, $N \in \N$ and $t \in
[0,\infty)$, we set \be \label{def-barensmall}
 \bare_{\newf}^{(N), \small} (t) \doteq
\sum_{k=0}^{\lfloor  t/\small \rfloor} \are_{\newf}^{(N)} \left( k
\small,\small \right) \quad \quad \mbox{ for } \small > 0
 \ee
and, for $n \in \N$,
 \be \label{def-remn} \ba{rcl}
 \remn_{\newf}(t,n) & \doteq & \ds \ds   \sum_{k=0}^{\lfloor n t \rfloor}\newf \left(0,\frac{k}{n} \right)
\newdn \left(\frac{k}{n},\frac{k+1}{n}\right) \\
& & \ds - \sum_{k=0}^{\lfloor n t \rfloor} \int_{[0,1/n)} x
\newf_x \left(y_{x,\frac{k}{n}},\frac{k}{n}\right) \, \measn_{(k+1)/n} (dx).
\ea
\ee

Since $\newf \in \newocdcpm$, $\measn [0,M) \leq N$,  and both
the mapping $s \mapsto \lan \newf_x (\cdot, s), \measn_s \ran$ and
the process $\kn$ have right-continuous paths,  it is clear that
the first and third sums on the right-hand side of
(\ref{eqn-i123}) converge pathwise,
 as $n \ra \infty$, to the corresponding Riemann-Stieltjes integrals:
\be
\label{lim-term2}
\lim_{n \ra \infty}
\sum\limits_{k=0}^{\lfloor nt
\rfloor} \dfrac{1}{n} \left\lan \dfrac{\newf(\cdot +
1/n,\frac{k}{n}) -\newf\left(\cdot,\frac{k}{n}\right)}{1/n}, \measn_{k/n} \right \ran
= \int_0^t \left \lan \newf_x (\cdot, s ), \measn_s \right \ran \, ds
\ee
and, likewise,
\be
\label{lim-term3}
\lim_{n \ra \infty}
\sum_{k=0}^{\lfloor n t \rfloor}\newf\left(0,\frac{k}{n}\right) \kn \left(\frac{k}
{n},\frac{k+1}{n}\right) = \int_{[0,t]} \newf (0, s) \, dK^{(N)} (s).
\ee
The next two results identify the limits of the remaining terms,
$\bare_{\newf}^{(N),1/n}$ and $\remn$, as  $n \ra \infty$.
Recall the definition of the process
 $\bare^{(N)}_{\newf}$ given in (\ref{def-baren}).

\begin{lemma}
\label{lem-baren} For every $N \in \N$, $t \in [0,\infty)$ and
$\newf \in \newocdcpm$, $ \bare_{\newf}^{(N),\small}(t)$ converges
a.s.\ to $\bare^{(N)}_{\newf}(t)$ as $\delta \ra 0$.
\end{lemma}
\begin{proof}
Fix $N \in \N$, $t \in [0,\infty)$ and $\newf \in \newocdcpm$, and
let $L <
 \infty$ be such that $\sup_{s \in [0,t],y\in[0,M)} |\newf_x(y,s)+\newf_t(y,s)| \leq L$.
For any $\small > 0$, and $j = -\lan \f1, \measn_0 \ran + 1, \ldots, 0$, define
$\tau^\small (j) = 0$ and for $j = 1, 2, \ldots, $ define
\[ \tau^\small (j) \doteq \inf \{ k \in \N: \agen_j (k \small + \ve) > 0
\quad \forall \ve > 0\}. \]
 Observe that $\tau^\small (j) \delta$
represents the first time after the $j$th customer enters service
that lies on the $\small$-lattice $\{k\small, k = 0, 1, \ldots,
\lfloor t/\small \rfloor \}$ (the introduction of $\ve$ in the
definition was necessary to ensure that $\tau^\small (j) = k$ if
the $j$th customer arrives precisely at $k \small$, and thus has
age $0$ at that time). Then for any $\small > 0$, a simple
interchange of summation shows that
\[
\ba{rcl} \bare_{\newf}^{(N), \delta} (t) & = &
\suli_{k=0}^{\lfloor t/\small \rfloor} \suli_{j = - \lan \f1,
\measn_0 \ran + 1}^{\kn (k \small)} \newf(\agen_j (k \small) +
\small,k \small) \ind_{\{ \agen_j (k \small) < v_j
\leq \agen_j ( k \small) + \small  \}} \\
& = & \suli_{j = - \lan \f1, \measn_0 \ran + 1}^{\kn (\lfloor
t/\small \rfloor \small)} \suli_{k = \tau^\delta (j) }^{\lfloor
t/\small \rfloor}\newf( \agen_j (k \small) + \small, k \small)
\ind_{\{ \agen_j (k \small) < v_j \leq \agen_j ( k \small) +
\small  \}}.
\ea
\]
However, when $\agen_j (k \small) < v_j \leq \agen_j (k \small) +
 \small$, we have
\[ \sup_{s \in [k \small, (k+1) \small]} \left|\newf (\agen_j (k \small) + \small,k \small) - \newf(v_j,s) \right| \leq L \delta \]
and we also  know that there exists a (unique) $s \in (k\delta, (k+1) \delta]$ such that
$da_j^{(N)} (s-)/dt  > 0$ and $a_j^{(N)} (s) = v_j$ (i.e.,  $s$ is the  unique time at which the
customer departs the system).
Hence, we can write
\[ \bare_{\newf}^{(N), \delta} (t) =
\suli_{j = - \lan \f1, \measn_0 \ran + 1}^{\kn(\lfloor t/\small
\rfloor \small)}  \sum_{s \in [0, (\lfloor t/\small \rfloor + 1)
\small]}\newf(v_j,s) \ind_{\{ d\agen_j (s-)/dt  > 0,  \agen_j(s) = v_j\}}  + O(\small).
\]
 Sending $ \small \ra 0$, since $\kn$ is c\`{a}dl\`{a}g, we see
 that $\bare_{\newf}^{(N), \small}(t)$ converges  to the quantity
\[
 \suli_{j = - \lan \f1, \measn_0 \ran + 1}^{\kn( t-)}  \sum_{s \in
[0, t]} \newf(\agen_j(s),s) \ind_{\{  d\agen_j (s-)/dt  > 0,    \agen_j (s) = v_j\}} =
\bare_\newf^{(N)}(t).\]
 The last  equality follows by replacing $\kn
(t-)$  by $\kn (t)$, which is justified (even though we need not
have $\kn (t-) = \kn(t)$) because
  every $v_j$ is a.s. strictly positive (since $G(0+) = 0$),
and so if customer $j$ enters service precisely at time $t$, then
$\ind_{\{\agen_j(s) = v_j\}} = 0$ for every $s \in [0,t]$.
\end{proof}

We now  show that a.s.,
$\remn (t) = O(1/n)$ uniformly for $t$ in compact sets.

\begin{lemma}
\label{lem-prelim1}
Almost surely, for every $T \in [0,\infty)$
  and  every   $\newf \in \newocdcpm$,
 \be
\label{remn-limit} \lim_{n \ra \infty} \sup_{t \in [0,T]}
\remn_{\newf} (t) = 0. \ee
\end{lemma}
\begin{proof}
We will establish the lemma by showing that a.s. for any
$\newf\in\newocdcpm$ both   terms on the right-hand side of
(\ref{def-remn}) converge uniformly to zero as $n \ra \infty$. Fix
$T < \infty$ and $t \in [0,T]$. Then for any $n \in \N$,
 from the representation (\ref{eq-newdn}) for $\newdn$ it immediately follows that
\[
\ds \suli_{k=0}^{\lfloor nt \rfloor} \newdn \left(\frac{k}{n},\frac{k+1}{n}\right)  \leq  
\ds \suli_{k=0}^{\lfloor nt \rfloor} \suli_{j=\kn
(k/n) + 1}^{\kn ((k + 1)/n)} \ind_{\{ v_j \leq 1/n\}} 
  \leq  
\suli_{j=1}^{\kn (\lfloor nt \rfloor + 1)/n } \ind_{\{ v_j \leq
1/n\}}.  
\]
Since $\kn (T)$ is a.s.\ finite, by the dominated convergence theorem
this in turn implies that almost surely,
\[\ba{l}
\ds \lim_{n \ra \infty}
\sup_{s \in [0,T]} \suli_{k=0}^{\lfloor ns
\rfloor} \newf\left(0, \dfrac{k}{n} \right) 
\newdn \left(\frac{k}{n},\frac{k+1}{n}\right) \\
\hspace{.6in}  \leq  \ds \nrm{\newf}_{\infty}\lim_{n \ra
\infty} \sum\limits_{j=1}^{\kn (T) + 1} \ind_{\{v_j \leq 1/n\}} 
\ds  = 
\ds \nrm{\newf}_{\infty}\sum\limits_{j=1}^{\kn (T) + 1} \ind_{\{v_j =
0\}} = 0,
\ea
\]
where the last equality follows a.s.\ since $G(0) = 0$.
The monotonicity  in $T$ of the convergent term allows us to conclude
that there  exists a set $\Omega_1$ of full $\P$-measure
 on which this convergence holds simultaneously for all $T$.

We now turn to  the second term on the right-hand side
of (\ref{def-remn}).
Let $C \doteq \sup_{y \in [0,M), s\in[0,T]} |\newf_x(y,s)| <
\infty$.
Then for any $t \in [0,T]$ and $n \in \N$,
\[ \sum_{k = 0}^{\lfloor n t \rfloor} \int_{[0,1/n)} |x \newf_x(y_{x,k/n},k/n) |
 \, \measn_{(k+1)/n} (dx)
\leq \dfrac{C}{n}  \sum_{k = 0}^{\lfloor n t \rfloor}
\measn_{(k+1)/n} [0,1/n).
\]
Now, any customer that has age in $[0,1/n)$ at time $(k+1)/n$ entered service
strictly after $k/n$ and has age $\geq 1/n$ at any time $k^\prime/n$,
$k + 1 < k^\prime \in \N$.  Hence for any fixed $n \in \N$,
the unit mass corresponding to any given customer
is counted in at most one term of the form $\measn_{(k+1)/n} [0,1/n)$,
$k \in \N$.  This implies the elementary bound
\[ \sup_{t \in [0,T]} \suli_{k=0}^{\lfloor nt \rfloor} \measn_{(k+1)/n} [0,1/n)
\leq  \xn (0) + \en (T+1).
\]
Now let $\Omega_2$ be the set of full $\P$-measure on which the property $\xn(0) + \en(t) < \infty$ for every
$t \in [0,\infty)$ is satisfied.
Then on $\Omega_2$,
 the right-hand side of the last expression, which is independent of $n$, is a.s.\ finite.
Therefore,
dividing both sides of the last relation by $n$,
the resulting expression clearly tends to $0$ a.s.\ as
 $n \ra \infty$.
Thus we have shown that
 on the set $\Omega_1 \cap \Omega_2$ of full $\P$-measure,
(\ref{remn-limit}) holds for every $T < \infty$ and
 $\newf\in\newocdcpm$
\end{proof}

We are now in a position to complete the proof of Theorem \ref{th-prelimit}. \\

\noi
{\em Proof of Theorem \ref{th-prelimit}.}
Fix $N \in \N$. The left-hand side of (\ref{measn-1}) is clearly
equal to
$\lan \newf, \measn_{t^n} \ran - \lan \newf, \measn_0\ran$,
where recall $t^n = (\lfloor nt \rfloor + 1)/n$.
 Since $\measn$ is right-continuous, as $n \ra \infty$ this
converges to $\lan \newf, \measn_{t} \ran - \lan \newf, \measn_0\ran$.
On the other hand,  combining  (\ref{eqn-new}), (\ref{lim-term1}),
(\ref{eqn-i123}), (\ref{lim-term2}), (\ref{lim-term3}) and
Lemmas \ref{lem-baren} and \ref{lem-prelim1}, it follows that the
right-hand side of (\ref{measn-1}) a.s.\ converges, as $n \ra \infty$, to
\[\ba{l}
\ds \int_0^t \lan \newf_t(\cdot,s)+\newf_x(\cdot,s), \measn_s \ran \, ds - \bare^{(N)}_\newf (t) +
\int_{[0,t]} \newf (0,s) \, d\kn (s),
\ea
\]
which  proves (\ref{eqn-prelimit1}).
The remaining relations (\ref{eqn-prelimit2}) and (\ref{comp-prelimit}) follow immediately
from (\ref{def-dn}), (\ref{def-idlen}), (\ref{def-nonidling}) and the observation that $\baren_{\f1} = \dn$
(see the comment below (\ref{def-baren})).
\qed

\subsection{A Useful Family of Martingales}
\label{subs-prelim}

An inspection  of the integral equation (\ref{eqn-prelimit1}) suggests that
identification of the limit of the sequence 
$\{(\fxn,\fmeasn)\}$  of scaled state processes
is likely to require a characterisation of the
limit of a scaled version  $\fqn_{\newf}$ of $\qn_{\newf}$.
In order to achieve this task, we first identify the compensator of
$\qn_{\newf}$
 (in  Corollary \ref{cor-compensator}) and then
identify the limit of the quadratic variation of the associated
scaled martingale
 $\fmartn_{\newf}$,
obtained as a compensated sum of jumps (see Lemma \ref{lem-mgale}).

We begin by introducing some notation. For any $\newf \in
\newcbpm$,
 consider the sequences of processes
$\{\dcompn_{\newf}, N \in \N\}$
 defined by
\be
\label{def-dcompn}
\dcompn_{\newf} (t) \doteq \int_0^t \left(
\int_{[0, M)} \newf(x,s) h(x) \, \measn_s (dx) \right) \, ds
\ee
for every $t \in [0,\infty)$.
We now derive an alternative representation for the process $\dcompn_\newf$,
 which shows, in particular, that
$\dcompn_\newf(t)$ is well-defined and takes values in $\R_+$ for
every $t \in [0,\infty)$. For $j \in \N$, let $\kninv_j \doteq
\inv[\kn] (j)$ be the time that the $j$th customer entered service
(recall the definition of $\inv$ given in (\ref{def-finv2})).
Then, interchanging the order of integration and summation
 and using the linear increase of  the age process,
for $t \in [0,\infty)$ we can write
 \be \label{altrep-dcompn}
\ba{rcl} \ds  \dcompn_\newf (t) & = &
  \ds \int_0^t \left( \suli_{j=-\lan \f1, \measn_0 \ran + 1}^{\kn (s)}
h \left(\agen_j (s) \right) \newf \left( \agen_j (s), s \right) \ind_{\{ \agen_j (s) < v_j \}} \right) \, ds \\
 & = & \ds    \suli_{j = -\lan \f1, \measn_0 \ran + 1}^{0} \int_0^t h \left(\agen_j (0) + s \right)
\newf \left( \agen_j (0) + s, s \right) \ind_{\left\{\agen_j (0) + s < v_j \right\} } \, ds  \\
& & \ds + \sum_{j=1}^{\kn (t)} \int_{\kninv_j}^t h\left(s -
\kninv_j \right)
\newf \left( s - \kninv_j, s \right) \ind_{\left\{ s < \kninv_j + v_j \right\}} \, ds.
\end{array}
\ee
Since $v_j < M$ a.s.\ and $h$ is locally integrable on $[0, M)$, this shows
that $\fdcompn$ is well-defined.
 Now, let $J^{(N)}_t$ be the
(random) set of jump points of the departure process $\dn$ upto
time $t$:
\[ J^{(N)}_t \doteq  \left\{s \in [0,t]: \dn(s) \neq \dn (s-) \right\} \]
and  set $J^{(N)} = J^{(N)}_\infty \doteq \cup_{t > 0} J^{(N)}_t$.
Recall that $\qn_{\f1} = \dn$.  We start by identifying the compensator of
$\dn$.

\begin{lemma}
\label{lem-compensator}
For every $N \in \N$, the process $\dcompn_{\f1}$ is the
${\cal F}_t^{(N)}$-compensator of the departure process $\dn$.
In other words,
$\dcompn_{\f1}$ is an increasing, ${\cal F}_t^{(N)}$-adapted
process, with $\E [\dcompn_{\f1} (t)] < \infty$ for every $t \in [0,\infty)$,
such that for every non-negative ${\cal F}_t^{(N)}$-predictable process $Z$,
\be
\label{to-prove}
 \E \left[ \sum\limits_{s \in J^{(N)}} Z_s \right]
= \E \left[ \int_0^\infty  Z_s \left( \int_{[0,M)} h(x) \, \measn_s (dx)  \right) \, ds
\right].
\ee
\end{lemma}
\begin{proof}
Fix $N \in \N$ and label the servers from $1, \ldots, N$.
In order to prove the lemma, we shall find it convenient to introduce the following notation.
For $k = 1, \ldots, N$ and $n \in \N$, let $\ftime_n^{(N),k}$ (respectively, $\stime_n^{(N),k}$) be the time at which
the $n$th customer to arrive at station $k$ starts (respectively, completes) service.
We also let $D^{(N),k} (t)$ represent the total number of customers that have departed
the $k$th station in the interval $[0,t]$, and let $S^{(N),k}(t)$ be the number of customers that are currently
being served at station $k$ at time $t$ (note that for every $t \in [0,\infty)$,  $S^{(N),k} (t)$ takes the value $0$ or the value $1$).
Then clearly
\[ \dn  =  \sum_{k=1}^N D^{(N),k} \quad \quad \mbox{ and } \quad \quad \lan \fmeasn,1 \ran = \sum_{k=1}^N S^{(N),k}. \]
For conciseness, for the rest of this proof we shall omit the explicit dependence of all quantities on $N$.

For $k = 1, \ldots, N$, the process $D^k = D^{(N),k}$ admits the decomposition
\[ D^k (t)  = \sum_{n=1}^\infty \left[ D^k \left(t \wedge \stime_n^k  \right)
-  D^k \left(t \wedge \ftime_n^k  \right) \right].
\]
If we define
\[ D_n^k \doteq  D^k \left(t \wedge \stime_n^k  \right)
-  D^k \left(t \wedge \ftime_n^k  \right) \]
to be the one-point process of the $n$th departure from station $k$, then
we assert that the  ${\cal F}_t$-compensator of $D_n^k$ is given by the process $A_n^k$,
defined for $t \in [0,\infty)$ by
\[ A_n^k (t) \doteq
\left\{
\begin{array}{rl}
\ds 0 & \mbox{ if } t  \in [0, \ftime_n^k] \\
 \ds \int_{\ftime_n^k}^t h \left(u - \ftime_n^k \right) \, du &
\mbox{ if } t \in  (\ftime_n^k, \stime_n^k] \\[1.2em]
\ds  \int_{\ftime_n^k}^{\stime_n^k}  h \left(u - \ftime_n^k \right) \, du & \mbox{ if } t \in  (\stime_n^k, \infty).
\end{array}
\right.
\]
 It is straightforward to verify that  $\ftime_n^k$ and $\stime_n^k$ are both ${\cal F}_t$-stopping times (this can be done by rewriting the events
$\{\ftime_n^k  \leq t\}$ and $\{\stime_n^k \leq t \}$ as events involving
 $\{\agen_j (s), s_j^{(N)} (s), s \in [0,t], j \in \{-N+1, \ldots, 0\} \cup \N\}$ --- the details are left to the reader).
As a result, it follows that $A_n^k$ is ${\cal F}_t$-adapted.
Moreover, by definition $A_n^k$ is  continuous, and hence ${\cal
F}_t$-predictable. Thus in order to establish the claim that it is
the $({\cal F}_t)$-compensator for $D_n^k$,  it only remains to
show that for every non negative${\cal F}_t$-predictable process
$\{Z_t\}$,
\[
\E \left[ \int_0^\infty Z_s \, d D_n^{(N),k} (s) \right] = \E \left[\int_0^\infty Z_s \, d A_n^k (s) \right].
\]
As is well known, the predictable filtration is generated by indicators of stochastic intervals, i.e., by processes of the
form $Z = \ind_{[0,T]}$, where $T$ is any ${\cal F}_t$-stopping time
\cite[Theorem 2.2]{jacshibook}.
Hence it suffices to show that for every ${\cal F}_t$-stopping time $T$,
\[  \E \left[ \int_0^\infty \ind_{[0,T]} (s) \, d D_n^{k} (s) \right] = \E \left[\int_0^\infty \ind_{[0,T]}(s) \, d A_n^k (s) \right].
\]
By an application of the monotone convergence theorem, it is clear that
 it in fact suffices to show
that the above holds true for any bounded stopping time.
Now, note that,
because neither $D^k$ nor $A^k$  increase outside $(\ftime_n^k, \stime_n^k]$,
 the last equation is  equivalent to the relation
\[
  \E \left[ \int_0^\infty \ind_{[0,T] \cap (\ftime_n^k, \infty)} (s) \, dD_n^{k} (s) \right]  =
 \E \left[\int_0^\infty \ind_{[0,T] \cap (\ftime_n^k, \stime_n^k]} (s) \, d A_n^k (s) \right].
\]
However, the term on the left-hand side can be rewritten as
\[  \E \left[ \int_0^\infty \ind_{[0,T] \cap (\ftime_n^k, \infty)} (s) \, dD^{k}_n (s) \right]  =
\lim_{m \ra \infty} \E \left[ \sum_{j=0}^\infty
\ind_{\left\{ \ftime_n^k \leq \frac{j}{2^m} < T, \frac{j}{2^m} < \stime_n^k \leq \frac{j+1}{2^m} \right\}}
\right]. \]
Since $T$, $\ftime_n^k$ and $\stime_n^k$ are all ${\cal F}_t$-stopping times, conditioning on ${\cal F}_{j/2^m}$
and using the fact that the age $t - (\ftime_n^k \wedge t)$ at time t of the $n$th customer to enter service
at station $k$  is ${\cal F}_{\ftime_n^k}$-measurable, it follows that for any $m \in \N$ and $j = 1, \ldots, 2^m$,
\rear{2.4}
\[
\ba{l}
\ds \E\left[ \ind_{\left\{\ftime_n^k \leq \frac{j}{2^m} < T\right\}}
\ind_{\left\{\frac{j}{2^m} < \stime_n^k \leq \frac{j+1}{2^m} \right\}} \right] \\
\quad \quad \quad = \ds \E \left[ \E \left[ \ind_{\left\{\ftime_n^k \leq \frac{j}{2^m} < T, \stime_n^k > \frac{j}{2^m}\right\}}
 \ind_{\left\{\stime_n^k \leq \frac{j+1}{2^m} \right\}} | {\cal F}_{\frac{j}{2^m}} \right] \right] \\
\quad \quad \quad = \ds \E \left[  \ind_{\left\{\ftime_n^k \leq \frac{j}{2^m} < T, \stime_n^k > \frac{j}{2^m}\right\}}
\int_{j/2^m}^{(j+1)/2^m}  \dfrac{ g(u - \ftime_k^n )}{1 - G(\frac{j}{2^m} - \ftime_n^k )} \, du \right] \\
\quad \quad \quad = \ds  \E \left[  \ind_{\left\{\ftime_n^k \leq \frac{j}{2^m} < T, \stime_n^k > \frac{j}{2^m}\right\}}
  \dfrac{ G\left( \frac{j+1}{2^m} - \ftime_n^k \right)  -
G\left(\frac{j}{2^m} - \ftime_k^n \right)}{1 - G\left(\frac{j}{2^m} - \ftime_k^n \right)} \right].
\ea
\]
Combining the last two displays and invoking the monotone convergence theorem, we conclude that
\[
\ba{l}
 \ds \E \left[ \int_0^\infty \ind_{[0,T] \cap (\ftime_n^k, \infty)} (s) \, dD^{k}_n (s) \right]  \\
 \quad \quad \ds = \ds \lim_{m  \ra \infty}\E \left[ \sum_{j=0}^\infty \ind_{\left\{\ftime_n^k \leq \frac{j}{2^m} < T,
\stime_n^k > \frac{j}{2^m}\right\}}
\dfrac{ G\left( \frac{j+1}{2^m} - \ftime_k^n \right)  - G\left(\frac{j}{2^m} - \ftime_k^n \right)}{1 -
G\left(\frac{j}{2^m} - \ftime_n^k \right)}  \right].
\ea
\]

To complete the proof, it only remains to show that the right-hand side of the
last equation
is equal to $\E [ \int_0^\infty \ind_{[0,T]} (s) \, d A_n^k (s) ]$.   For this, first  note that the term within the expectation on the right-hand side of the
last equation  can be rewritten in the
form $\E [ \int_0^\infty \ind_{[0,T]} (s) dA_{m,n}^k(s) ]$ where,
for $m \in \N$, $A_{m,n}^k$ is the random measure defined by
 \[ A_{m,n}^{k} (\omega) = \sum_{j=0}^\infty
\delta_{\frac{j+1}{2^m}}
\ind_{\{\ftime_n^k (\omega) \leq \frac{j}{2^m} < \stime_n^k(\omega)\}}
\dfrac{ G \left( \frac{j+1}{2^m} - \ftime_n^k (\omega) \right) -
G \left( \frac{j}{2^m} - \ftime_n^k (\omega) \right)}{1 - G \left( \frac{j}{2^m} - \ftime_n^k (\omega) \right)}
\]
where $\delta_x$ is, as usual, the Dirac mass at $x$.
Next, observe that
\[\ba{rcl}
\ds \int_{\ftime_n^k + \frac{1}{2^m}}^{\stime_n^k} \dfrac{g(u - \ftime_n^k)}{1 - G(u - \ftime_n^k - \frac{1}{2^m})} \, du & \leq &
 \ds  A_{m,n}^{k}[0,\infty) \\
& \leq &
\ds \int_{\ftime_n^k}^{\stime_n^k + \frac{1}{2^m}} \dfrac{g(u - \ftime_n^k)}{1 - G(u - \ftime_n^k)} \, du \\
  & = & \ds - \ln \left( 1 - G \left(\stime_n^k + \frac{1}{2^m} - \ftime_n^k\right) \right).
\ea
\]
Since $\stime_n^k - \ftime_n^k$ represents a service time, it is distributed according to
$G$.  Hence $G(\stime_n^k - \ftime_n^k)$ is uniformly distributed in $(0,1)$.  Since
$G$ is continuous, for every $\omega$,
 there exists a sufficiently large $m_0 = m_0(\omega)$ such that
for $m \geq m_0$,
$G(\stime_n^k (\omega) - \ftime_n^k (\omega) + 1/2^m) < 1$, so that
$-\ln \left(1 - G \left(\stime_n^k (\omega) + \frac{1}{2^m} - \ftime_n^k (\omega) \right) \right)< \infty$.
  Combining the last three statements, we conclude that for each $\omega$,
$A_{m,n}^{k} (\omega)[0,\infty)$ is finite for sufficiently
large $m$ and, moreover, that
 as $m \ra \infty$, $A_{m,n}^k (\omega)$ converges vaguely
to the measure that has density
\[ \dfrac{g(u-\ftime_n^k(\omega))}{1 - G(u-\ftime_n^k(\omega))} \ind_{\ftime_n^k (\omega) \leq u \leq \stime_n^k (\omega)\}} =
h(u-\ftime_n^k(\omega)) \ind_{\{\ftime_n^k (\omega) \leq u \leq \stime_n^k (\omega)\}},
\]
which is precisely the measure $dA_n^k (\omega)$.
The latter measure does not charge points, and in particular does not charge $u = T(\omega)$. So we conclude that for every $\omega$,
\[
 \lim_{m\ra \infty} \int_0^\infty \ind_{[0,T(\omega)]} dA_{m,n}^k (\omega)  =
\int_0^\infty \ind_{[0,T(\omega)]} dA_n^k(\omega)
\leq \int_0^L h(u) \, du,
\]
where $L$ is an upper bound on the stopping time $T$ and the
 last term is finite due to the local integrability of
 the hazard rate function on $[0,M)$
(note that if $M<\infty$ we may restrict our attention to stopping times
$T<\theta^k_n+M$).
The limit  above, along with the bounded convergence
theorem, then implies the desired convergence:
\[
 \ba{l}
\ds \lim_{m  \ra \infty}\E \left[ \sum_{j=0}^\infty \ind_{\left\{\ftime_n^k \leq \frac{j}{2^m} < T,
\stime_n^k > \frac{j}{2^m}\right\}}
\dfrac{ G\left( \frac{j+1}{2^m} - \ftime_k^n \right)  - G\left(\frac{j}{2^m} - \ftime_k^n \right)}{1 -
G\left(\frac{j}{2^m} - \ftime_n^k \right)}  \right]  \\
\quad \quad = \ds \E \left[\lim_{m \ra \infty} \int_0^\infty
\ind_{[0,T]} (u) \, d A^k_{m,n} (u) \right] \\
\quad \quad = \ds
\E\left[ \int_0^\infty \ind_{[0,T]} (u) \, d A_n^k (u) \right].
\ea
\]
This establishes (\ref{to-prove}).   In particular, this shows that
$\E [\dcompn_{\f1} (t)] = \E [\dn(t)] \leq \E[\en(t) + \xn(0)] < \infty$, which guarantees
that $\dcompn_{\f1}$ is locally integrable.
\end{proof}

>From the last lemma and the comments below (\ref{def-baren}), it is clear that
 for every $\newf \in \newcbpm$ and $t \in [0,\infty)$,
$\bare^{(N)}_{\newf}(t)$ and $\dcompn_\newf (t)$ have finite expectation.
Together with the fact that
 the ages of customers are (left) continuous and hence predictable, the last
lemma immediately implies the
following (seemingly stronger) result.

\begin{cor}
\label{cor-compensator} For every $\newf \in \newcbpm$, the
process $\dcompn_{\newf}$ is the ${\cal F}_t^{(N)}$-compensator of
the  process $\bare^{(N)}_{\newf}$. In particular, the  process
$\martn_{\newf}$ defined by \be \label{def-martn}
 \martn_{\newf}  \doteq \ds \bare^{(N)}_{\newf} - \dcompn_{\newf}
\ee
 is a local ${\cal F}_t^{(N)}$-martingale.
\end{cor}

As usual, let $\fbaren_{\newf}$, $\fdcompn_{\newf}$ and
$\fmartn_{\newf}$, respectively,  denote the scaled processes
$\baren_{\newf}/N$, $\dcompn_{\newf}/N$ and  $\martn_{\newf}/N$.
The following lemma will be used in Section \ref{subs-tight}
to establish tightness of these processes.

\begin{lemma}
\label{lem-estimate} For every $T < \infty$, $\sup_{N} \E \left[ \fdn (T) \right] < \infty$.
Also, for $t \in [0,\infty)$ and  $N \in \N$,
\be
\label{lim-depart}
\lim_{\delta \ra 0} \E \left[ \fdn (t +\delta) - \fdn (t) \right] = 0.
\ee
Moreover, for every $\delta > 0$
and  interval $\zeroset = [L+\delta, M)$
with $L \in [-\delta, M - \delta)$,
\be
\label{est-renewal}
\E \left[ \fbaren_{\ind_\zeroset} (t+\delta) - \fbaren_{\ind_\zeroset} (t) | {\cal F}_t^{(N)} \right]
\leq U(\delta) \fmeas_t^{(N)} [L^+, M)
\ee
where $U(\cdot)$ is the renewal function associated with the service distribution
$G$.
\end{lemma}
\begin{proof}
For notational conciseness,
 throughout this proof we will use
 $f(t,t+\delta)$ to denote $f(t+\delta) - f(t)$ for any function $f$,
$t \in [0,\infty)$ and $\delta > 0$. Since $\ind_\zeroset$ is only
a function of $x$, we can write 
\[ \ba{rcl} \ds \baren_{\ind_{\zeroset}} (t,t+\delta)  =  \ds
\sum_{j=-\lan 1,\measn_0\ran+1}^{\kn(t+\delta)} \suli_{s \in
[t,t+\delta]} \ind_{\{d\agen_j(s-)/dt>0,\agen_j (s) = v_j\}}
\ind_{\zeroset} \left(\agen_j (s)\right),\ea \] which is simply
the number of departures from the $N$th system during the time
interval $[t,t+\delta]$ by customers whose service times lie in
the set $\zeroset$ at the time of departure.

 We shall bound the  departures during the time interval $[t,t+\delta]$ in the $N$th system
by the departures in another system that is easier to analyse.
Consider a modified system in which after time $t$, there are an
infinite number of arrivals (or, equivalently, customers in queue)
so that at each station, every time a customer finishes service, a
new customer joins. Let $\tilde{D}_1 (\delta|x)$ denote the number
of departures from one station in this modified system during the
period $[t,t+\delta]$, given that at time $t$ there exists a
customer with age $x$ in that station (note that, as the notation
reflects, this quantity is independent of $t$ and the choice of station).
In fact, The quantity
$\tilde{D}_1 (\delta|x)$ is simply the number of renewals in the
interval $[0,\delta]$ of a delayed renewal process with initial
distribution that has density $g_0(y) = g(y+x)/(1 - G(x))$, and
inter-renewal distribution $G$. Thus, as is well-known (see, for
example, Theorem 2.4(iii) of \cite{asmbook}), $\E [ \tilde{D}_1
(\delta|x)]$ is bounded above by $U(\delta)$, where $U(\cdot)$ is
the renewal function of the pure renewal process that has
inter-renewal distribution $G$ (and a renewal at $0$).

Let $\tilde{D}(t,t+s]$ be the departure process from this system
during the interval $(t,t+s]$. Since both the original and
modified systems are in the same state at time $t$, and the
cumulative arrivals into service in the modified system are
greater than in the original system, a simple monotonicity
argument shows  that
\be
\label{series}
 \ba{rcl}
\ds \E\left[ \fdn (t+\delta) - \fdn (t)|{\cal F}_t^{(N)} \right]
  & \leq &  \dfrac{1}{N} \ds \E \left[\tilde{D} (t,t+\delta]|{\cal
 F}_t^{(N)}\right] \\
& = &  \dfrac{1}{N} \ds \int_{[0, M)} \E \left[ \tilde{D}_1 (\delta|x) \right]
\, \measn_t (dx) \\
& \leq & \ds   U(\delta) \fmeasn_t [0, M)  \\
& \leq & \ds  U(\delta).
\ea
\ee
Now, $U(\delta)$ is finite for any finite $\delta$ and non-decreasing
(see, e.g.,  Theorem 2.4(i) of \cite{asmbook}).
Taking $t = 0$ and $\delta = T$ and expectations on
both sides, this immediately shows that $\sup_{N \in \N} \E[\fdn (T)] \leq U(T) < \infty$.
Also, since $\E[ \tilde{D}_1 (\delta | x)]$ converges monotonically down to zero as $\delta \ra 0$,
 the bounded convergence theorem shows that
for every $N \in \N$,
\[
 \lim_{\delta \ra 0}\int_{[0, M)} \E \left[ \tilde{D}_1 (\delta|x) \right] \,
\fmeasn_t (dx)
 = 0.
\]
Taking expectations of both sides of (\ref{series}) and then sending $\delta \ra 0$,
the last display and another application of the bounded convergence theorem then
shows that for every $N \in \N$,  (\ref{lim-depart}) holds.

To establish the second estimate, fix $\delta > 0$ and
$L \in [-\delta, M - \delta)$.  If $L \leq 0$ then the left-hand side
of (\ref{est-renewal}) is bounded by $\E [\fdn(t,t+\delta)|{\cal F}_t^{(N)}]$
and so the bound (\ref{est-renewal}) follows from the second inequality
in (\ref{series}).
Now suppose $L \in (0, M - \delta)$.
Then any customer whose service time is greater than or equal to $L+\delta$
and  who departed the system during the time interval
$[t,t+\delta]$
must have been in the system at time $t$ with age greater than or equal to
$L >0$.
Thus the total number of such departures is bounded above
by the number of departures in the modified system from
stations that had a customer present at time $t$ with
age greater than or equal to $L$.
By the same reasoning provided above, this implies that
\[ \E \left[ \fbaren_{\ind_{[L+\delta, M)}} (t,t+\delta) |{\cal F}_t^{(N)}\right]
\leq \int_{[L, M)} \E \left[ \tilde{D}_1 (\delta|x) \right] \fmeasn_t (dx)  \leq U(\delta)\fmeasn_t[L, M),
\]
 which completes the proof of the lemma.
\end{proof}

We now use this estimate to establish some convergence results,
which will be used to prove tightness in Section \ref{subs-tight}.
The assumptions (\ref{eq-dcompn}) and (\ref{eq-dcompn2})
 of the
lemma below are shown to follow from Assumption \ref{ass-init}(3)
in Lemma \ref{lem-tight1}.

\begin{lemma}
\label{lem-fdcompn}
 Suppose that the limit
\be \label{eq-dcompn}
 \lim_{L \ra M} \sup_{N \in \N} \E \left[\fmeasn_0 (L,M)\right] =
 0
\ee
holds and, if $M < \infty$, then
\be
\label{eq-dcompn2}
\lim_{L \ra M} \sup_{N \in \N}\E \left[ \int_{[0,L)}
  \frac{1-G(L)}{1-G(x)}\,  \fmeasn_0(dx) \right]
=0
\ee
is also satisfied.
 Then for $t \in [0,\infty)$,  the following three properties hold.
\begin{enumerate}
\item
\[ \lim_{L \ra M} \sup_{N} \E \left[ \int_0^t \left( \int_{[L,M)} h(x) \, \fmeasn_s (dx) \right) \, ds\right] = 0.
\]
\item
For every $\newf \in \newcbpm$,
\[ \lim_{\delta \ra 0} \limsup_N  \E \left[ \sup_{t \in [0,T]} \left(\fdcompn_{\newf} (t+\delta)
- \fdcompn_{\newf} (t)  \right) \right] = 0
\]
and
\[ \lim_{\delta \ra 0}  \limsup_N  \E \left[   \fbaren_{\newf} (t+\delta)
- \fbaren_{\newf} (t)  \right] = 0.
\]
\item
Given $L < M$ and any sequence of subsets $H_R \subset
[0,L]$ such that the Lebesgue measure of $H_R$ goes to zero as $R
\ra \infty$, we have for every $t \in [0,\infty)$, \be
\label{est-onemore}
  \lim_{R \ra \infty} \limsup_{N} \E \left[ \sup_{t \in [0,T]} \fdcompn_{\ind_{H_R}} (t) \right]
= 0. \ee
\end{enumerate}
\end{lemma}
\begin{proof}
As in the last two proofs,
 throughout this proof we will use
 $f(t,t+\delta)$ to denote $f(t+\delta) - f(t)$ for any function $f$,
$t \in [0,\infty)$ and $\delta > 0$.  We shall divide the
 proof of the first property into two cases.  \\
{\em Case 1.} $M = \infty$.   We start by proving a preliminary result,
 (\ref{limit1}) below.
 Due to the linear
increase of the ages of customers, the assumed limit
(\ref{eq-dcompn}) automatically implies that for every $s \in
[0,\infty)$,
\be \label{nut-ineq1}
   \lim_{L \ra \infty} \sup_{N \in \N} \E \left[\fmeasn_s [L+s,\infty)\right]  = 0.
\ee   For $L, \Delta, s \in [0,\infty)$, let  $\newchi_{L,\Delta,s}
\in \cb$ be such that
\[ \ind_{[2L+\Delta+s,\infty)} \leq \newchi_{L,\Delta,s} \leq \ind_{[L+\Delta+s,\infty)}.
\]
Then by Lemma \ref{lem-compensator} and (\ref{est-renewal}), we
have for every $N \in \N$,
\[
\ba{rcl} \ds  \E \left[ \fdcompn_{\newchi_{L,\Delta, s}}
(s,s+\Delta)|{\cal F}_s^{(N)} \right] &  = & \ds
\E \left[ \fbaren_{\newchi_{L,\Delta, s}} (s,s+\Delta) |{\cal F}^{(N)}_s \right]   \\
& \leq & \ds \E \left[ \fbaren_{\ind_{[L+\Delta + s,\infty)}} (s,s+\Delta) |{\cal F}_s^{(N)} \right]   \\
& \leq & \ds U(\Delta) \fmeasn_s [L+s,\infty) \ea
\]
Taking expectations of both sides,  we see that \be \label{limit1}
  \E \left[ \fdcompn_{\newchi_{L,\Delta,s}} (s,s+\Delta) \right] \leq
U(\Delta) \E \left[ \fmeasn_s [L+s,\infty)\right]. \ee
 We  now show how the first property (in the case $M = \infty$) follows
from the the above estimate.
 Fix $t \in [0,\infty)$,  choose $L > t$ and let
$\tilde{L} \doteq (L - t)/2$. Then we have
\begin{eqnarray*}
 \E \left[ \int_0^t \left( \int_{[L,\infty)} h(x) \, \fmeasn_s (dx) \right) \, ds  \right] & = & \E \left[ \int_0^t \left( \int_{[2\tilde{L}+t,\infty)} h(x) \, \fmeasn_s (dx) \right) \, ds  \right] \\
& \leq &   \E \left[ \fdcompn_{\newchi_{\tilde{L},t,0}} (t) \right] \\
& \leq &  U(t) \E\left[\fmeasn_0 [\tilde{L},\infty)\right],
\end{eqnarray*}
where
 the last inequality is justified by the estimate (\ref{limit1}),
with $L,\Delta$ and $s$ replaced by $\tilde{L}, t$ and $0$,
respectively. Taking the supremum of both sides over $N$, and then
sending  $L \ra \infty$ (in which case
 $\tilde{L} \ra \infty$),  the relation
  (\ref{eq-dcompn}) ensures that property (1) holds for the case $M=\infty$. \\
{\em Case 2.}   $M<\infty.$   In this case, for $L < M$,
 it is not hard to infer that
\[
\begin{array}{l}
\ds \E \left[ \int_0^t \left( \int_{[L,M)} h(x) \, \fmeasn_s (dx) \right) \, ds
  \right]  \\
  \le
  \ds \E\left[\fmeasn_0(L,M)\right]
+ \E \left[ \int_{[0,L]}\left( \frac{1-G(L)}{1-G(x)} \right) \, \fmeasn_0(dx) \right]
 \ds + (G(M)-G(L)) \E\left[\fen(t)\right],
  \end{array}
\]
  where the last inequality uses the independence of the
  arrival process and the service requirements of the customers.
  Taking the supremum of both sides over $N$ and then sending $L\ra
  M$,  (\ref{eq-dcompn}), (\ref{eq-dcompn2}) and
Assumption 1(1)
ensure that   property (1) is satisfied.

We now turn to the proof of property (2).
 This property holds for all $\newf \in \newcbpm$ if and only
if it holds for $\newf = \f1$.
Moreover,  to establish the property for
$\newf = \f1$, we claim that it suffices to show that
 for every $L < M$,
\be
\label{claim-fdcompn}
 \lim_{\delta \ra 0} \limsup_{N} \E
\left[ \fdcompn_{\ind_{[0,L]}} (t,t+\delta) \right] = 0.
\ee
To see why this is true, observe that for any $L < M$, we
have
\[
\ba{rcl} \E\left[ \fdcompn_\f1 (t,t +\delta) \right] & = & \ds
\E \left[ \int_t^{t+\delta} \left( \int_{[0,L)}  h(x) \, \fmeasn_s (dx) \right) \, ds \right] \\
& & \quad \quad \ds + \E \left[ \int_t^{t+\delta} \left(  \int_{[L, M)}
h(x) \, \fmeasn_s (dx)\right) \, ds \right]. \ea
\]
Taking $\limsup$ over $N$, and then sending $\delta \ra 0$,
(\ref{claim-fdcompn}) implies that the first term vanishes.
Subsequently  sending $L \ra M$,  the second term goes to
zero by property (1), thereby yielding the first relation of property (2).
The second relation then follows trivially from the first on
account of
 Corollary \ref{cor-compensator}.

In order to establish (\ref{claim-fdcompn}), for $\delta > 0$  we
first rewrite
 $\dcompn_{\ind_{[0,L]}} (t,t+\delta)$,
 in a manner analogous to (\ref{altrep-dcompn}),
as
\begin{eqnarray*}
\dcompn_{\ind_{[0,L]}} (t,t+\delta) &  = &  \suli_{-\lan \f1,
\measn_0 \ran + 1}^0 \int_t^{t+\delta} h \left(\agen_j (0) + s
\right) \ind_{\left\{ \agen_j (0) + s \leq L\right\}} \, ds
 \\
& & \ds \quad  \quad + \sum_{j=1}^{\kn(t+\delta)}
\int_t^{t+\delta} h \left( s - \kninv_j \right) \ind_{\left\{ s- \kninv_j \leq L \right\}} \, ds \\
& \leq &   \left[ N + \en (t+\delta) \right] \sup_{u \in [0,L]}
\int_{u}^{u+\delta} h(r) \, dr.
 \end{eqnarray*}
Since $L < M$, the local integrability of $h$ on $[0,M)$ implies that
\[ C_L (\delta) \doteq \sup_{u \in [0,L]} \int_{u}^{u+\delta} h(r) \, dr \ra 0
\quad \mbox{ as } \quad  \delta \ra 0.
\]
Therefore, dividing the previous relation by $N$, taking the
supremum over $t \in [0,T]$,  then taking limits first as $N \ra \infty$, and then $\delta \ra 0$
and invoking Assumption \ref{ass-init}(1) we see that, a.s.,
\[ \lim_{\delta \ra 0}
\limsup_{N \ra \infty} \sup_{t \in [0,T]} \fdcompn_{\ind_{[0,L]}}
(t,t+\delta) \leq \lim_{\delta \ra 0} \left[1 + \fe(T) \right] C_L
(\delta)  = 0.
\]
Since $\E [\fe(T)] < \infty$ by Assumption \ref{ass-init}(1)
an application of the dominated convergence theorem then yields
(\ref{claim-fdcompn}).

The proof of the last property of the lemma is similar to that of
(\ref{claim-fdcompn}). Indeed, using the representation
(\ref{altrep-dcompn}) and, arguing as above, we can conclude that
for any $T < \infty$,
\[ \limsup_{N} \sup_{t \in [0,T]}
\E \left[ \fdcompn_{\ind_{H_R}} (t)\right]  \leq  \left( 1 +
\E[\fe(T)] \right)  \int_{H_R} h(r) dr. \] Since the sets $H_R$
are all contained in a compact set in $[0,M)$, taking limits as $R
\ra \infty$, the right-hand side vanishes due to the local
integrability of $h$.
\end{proof}

We now show that the quadratic variation process
 $\lan \fmartn_\newf \ran$ converges to zero, as $N \ra \infty$.
$\fmartn_\newf$ has a well-defined predictable quadratic variation
process because the fact that  $\fmartn_{\newf}$ is a compensated
sum of jumps with bounded jump sizes automatically implies that it
is locally square integrable (see, for example,  (4.1) on p.\ 38
of \cite{jacshibook}).

\begin{lemma}
\label{lem-mgale} For every $\newf \in \newcbpm$ and $t \in
[0,\infty)$,
 \be \label{lim-qv}
 \lim_{N\ra\infty}\E \left[ \lan \fmartn_{\newf} \ran (t)\right] = 0. \ee
 Consequently, $\fmartn_\newf \Rightarrow \zerof$ as $N
\ra \infty$.
\end{lemma}
\begin{proof}
 Since the compensator $\dcompn_{\newf}$ is continuous, the martingale $\martn_{\newf}$ does not
have any predictable jumps, and so by Proposition II.2.29 of
\cite[page 82]{jacshibook},  the predictable quadratic variation
of the martingale is given by
\[ \lan \martn_{\newf}\ran (t)  =
\int_0^t \left( \int_{[0, M)} \newf^2(x,s) h(x) \, \measn_s
(dx) \right) \, ds. \]
 This means that the scaled process
$\fmartn_\newf$ has predictable quadratic variation
\[ \lan \fmartn_{\newf} \ran (t) = \dfrac{1}{N^2} \lan \martn_{\newf} \ran (t)  =
\dfrac{1}{N} \left[ \int_0^t \left( \int_{[0, M)} \newf^2(x,s)
h(x) \, \fmeasn_s (dx) \right) \, ds \right],
\]
which implies that for any $t \in [0,\infty)$,
\[
 \lan  \fmartn_{\newf} \ran (t)
 \leq \dfrac{\nrm{\newf}_\infty^2}{N}
 \fdcompn_{\f1} (t). \]
Since $\E[\fdcompn_{\f1} (t) ] = \E [ \fdn (t)]$ by
Lemma  \ref{lem-compensator},  the estimate in
(\ref{series}) (with $s$ and $t$ there replace by $t$ and $0$, respectively)
shows that $\sup_{N\in\N}
\E[\fdcompn_{\f1}(t)] \le U(t) < \infty$, where $U(\cdot)$ is the renewal function
associated with the service distribution $G$.
Thus taking first expectations and then limits as $N \ra \infty$
in the last display, we obtain  the first assertion of the lemma.
In order to show that $\fmartn_\newf \Rightarrow
\zerof$ as $N \ra \infty$, we note that by Doob's Lemma, for
any $\lambda>0$
\[ \P\left(\sup_{s\in [0, T]}\fmartn_{\newf}>\lambda\right)\le \dfrac{\E[\lan
\fmartn_{\newf}\ran (T)]}{\lambda^2},\]
which converges to $0$ as $N\ra\infty$ by the first assertion.
Since this is true for all $\lambda > 0$, this
 completes the proof of the lemma.
\end{proof}

\subsection{Proof of Relative Compactness.}
\label{subs-tight}

We now  establish the relative compactness of the sequence of
scaled state processes $\{(\fxn, \fmeasn), N \in \N\}$, as well as
of several of the auxiliary processes. For this, it will prove
convenient to use Kurtz' criteria for relative compactness of
processes $\{Y^{(N)}\}$ with sample paths in ${\cal D}_{\R_+}
[0,\infty)$.


\noi {\em Kurtz' criteria. }
\begin{enumerate}
\item[K1.] For every rational $t \geq 0$,
 \be
\label{K1}
  \lim_{R \ra \infty} \sup_{N} \P ( Y^{(N)} (t) > R) = 0;
\ee
\item[K2.]
For each $t > 0$, there exists $\beta > 0$ such that \be
\label{K2}
 \lim_{\delta \ra 0} \sup_{N} \E \left[ \left|Y^{(N)} (t+\delta) - Y^{(N)} (t)\right|^\beta \right] = 0.
\ee
\end{enumerate}
The sufficiency of these conditions for relative compactness
follows from Theorem 3.8.6 of \cite{ethkurbook} (condition K1
corresponds to condition (a) of Theorem 3.7.2 in \cite{ethkurbook}
and condition K2 follows from condition (b) of Theorem 3.8.6 and
Remark 3.8.7 in \cite{ethkurbook}).

\begin{lemma}
\label{tight-xn} Suppose  Assumption \ref{ass-init} holds. Then
the sequences  $\{\fxn\}$ and  $\{\lan \f1, \fmeasn \ran\}$ and,
for every $\newf \in \newcbpm$, the sequences $\{\fbaren_\newf\}$,
$\{\fdcompn_\newf\}$ and $\{\fdn\}$ are relatively compact in
${\cal D}_{\R_+} [0,\infty)$.
\end{lemma}
\begin{proof}
By Corollary \ref{cor-compensator} and Lemma \ref{lem-estimate},
we know that for every $\newf \in \newcbpm$ and
 $t \in [0,T]$,
\[  \sup_{N} \E[\fdcompn_\newf (T)] = \sup_{N} \E[ \fbaren_\newf (T)] \leq  \nrm{\newf}_\infty \sup_{N}
\E [\fdn (T)] < \infty.
\]
This is easily seen to imply that $\{\fdcompn_\newf\}$ and
$\{\fbaren_\newf\}$ satisfy condition K1. In addition, from Lemma
\ref{lem-fdcompn}(2) it follows that for every $\newf \in \newcb$,
$\{\fdcompn_\newf\}$ and $\{\fbaren_\newf\}$ satisfy criterion K2
(with $\beta=1$) , and thus they are relatively compact.
Since $\fdn = \fbaren_{\f1}$, this also implies that
$\fdn$  satisfies conditions K1 and K2 (the latter with
 $\beta=1$).

By Assumption \ref{ass-init},
 the sequences $\{\fen\}$ and $\{\fxn (0)\}$ are relatively compact.
Therefore, by  Theorem 3.7.2 of \cite{ethkurbook},
 they satisfy (\ref{K1}).
 The elementary bound
 \[  \lan \f1, \fmeasn_t \ran  \leq \fxn (t) \leq \fxn(0) + \fen (t), \]
then shows that the sequences $\{\fxn\}$ and $\{\lan \f1,
\fmeasn\ran\}$ also satisfy condition K1. To complete the proof of
relative compactness of these sequences, we will use a slightly
different set of criteria, namely K1 above and condition (b) of
Theorem 3.7.2 of \cite{ethkurbook}, which is expressed in terms of
the modulus of continuity  $w^\prime (f, \delta, T)$ of a function
$f$ (see (3.6.2) of \cite{ethkurbook} for a precise definition).
Now,  for every $0 \leq s \leq t < \infty$, from
(\ref{eqn-prelimit2}) it is clear that
\[  \left|\fxn (t) - \fxn(s)\right| \leq  \left| \fen (t) - \fen (s) \right| \vee  \left| \fdn (t) - \fdn (s) \right|
\]
and the complementarity condition (\ref{comp-prelimit})  shows
that
\[ \left| \lan \f1, \fmeasn_s \ran  - \lan \f1, \fmeasn_r \ran \right| \leq
\left| \left[1 - \fxn_s \right]^+ - \left[1 - \fxn_r \right]^+
\right| \leq \left|\fxn_s - \fxn_r \right|.
\] From this it is easy to see that for every $N \in \N$, $\delta > 0$ and $T < \infty$,
\[ w^\prime (\lan \f1, \fmeasn\ran, \delta, T) \vee w^\prime (\fxn, \delta, T) \leq w^\prime (\fen, \delta, T) \vee w^\prime (\fdn, \delta, T). \]
The relative compactness of $\{\fxn\}$ and $\{\lan \f1,
\fmeasn\ran\}$ is then a direct consequence of the  above
estimate, the relative compactness of  $\{\fen\}$ and $\{\fdn\}$
and
 Theorem  3.7.2 of \cite{ethkurbook}.
\end{proof}

Next, we show that the sequence of measure-valued processes $\{\fmeasn\}$ is
tight. We shall first establish  tightness of the measure-valued
process $\{\fnewmeasn\}$, defined by \be \label{def-newmeas}
 \fnewmeasn  \doteq \delta_0 \fidlen + \fmeasn,
\ee where $\fidlen = 1 - \lan \f1, \fmeasn\ran$ is the process
that keeps track of  the fraction of idle servers in the $N$th
system. In other words, for every Borel set $A \subset [0, M)$
and $t \in [0,\infty)$, we have
\[ \fnewmeasn_t (A) = \ind_{A} (0)\fidlen (t) + \fmeasn_t(A).  \]
The reason for introducing $\fnewmeasn$ is that now for every $t$,
$\fnewmeasn_t$ is a probability measure
(as opposed to a sub-probability measure),
and so $\fnewmeasn$  is a random element taking
values in ${\cal D}_{{\cal M}_1[0,M)} [0,\infty)$.
 Thus, in order to prove tightness of the sequence $\{\fnewmeasn\}$,
 we can apply Jakubowski's criteria for tightness of measures
on ${\cal D}_{{\cal M}_1[0,M)} [0,\infty)$,
which are summarised below (see, e.g., \cite{jak86} or Theorem 3.6.4 of \cite{dawbook} for the case $M = \infty$;  the conditions listed below for the general case $M < \infty$ follow by analogy).  \\

\noi {\em Jakubowski's Criteria.} A sequence $\{\fnewmeasn\}$ of
${\cal D}_{{\cal M}_1[0,M)} [0,\infty)$-valued random elements
defined on $(\Omega, {\cal F}, \P)$ is tight if and only if the
following two conditions are satisfied.
\begin{enumerate}
\item[J1.]
For each $T > 0$ and $\eta > 0$ there exists a compact set
${\cal K}_{T,\eta} \subset {\cal M}_1 [0,M)$ such that for all $N \in
\N$,
\[\P \left( \fnewmeasn_t \in {\cal K}_{T,\eta} \mbox{ for all } t \in [0,T] \right) > 1 - \eta. \]
This is referred to as the {\em compact containment condition}.
\item[J2.]
There exists a family $\F$ of real continuous functions $F$ on
${\cal M}_1[0,M)$ that separates points in ${\cal M}_1[0,M)$ and
is closed under addition such that
 $\{\fnewmeasn\}$ is $\F$-weakly tight, i.e., for every $F \in \F$,
the sequence  $\{F(\fnewmeasn_s), s \in [0,\infty)\}$   is tight
in ${\cal D}_{\R}[0,\infty)$.
\end{enumerate}

\begin{lemma}
\label{lem-tight1} Suppose Assumption \ref{ass-init} holds. Then
the sequence  $\{\fnewmeasn\}$ satisfies Jakubowski's criterion
J1. Moreover, the limits (\ref{eq-dcompn}) and (\ref{eq-dcompn2})
hold.
\end{lemma}
\begin{proof}
We split the proof into two cases, depending on $M$.   \\
{\em Case 1.}  $M = \infty$. By Assumption \ref{ass-init}(3) and
the complementarity condition (\ref{comp-prelimit}), there exists
a set $\tilde{\Omega}$ of measure $1$ such that for all $\omega
\in \tilde{\Omega}$, $\fmeasn_0(\omega)$ converges weakly,  as $N
\ra \infty$, to a sub-probability measure $\fmeas_0 (\omega)$,
which in turn implies that $\fidlen (0) (\omega)$ converges to
$\fidle(0)(\omega) = 1 - \lan \f1, \fmeas_0 (\omega) \ran$. Fix
$\omega \in \tilde{\Omega}$. Then by Prohorov's theorem (see
Section 3.2 of \cite{ethkurbook}), the sequence $\{ \fnewmeasn_0
(\omega), N \in \N\}$ must be tight, and hence for every $\ve > 0$
there exists a constant $r(\omega, \ve) < \infty$ such that
\[ \fnewmeasn_0 (\omega) \left((r(\omega, \ve), \infty)\right) < \ve  \quad \mbox{ for all } N \in \N. \]
Since $r(\omega, 1/n) < \infty$ for every $\omega \in
\tilde{\Omega}$ and $n \in \N$, there exists a sequence $r(n)$
that converges to $\infty$ as $n \ra \infty$ such that $\P(\omega:
r (\omega, 1/n ) > r(n)) \leq 2^{-n}$. If we define $A_n \doteq
\{\omega: r(\omega, 1/n) > r(n)\}$,  then by the Borel-Cantelli
lemma, almost surely $A_n$ occurs only finitely often.
Furthermore, $\P\left( \cup_{n \geq N_0} A_n \right) \leq
2^{-N_0}$  for every $N_0 \in \N$. Now fix $T < \infty$, and note
that since the age process of each customer in service increases
linearly, for every $n \in \N$ and $t \in [0, T]$,
\[\ba{rcl}
 \ds \left\{ \omega: \fnewmeasn_0 (\omega) (r(n), \infty) \leq  \frac{1}{n} \right\}
& \subseteq & \ds \left\{ \omega: \fnewmeasn_0 (\omega) (r(n) - t + T, \infty) \leq \frac{1}{n} \right\} \\
& \subseteq & \ds  \left\{  \omega: \fnewmeasn_t (\omega) (r(n) +
T, \infty) \leq \frac{1}{n} \right\}. \ea
\]

Thus given $\eta > 0$, now  define
\[ {\cal K}_{\eta, T} \doteq  \left\{ \mu \in {\cal M}_1(\R_+):
\mu \left( r(n) + T, \infty \right) \leq \frac{1}{n} \mbox{ for
all } n > N_0 (\eta) \right\},
\]
where we choose $N_0 (\eta) \doteq -\lceil \ln \eta/\ln 2 \rceil$
so that  $2^{-N_0} < \eta$. Then observe that
\[ \inf_{C \subset \R_+:C compact} \sup_{\mu \in {\cal K}_{\eta,T}}
\mu (C^c) \leq \inf_{n > N_0(\eta)} \sup_{\mu \in {\cal
K}_{\eta,T}} \mu (r(n) + T, \infty) = 0.  \] Therefore, another
application of  Prohorov's theorem shows that ${\cal K}_{\eta, T}$
is a relatively compact subset of ${\cal M}_1(\R_+)$ (equipped
with the Prohorove metric). Let $\overline{{\cal K}}_{\eta,T}$ be
its closure in the Prohorov metric.
 Then for every $N \in \N$,
\[ \ba{l}
\ds \P \left(\fnewmeasn_t \in \overline{{\cal K}}_{\eta,T} \mbox{ for every } t \in [0,T] \right)  \\
\quad \quad  \quad \quad \ds \geq \P \left(\fnewmeasn_0 (r(n),
\infty) \leq \frac{1}{n} \mbox{ for every } n > N_0(\eta) \right)
 \geq  \ds 1 - 2^{-N_0} \geq 1 - \eta,
\ea
\]
which proves the compact containment condition when $M = \infty$.

In addition, this also shows that (\ref{eq-dcompn}) holds.
 Indeed, if $\eta > 0$ and $N_1 (\eta) \in \N$
satisfies $N_1(\eta) \geq N_0 (\eta) \vee [1/\eta]$, then the last
display implies that for every $K \geq  r (N_1 (\eta))$,
\[ \inf_{N \in \N} \P \left( \fnewmeasn_0 [K,\infty) \leq \dfrac{1}{N_1 (\eta)} \right) \geq 1 - \eta,\]
which in turn shows that for every for every $K \geq  r (N_1
(\eta))$,
\[ \sup_{N \in \N} \E \left[ \fmeasn_0 [K,\infty) \right] \leq \dfrac{1}{N_1 (\eta)} + \eta. \]
The result then follows by letting first $K \ra \infty$ and then
$\eta \ra 0$.

\noi
{\em Case 2.}  $M < \infty$.  We start by establishing (\ref{eq-dcompn}) and (\ref{eq-dcompn2})
using an argument similar to that used in Case 1 to prove (\ref{eq-dcompn}).
The almost sure weak convergence of
$\fmeasn_0$ to $\fmeas_0$  in  ${\cal M}_{\leq 1}[0,M)$
 implies that the sequence $(r(n))$ considered in Case 1 can be taken
strictly smaller than $M$ and converging to $M$. Defining
$N_1(\eta)$ and $K$ as in Case 1,  we see that for $K \geq r(N_1 (\eta))$,
\[ \sup_{N \in \N} \E \left[ \fmeasn_0 [K,M) \right] \leq \dfrac{1}{N_1 (\eta)} + \eta,  \]
and the result  follows as above by sending first $K\ra M$ and then
$\eta\ra 0$. The limit (\ref{eq-dcompn2}) follows from the
weak convergence of $\fmeasn_0$ to $\fmeas_0$ and the fact that
$(1-G(L))/(1-G(x))$ is bounded (by $1$) and continuous on
$[0,L)$.

It only remains to show that the compact containment condition is
satisfied when  $M<\infty$.  For this, we
 need to show that for every $\ve>0$, $\eta>0$ we can find $L(\ve)<M$
so that
\[ \inf_{N} \P\left(\fnewmeasn_t [0, L(\ve) ]>1-\ve \quad \mbox {for
every}\quad t\in [0,T]\right)>1-\eta.
\]
 However, for any $L < M$, we have
\begin{eqnarray*}
\P\left(\fnewmeasn_t(L,M)>\ve\quad \mbox{for some}\quad t\in[0,T]\right)
& \le &  \P\left(\fqn_{\ind_{(L,M)}}(T+M)>\ve \right) \\
& \le & \dfrac{\E[\fqn_{\ind_{(L,M)}}(T+M)]}{\ve} \\
& = &  \dfrac{\E[\fdcompn_{\ind_{(L,M)}}(T+M)]}{\ve}
\end{eqnarray*}
where the last equality follows from Corollary \ref{cor-compensator}.
Using now  Lemma
\ref{lem-fdcompn}(1) (which is justified since we have already established
(\ref{eq-dcompn}) and (\ref{eq-dcompn2})),
 one can find $L(\ve)$
close enough to $M$ to make the supremum over $N$  of the
right-hand side above smaller than
$\eta$, thus yielding the desired result.
\end{proof}

zz

\begin{lemma}
\label{lem-tight2} Suppose Assumption \ref{ass-init} is satisfied.
Then for every $f \in \ocdbm$, the sequence $\{\lan f,
\fnewmeasn_\cdot \ran\}$ of ${\cal D}_{\R} [0,\infty)$-valued
random variables is tight.
\end{lemma}
\begin{proof}  Let $f \in \ocdbm$. By Prohorov's theorem, it suffices to
establish relative compactness of the sequence $\{\lan f,
\fnewmeasn_\cdot\ran\}$.
 First substituting $\newf = f$ (interpreting $f$ as a function on $[0,M) \times \R_+$ that depends only
on the first variable) in the equation (\ref{eqn-prelimit1})
satisfied by the prelimit,  then dividing the equation by $N$  and
using the definition of $\fnewmeasn$
we obtain  for any $t \in [0,\infty)$,
\[
\ba{rcl} \ds \left\lan f, \fnewmeasn_{t} \right\ran -  \left\lan
f, \fnewmeasn_{0} \right\ran \ds & = & \left\lan f, \fmeasn_{t}
\right\ran -  \left\lan f, \fmeasn_{0} \right\ran \ds
+ f(0) [\fidlen (t) - \fidlen (0)] \\
& = &  \ds \int_0^t  \lan f^\prime, \fnewmeasn_{s} \ran \, ds  -
f^\prime (0) \int_0^t \fidlen (s) \, ds   \\
& & \quad \ds -\fbaren_f (t)
+  f(0) \left[ \fkn (t) + \fidlen (t) - \fidlen (0)\right] \\
& = &  \ds \int_0^t  \lan f^\prime, \fnewmeasn_{s} \ran \, ds  -
f^\prime (0) \int_0^t \fidlen (s) \, ds  \\
& & \ds \quad  -\fbaren_f (t) + f(0) \fbaren_\f1 (t),  \\
\ea
\]
where the last equality uses the relation $\fkn + \fidlen  -
\fidlen (0) = \fdn = \fbaren_\f1$, which follows from
(\ref{def-idlen}) and (\ref{def-kn}).
Thus to show that $\{\lan f, \fmeasn\ran\}$ is relatively compact,
it suffices to show that $\{ \lan f, \fmeasn_0 \ran \}$ and the
sequences associated with each of the four terms on the right-hand
side of the last display is relatively compact. Lemma
\ref{tight-xn} immediately guarantees this for the last two terms
and
 Assumption \ref{ass-init}(3) ensures this for  $\{ \lan f, \fmeasn_0\ran\}$.
In addition, since $\fmeasn$ is a sub-probability measure and
$\fidlen  (s) \in [0,1]$ for all $s$,   for every $N \in \N$, the
first two terms are uniformly bounded by $\nrm{f^\prime}_\infty t$
and, moreover,
\[ \left[\int_t^{t+u} \left|\lan f^\prime, \fnewmeasn_{s} \ran\right| \, ds\right] \vee \left[ f^\prime (0) \int_t^{t+u} \fidlen (s) \, ds \right] \leq  \nrm{f^\prime}_\infty u.
\]
This verifies Kurtz' criteria K1 and K2 (with $\beta = 1)$, and
thus establishes relative compactness of the sequences associated
with the remaining two terms.
 \end{proof}

We are now ready to state the main tightness result.

\begin{theorem}
\label{th-tight} Suppose Assumption \ref{ass-init} is satisfied .
Then the sequence \newline $\{(\fxn, \fmeasn, \fdcompn_\f1,
\fbaren_\f1), N \in \N\}$ is relatively compact. 
\end{theorem}
\begin{proof}
Relative compactness of the sequences $\{\fxn\}$, $\{ \fdcompn_\f1
\}$ and $\{\baren_\f1\}$ was established in Lemma \ref{tight-xn}.
Consider the family of real-valued functions on ${\cal M}_1[0,M)$
\[ \F \doteq \{F: \exists f \in \ocdbm \mbox{ such that } F(\mu) = \lan f, \mu \ran \quad \forall \mu \in {\cal M}_1[0,M)\}. \]
Then every function in $\F$ is clearly continuous with respect to
the weak topology on ${\cal M}_1[0,M)$  and the class $\F$ is trivially
closed with respect to addition. Moreover, $\F$ clearly separates
points in ${\cal M}_1[0,M)$. 
Then tightness of the sequence $\{\fnewmeasn\}$ follows from
Jakubowski's criteria, Lemma \ref{lem-tight1} and Lemma
\ref{lem-tight2}. On the other hand, $\fmeasn = \fnewmeasn -
\fidlen \delta_{0}$ and $\fidlen = 1-\lan \f1, \fmeasn \ran$ is
tight by Lemma \ref{tight-xn}. Thus the sequence $\{\fmeasn\}$ is
also tight. By Prohorov's theorem, this implies $\{ \fmeasn \}$ is
relatively compact.
\end{proof}

\subsection{Characterisation of Subsequential Limits}
\label{subs-pf3}

The main result of this section is the following theorem.

\begin{theorem}
\label{th-sublimit} Suppose Assumption \ref{ass-init} is
satisfied and $(\fx,\fmeas)$  is the limit of any subsequence of
 $\{(\fxn,\fmeasn)\}$.
If either (i) $h$ is continuous on $[0,M)$ or
(ii) $\fe$ and $\fmeas_0$ are absolutely continuous and
the set of discontinuities of $h$ in $[0,M)$ is a closed set having zero
Lebesgue measure,
then $(\fx, \fmeas)$ satisfies the fluid equations.
\end{theorem}

The rest of the section is devoted to the proof of this theorem.
Let $(\fe, \fx(0), \fmeas_0)$ be the  ${\cal S}_0$-valued random
variable that satisfies Assumption \ref{ass-init}. For simplicity
of notation, we denote the subsequence again by $N$. By Assumption
\ref{ass-init} and Theorem \ref{th-tight}, the associated sequence
\[ \{(\fen, \fxn(0), \fmeasn_0, \fxn, \fmeasn, \fdcompn_\f1, \fbaren_\f1)\}
\]  is
relatively compact.
 Hence it converges in distribution, along
a subsequence that we denote again by $N$, to a limit $(\fe,
\fx(0), \fmeas_0, \fx, \fmeas, \fdcomp_\f1, \fbare_\f1)$. We now
identify some  properties of the limit that will be used to prove
Theorem \ref{th-sublimit}. Appealing to the Skorokhod
representation theorem, we will somewhat abuse notation and assume
that $(\fxn,\fmeasn,  \fdcompn_\f1, \fbaren_\f1)$ converges almost
surely to $(\fx,\fmeas,\fdcomp_\f1,\fbare_\f1)$. Then, we
immediately have $\lan \f1, \fmeasn \ran \wconv \lan \f1,
\fmeas\ran$, and by (\ref{def-nonidling}), (\ref{def-dn}),
(\ref{def-kn}) and the fact that $\fmartn_\f1\Rightarrow \zerof$,
we infer that the non-idling condition (\ref{eq-fnonidling}) holds, that
\[  \fx = \fx (0) + \fe - \fbare_\f1 = \fx (0) + \fe - \fdcomp_\f1 \]
and that
\[ \fk = \lan \f1, \fmeas_{\cdot} \ran - \lan \f1, \fmeas_0 \ran
+ \fdcomp_\f1.
\]
Comparing the last two relations with equations (\ref{eq-fx}) and (\ref{eq-fk}),
it is clear that  in order to  prove that $(\fx, \fmeas)$
satisfies the fluid equations, one needs to show that
$\fbare_\f1 = \fdcomp_\f1 = \fd$, where $\fd$ is defined in terms
of $\fmeas$ via (\ref{def-fd}).
 If $h$ is continuous and uniformly bounded on $[0, M)$
(as is the case, for example, for the lognormal distribution),
then
  this is a simple consequence of the (almost sure)
weak convergence of $\fmeasn$ to $\fmeas$. However, additional
work is required to justify this convergence for general $h$ that
is only a.e.\ continuous on $[0,M)$.
As shown in the next two lemmas,
this is possible under some addditional, but still very mild,
assumptions on the initial conditions.

\begin{lemma}
\label{abs-nus} If $\fmeas_0$ and $\fe$ are absolutely continuous,
then so is $\fmeas_s$ for every $s \in [0,\infty)$.
\end{lemma}
\begin{proof}
Using the fact that $\fdcompn \ra \fdcomp$ and Fatou's lemma, the
proof of the second assertion of Lemma \ref{lem-fdcompn},
immediately shows that $\fdcomp_1$ is absolutely continuous.
For simplicity, we will assume that $M = \infty$ for the
rest of the proof.  The case $M < \infty$ is analogous and left to the
reader.  Since
$\fe$ is absolutely continuous by assumption, the two equations
preceding the statement of the lemma allow us to deduce that $\fx$ and $\fk$ are also absolutely
continuous. Fix $T < \infty$, and  let $C_\ve$ denote the set of
collections of finite disjoint
 intervals
$(a_i,b_i) \subset [0,T], i = 1, \ldots, n$,  such that
$\sum_{i=1}^n (b_i - a_i) \leq \ve$.  Given $\delta > 0$, choose
$\ve > 0$ small enough so that
\[  \sup_{\{(a_i,b_i)\} \in C_\ve}
\left[ \fmeas_0 \left(\cup_{i=1}^n [a_i,b_i] \right)\right] \vee
\left[ \sup_{s \in [0,T]} \sum_{i =1}^n \fk (s-b_i, s-a_i) \right]
 < \dfrac{\delta}{2}.
\]
Now, consider a particular collection $\{(a_i,b_i)\} \in C_\ve$,
and define $J_1 = \{i \in \{1, \ldots, n\}: a_i - s \geq 0 \}$ and
let $J_2 = \{1, \ldots, n\} \sm J_1$.  Then for any $s \in [0,T]$,
note that  we have
\begin{eqnarray*}
 \sum_{i=1}^n \fmeasn_s (a_i, b_i)
& = &  \sum_{i\in J_1} \fmeasn_s (a_i, b_i)  + \sum_{i \in J_2} \fmeasn_s (a_i,b_i) \\
& \leq & \sum_{i \in J_1} \fmeasn_0 (a_i - s, b_i-s) + \sum_{i \in
J_2}
\fmeasn_{s-a_i} (0,b_i - a_i) \\
& \leq &  \sum_{i \in J_1} \fmeasn_0 (a_i - s, b_i-s)+ \sum_{i \in
J_2} \fkn (s- b_i, s- a_i ).
\end{eqnarray*}
Taking limits as $N \ra \infty$ and then  the max over collections
in $C_\ve$, and using the fact that $\fkn \ra \fk$, Assumption
\ref{ass-init}(3) and Portmanteau's theorem, we see that
\begin{eqnarray*}
 \fmeas_s (\cup_{i=1}^n (a_i, b_i))  & \leq &  \liminf_{N}
 \fmeasn_s(\cup_{i=1}^n (a_i, b_i)) \\
& =  &
\liminf_{N} \sum_{i=1}^n \fmeasn_s (a_i, b_i)  \\
& \leq & \limsup_{N}  \sum_{i\in J_1} \fmeasn_0 [a_i-s, b_i-s] +
\sum_{i \in J_2} \lim_{N}  \fkn (s- b_i, s- a_i) \\
& \leq & \fmeas_0 \left(\cup_{i \in J_1} [a_i-s,b_i-s]\right) +
\sum_{i \in J_2}   \fk (s- b_i,s-a_i).
\end{eqnarray*}
Taking the supremum over all collections of intervals in $C_\ve$,
we conclude that for every $\delta > 0$ there exists $\ve > 0$
such that
\[ \sup_{\{(a_i,b_i)\} \in C_\ve\}} \sum_{i=1}^n \fmeas_s (a_i, b_i)
\leq \delta, \] which shows that $\fmeas_s$ is absolutely
continuous.
\end{proof}

\begin{lemma}
\label{lem-depart}
Suppose that for a.e.\ $s \in [0,\infty)$, $h$ is $\fmeas_s$-a.e.
 continuous on $[0, M)$, and for every $L < M$, there is a sufficiently large
 $C < \infty$  such that the boundary of the set
$\{x \in [0,L]: h(x) < C\}$ has zero $\fmeas_s$ measure.
 Then
\[ \fdcomp_\f1 = \fd =  \int_0^{\cdot} \lan  h, \fmeas_s \ran \, ds.\]
\end{lemma}
\begin{proof}
Fix $T \in (0,\infty)$.  Then for any $L < M$, we have the
elementary bound
\be
\label{basic}
\sup_{t \in [0,T]}
\left|\fdcompn_\f1 (t) - \fd (t) \right| 
\leq    R_1^{(N)} (L) + R_2^{(N)} (L) + R_3^{(N)}(L), \ee where
\begin{eqnarray*}
 R_1^{(N)} (L) & \doteq &  \int_0^T \left( \int_{(L,M)}
h(x) \fmeasn_s (dx) \right) \, ds;  \\
 R_2^{(N)} (L) & \doteq  &     \int_0^T \left|\left( \int_{[0,L]} h(x) \fmeasn_s (dx) \right)
- \left( \int_{[0,L]} h(x) \fmeas_s (dx) \right) \right| \, ds;   \\
R_3^{(N)} (L) & \doteq  & \int_0^T \left( \int_{(L, M)} h(x)
\fmeas_s (dx) \right) \, ds.
\end{eqnarray*}
Now by  Lemma \ref{lem-fdcompn}(1), we know that \be \label{know1}
 \lim_{L \ra M} \sup_{N} \E\left[ R_1^{(N)} (L)\right] = 0.
\ee On the other hand, since $\fmeasn_s
\wconv \fmeas_s$ for a.e. $s \in [0,t]$,
 by Theorem A.3.12 of \cite{dupellbook} we know
that
\[ \lan \ind_{(L,M)} h, \fmeas_s \ran \leq
\liminf_{N \ra \infty} \lan \ind_{(L,M)}  h, \fmeasn_s \ran. \]
 Combining
this with Fatou's lemma, (\ref{know1}), and the fact that
$R_3^{(N)}$ is independent of $N$, we conclude that
\be
\label{know2} \lim_{L \ra M}\E\left[ R_3^{(N)} (L) \right] \leq
\lim_{L \ra M} \liminf_{N \ra \infty} \E [R_1^{(N)} (L)] = 0.
 \ee

Now fix $\ve > 0$, and let $L < M$ be such that
\[ \limsup_{N} \E\left[ R_1^{(N)} (L)\right] \leq \dfrac{\ve}{2} \quad \quad
\mbox{ and }  \quad \quad
 \limsup_{N} \E \left[R_3^{(N)} (L) \right] \leq \dfrac{\ve}{2}.
\]
For this fixed $L < \infty$, choose $C < \infty$  large with
$\fmeas_s \left(\partial \{h < C\}\right)=0$, for
a.e. $s\in[0,T].$ Define now
\[ h_C(x)=\left\{\begin{array}{ll}
                  h(x)  &\mbox{if $h(x)<C$}\\
                  0     &\mbox{if $h(x)\ge C$}
                  \end{array}
                  \right. \]
 Then $h_C$ is bounded by $C$ and so  for a.e.\
 $s\in [0,T]$, the fact that $\fmeas_s^{(N)} \wconv \fmeas_s$ and Portmanteau's theorem imply
\[ \lim_{N \ra \infty} \left| \int_{[0,L]} h_C(x) \fmeasn_s (dx)
-  \int_{[0,L]} h_C(x) \fmeas_s (dx) \right| = 0.
\]
On the other hand, we have
\[  \int_{[0,L]}\left(h(x)-h_C(x)\right)\fmeasn_s(dx) =
 \int_{[0,L]\cap\{h\ge
C\}} h(x)\fmeasn_s(dx)  \]
and, likewise,
\[\int_{[0,L]}\left(h(x)-h_C(x)\right)\fmeas_s(dx)
=\int_{[0,L]\cap\{h\ge C\}}
h(x)\fmeas_s(dx). \]
Combining the last three relations and the bounded convergence theorem, we obtain
\[
\begin{array}{l}
\ds \limsup_N R_2^{(N)} (L)  \\
\quad \quad \ds
\le \sup_N \int_0^T\int_{[0,L]\cap\{h\ge C\}}h(x)\fmeasn_s(dx)ds
+\int_0^T\int_{[0,L]\cap\{h\ge C\}}h(x)\fmeas_s(dx)ds.
\end{array}
\]
Given any $\delta > 0$, the expectations of both terms on the
right-hand side are bounded by $\delta/2$ for all sufficiently
large $C$ --  the bound for the first term holding by an
application of relation (\ref{est-onemore}) of Lemma
\ref{lem-fdcompn}, while that for the second being a result of
Fatou's lemma applied to this relation. Since $\delta > 0$ is
arbitrary, this implies that
\[ \limsup_{N\ra\infty}\E \left[R^{(N)}_2 (L)\right] = 0. \]

Thus we have shown that
\[ \lim_{L \ra M} \limsup_{N}
\E[ R_1^{(N)} (L) +  R_2^{(N)}(L) +   R_3^{(N)} (L) ] = 0.  \]
When combined with the bound (\ref{basic}), Fatou's lemma
 and the fact that $\fdcompn \ra \fdcomp$, this implies that
\[
\E \left[ \sup_{t \in [0,T]} |\fdcomp (t) - \fd (t) | \right]
\leq \limsup_{N} \E \left[  \sup_{t \in [0,T]}|\fdcompn (t) - \fd (t) | \right]
= 0,
\]
which proves the lemma.
\end{proof}

Combining the above results, we now have a
 proof of the main result of this section. \\
\noi {\em Proof of Theorem \ref{th-sublimit}.} The conditions of
Lemma \ref{lem-depart} hold trivially if $h$ is continuous on
$[0, M)$.  On the other hand,  if $\fe$ and $\fmeas_0$ are
absolutely continuous, then due to Lemma \ref{abs-nus} the first
condition of Lemma \ref{lem-depart} is satisfied if
 $h$ is a.e.\ continuous on $[0,M)$.  If, in addition,
the set of discontinuities of $h$ are closed then the second condition of
Lemma \ref{lem-depart} is also satisfied due to Lemma \ref{lem-applem}.
 In either of
these situations, Lemma \ref{lem-depart} and the discussion
immediately prior to Lemma \ref{abs-nus} shows that $\fdn \ra \fd$
and (\ref{eq-fx}), (\ref{eq-fk}) and (\ref{eq-fnonidling}) are
satisfied. It only remains to establish (\ref{eq-ftmeas}). Since
both $y\ra\newf(y,s)$ and $y\ra
\newf_x(y,s)+\newf_t(y,s)$ are bounded and continuous, Assumption
\ref{ass-init}(3)  guarantees that
  $\lan
\newf(\cdot,0), \fmeasn_0 \ran  \Rightarrow \lan \newf(\cdot,0), \fmeas_0\ran$ and
the convergence of $\fmeasn$ to $\fmeas$ guarantees that for a.e.\
$s \in [0,t]$, $\lan \newf_x(\cdot,s)+\newf_t(\cdot,s), \fmeasn_s
\ran \Rightarrow \lan \newf_x(\cdot,s)+\newf_t(\cdot,s), \fmeas_s
\ran$. Thus, dividing (\ref{eqn-prelimit1}) by $N$, letting $N \ra
\infty$ and invoking the bounded convergence theorem, we conclude
that $\fmeas$ satisfies (\ref{eq-ftmeas}) for all $t \in
[0,\infty)$ such that $\fmeasn_t \Rightarrow \fmeas_t$.  Using
right-continuity of both sides, we then extend the result to all
$t \in [0,\infty)$. Since we used the Skorokhod Representation
Theorem,  the almost sure convergence asserted in the limits above
must be replaced by convergence in distribution. This shows the
desired result that $(\fx,\fmeas)$ satisfies the fluid equations.
\qed \\

We now obtain Theorem \ref{th-conv} as an immediate corollary. \\

\noi
{\em Proof of Theorem \ref{th-conv}. }
The first statement  of Theorem \ref{th-conv} is a direct consequence of Theorem \ref{th-tight}.
The remainder of the theorem follows from Theorem \ref{th-sublimit},
 the uniqueness of solutions to the fluid equations established in Theorem \ref{th-fluniq}
and the usual standard argument by contradiction that shows that the original
sequence converges to the solution of the fluid limit whenever the latter is unique.
\qed \\

\beginsec

\section{Convergence of the Fluid Limit to Equilibrium}
\label{sec-longtime}

In this section we characterise the behaviour of the unique solution $(\fx,\fmeas)$ to the
fluid equations, as $t \ra \infty$.

\begin{lemma}
\label{lem-idle}
Suppose $\fe$ is an absolutely continuous, strictly increasing function with
$d\fe/dt = \flam(\cdot)$ and  $(\fx, \fmeas)$ is the unique solution
to the fluid equations associated with the initial condition
$(\fe, 0, \zerof) \in \newspace$ and $h$.  If
\be
\label{def-tau1}
\tau_1 \doteq \inf \left\{t > 0: \int_0^t (1 - G(t-s)) \flam (s) \, ds = 1 \right\}
\ee
then we have the following properties.
\begin{enumerate}
\item
As $t \ra \tau_1$,
\[ \lan \f1, \fmeas_t \ran = \fx (t) \ra \int_0^{\tau_1} (1 - G(t-s)) \flam (s) \, ds, \]
and for every  function $f \in \cb$,
\[ \lan f, \fmeas_t \ran \ra
\int_{0}^{\tau_1} f(t-s) (1 - G(t-s)) \flam(s) \, ds. \]
\item
Suppose $\flam$ is constant.  Then, as $t \ra \infty$,
$\lan \f1, \fmeas_t \ran \ra \flam \wedge 1$ monotonically and, if $\flam \leq 1$, then
$\fmeas_t \wconv [\flam \wedge 1] \fmeass$ monotonically,
 where  $\fmeass$ is the
probability measure on $[0,\infty)$ with density $1-G(x)$, as defined in (\ref{def-meass}).
\item
Suppose
 $(\fe^\prime,\fx^\prime,\fmeas^\prime)$ is the solution to the fluid equations associated with any
other initial condition  
 $ (\fe^\prime, \fx^\prime(0), \fmeas_0^\prime)
\in \newspace$. If $\fe^\prime = \fe$ then for all $t \in
[0,\tau_1)$, \be \label{ineq-fidle} \lan \f1, \fmeas_t^\prime \ran
\geq \lan \f1, \fmeas_t \ran. \ee
\end{enumerate}
\end{lemma}
\begin{proof}
Let $(\fx, \fmeas)$ be the unique solution to the fluid equations associated with
the initial conditions $(\fe,0,\zerof)$, and
define
\[ \tau \doteq \inf \{t > 0: \lan \f1, \fmeas_t \ran  = 1 \}. \]
Then $\tau > 0$ by continuity and for $t \in (0,\tau)$, $\lan \f1, \fmeas_t \ran < 1$.
The non-idling property (\ref{eq-fnonidling})  then implies that $\fx(t) = \lan \f1, \fmeas_t \ran < 1$.
Combining relations (\ref{eq-fx}) and (\ref{eq-fk}), we see that $\fk(t) = \fe(t)$ for
$t \in [0,\tau)$.
Since $\fmeas$ and $\fk$ satisfy the fluid equation (\ref{eq-fmeas}) and $\fmeas_0 \equiv 0$,
Theorem \ref{th-pde} shows that for any $f \in \cb$ and $t \in [0,\tau)$,
\[ \lan f, \fmeas_t \ran = \int_0^t f(t-s) (1 - G(t-s)) \, d\fe(s) = \int_{0}^{t} f(t-s) (1 - G(t-s)) \flam(s) \, ds.  \]
Substituting $f = \f1$, the left-hand side above is equal to $\fx (t) = \lan \f1, \fmeas_t \ran$, from which it
is easy to see that  $\tau = \tau_1$.
Since the integrand on the right-hand side is continuous in $t$, this proves property (1).

Now, consider the time-homogeneous setting where $\flam(\cdot) \equiv \flam$ is constant.
When $\flam < 1$ or $\flam = 1$ and the support of $1 - G$ is $[0,\infty)$,
 the results in (1) show that $\tau_1 = \infty$ and $\fmeas_t \wconv \flam \fmeass$ monotonically as $t \ra \infty$. 
 If
$\flam \geq 1$, then it is easy to verify directly that the unique
fluid solution satisfies for every $t \geq \tau_1$,
\[ \fx (t) = 1 + (\flam - 1)(t- \tau_1) \quad \quad \mbox{ and }
\quad \quad \fmeas_t = \fmeass.
\]
Indeed, this is a simple consequence of the fact that with this
definition  for $t \geq \tau_1$, $\fx (t) \geq 1$ and $\lan \f1,
\fmeas_t \ran = \f1 = \f1 \wedge \flam$. This completes the proof
of property 2.

For the last property we first show that for every $t\in
(0,\tau_1),$ $\fx^\prime (t)\ge \fx(t).$ To this end we use (3.6),
(3.8) and the non-idling condition (3.7) to deduce that for every
$t,$ $\fk^\prime(t)\le (\fx^\prime(0)-\lan
1,\fmeas_0^\prime\ran)^+ +\fe(t).$
It now follows from
(\ref{eq-fk}) and (\ref{eq2-fmeas}) that
 \begin{eqnarray*}
\fx^\prime(t) & = & \fx^\prime(0)+\fe(t)-\int_{[0,M)}
\frac{G(x+t)-G(x)}{1-G(x)}\fmeas^\prime_0(dx)-\int_0^t
g(t-s)\fk^\prime(s)ds \\&\ge &
\fx^\prime(0)+\fe(t)-\int_{[0,M)}
\frac{G(x+t)-G(x)}{1-G(x)}\fmeas^\prime_0(dx) \\
& & \quad -\int_0^t
g(t-s)((\fx^\prime(0)-\lan 1,\fmeas^\prime_0 \ran )^++\fe(s))ds\\
&\ge & \fx^\prime(0)+\fe(t)-\lan 1,\fmeas^\prime_0\ran -
\int_0^t\fe(s)g(t-s)ds-(\fx^\prime(0) -\lan 1,\fmeas^\prime_0\ran)^+G(t)\\
&\ge&
\fx^\prime(0)+\fe(t)-\int_0^t\fe(s)g(t-s)ds-(\fx^\prime(0)-\lan
1,\fmeas^\prime_0\ran)^+ -\lan 1,\fmeas^\prime_0\ran.
\end{eqnarray*}
We now recall that the
non-idling condition implies that $(\fx^\prime(0)-\lan
1,\fmeas^\prime_0\ran)^++\lan 1,\fmeas^\prime_0\ran
=\fx^\prime(0)$ and therefore $ \fx^\prime(t)\ge
\fe(t)-\int_0^t\fe(s)g(t-s)ds=\fx(t),$ where the last equality
follows by combining relations (\ref{eq-fx}), (\ref{eq-fk}),
(\ref{eq2-fk}) and the fact that, due to (\ref{def-fsmallk}),
$\fk(s)=\fe(s)$ for all $s\in[0,\tau_1)$. To prove our result we
note that if $\fx^\prime(t)\ge 1$ then $\lan 1,\fmeas^\prime_t
\ran=1\ge \lan 1,\fmeas_t\ran$, whereas if $\fx^\prime(t)< 1$ then
$\lan 1,\fmeas^\prime_t\ran=\fx^\prime(t)\ge \fx(t)=\lan
1,\fmeas_t\ran$ which proves the last part of our lemma.
\end{proof}

\noi
{\bf Proof of Theorem \ref{th-larget}. }
The first statement of Theorem \ref{th-larget} follows from properties (1) and (2) of Lemma \ref{lem-idle}.
For the second statement of the theorem,
consider an arbitrary initial condition of the form
$(\flam \f1,\fx(0), \fmeas_0) \in \newspace$ for $\flam \in (0,\infty)$, let
$(\fx, \fmeas)$  be the unique solution to the associated fluid equations and let
$\fk$ and $\fd$ be the related processes defined in
(\ref{eq-fk}) and (\ref{def-fd}), respectively.
Since $\fe = \flam \f1$ is absolutely continuous, by Theorem \ref{th-fluniq}
 $\fk$ is also absolutely continuous.
Let $\fsmallk$  denote the
derivative of $\fk$ and
also define
\be
\label{def-fidle}
\fidle (t)  \doteq 1 -\lan \f1, \fmeas_t \ran  \quad \quad \mbox{ for } t \in [0,\infty).
\ee
Then Lemma \ref{lem-idle} shows that $\fidle (t) \ra 0$ as $t \ra \infty$.
We now consider the following three mutually exhaustive cases.  From the proof of Lemma \ref{lem-idle}
it is clear that if $\flam > 1$ then only Case 2 is possible.  \\
{\em Case 1.} There exists $T_0 \in (0,\infty)$ such that $\fidle (t) > 0$ for all $t \geq T_0$. \\
In this case, $\flam \equiv 1$ and by (\ref{def-fsmallk}), it follows that
 $\fsmallk (t) = 1$ for $t \geq T_0$.
As a result, by Lemma \ref{lem-sgroup} and equation (\ref{eq2-fmeas}) of Theorem \ref{th-pde},
  for any $u \geq 0$ and $t = T_0 + u$, we have
\[ \lan f, \fmeas_t \ran = \int_{[0,M)} f(x+u) \dfrac{1 - G(x+u)}{1 - G(x)} \, \fmeas_{T_0} (dx)
+ \int_0^u f(u-s)(1 - G(u-s))  \, ds.
\]
Since $f$ is uniformly bounded on $[0,\infty)$ and for every $x \in (0,\infty)$,
$(1 - G(x+u))/(1 - G(x)) \ra 0$ as $u \ra \infty$, by the bounded convergence theorem,
the first term converges to zero as $u \ra \infty$.
On the other hand, the second term trivially
converges to $\int_0^\infty f(x) (1 - G(x))  \,  dx = \lan f, \fmeass \ran$ as $u \ra \infty$, which
completes the proof of  the theorem in this case.  \\

\noi
{\em Case 2. } There exists $T_0 < \infty$ such that $\fidle (t) = 0$ for all $t \geq T_0$.   \\
In this case for all $u \geq 0$, $\lan \f1, \fmeas_{T_0 + u}\ran
= \lan \f1, \fmeas_{T_0}\ran = 1$.
Together with Corollary \ref{cor-pde} and Lemma \ref{lem-sgroup}, this shows that
for every $u \geq 0$,
\[
 \begin{array}{l}
\fk (T_0 + u) - \fk (T_0)  \\
\quad \quad \quad  =  \ds   \int_{[0, M)} \dfrac{G(x+u) - G(x)}{1 - G(x)} \, \fmeas_{T_0} (dx)
+ \int_{0}^u g(u-s) \left( \fk (T_0 + s) - \fk (T_0) \right) \, ds \\
\quad \quad \quad  =  \ds  \int_{[0, M)} \dfrac{G(x+u) - G(x)}{1 - G(x)} \, \fmeas_{T_0} (dx)
 + \int_0^u  G(u-s) \fsmallk (T_0 + s) \, ds,
 \end{array}
\]
where the last equality follows by integration by parts.
Differentiating the last equation with respect to $u$, we obtain the renewal equation
\[ \begin{array}{rcl}
\fsmallk (T_0 + u) &  = & \ds \int_{[0, M)} \dfrac{g(x+u)}{1 - G(x)} \fmeas_{T_0} (dx)
 + \int_0^u g (u-s) \fsmallk (T_0 + s) \, ds.
\end{array}
\]
Since $g$ is directly Riemann integrable by assumption, an application of
the key renewal theorem (see \cite{asmbook}) and Fubini's theorem yields the relation
\[\ba{rcl} \ds
 \lim_{u \ra \infty} \fsmallk (T_0 + u) & = & \ds
 \int_0^\infty \left( \int_{[0,M)} \dfrac{g(x+u)}{1 - G(x)} \, \fmeas_{T_0} (dx) \right)\, du \\
& = &  \ds \int_{[0, M)} \left( \int_0^\infty \dfrac{g(x+u)}{1 - G(x)} \, du \right)\fmeas_{T_0} (dx)
= \ds \fmeas_{T_0}([0,M))  =   1.
\end{array}
\]
Now fix $\ve > 0$ and choose $T_0^\prime < \infty$ such that $|\fsmallk (T_0^\prime + u) - 1| \leq \ve$ for
every $u \geq 0$.
Then  (\ref{eq2-fmeas})  and another application of Lemma \ref{lem-sgroup} show that  for any $f \in \cb$
and any $t = T_0^\prime + u$ with $u \geq 0$,  we have
\[
\lan f, \fmeas_{t} \ran  =
\int_{[0, M)} f(x+u) \dfrac{1 - G(x+u)}{1 - G(x)} \, \fmeas_{T_0^\prime} (dx)
 + \int_0^u  f(u-s) (1 - G(u-s)) \fsmallk (T_0^\prime + s)  \, ds.
\]
Sending $u \ra \infty$, and thus $t \ra \infty$, since $1 - G(x+t) \ra 0$,
$\fmeas_{T_0}$ is a probability measure on $[0,M)$,  and $f$ is uniformly bounded on $\R_+$,
the first term on the right-hand side converges to $0$ by the bounded convergence theorem.
 On the other hand, since $1 - G$ is a probability density, $f$ is uniformly
bounded on $[0,\infty)$, and $|\fsmallk (T_0^\prime + s) - 1| \leq \ve$ for all $s \geq 0$,
 another application of the bounded convergence theorem shows that
\[  \left|\lim_{t \ra \infty} \lan f, \fmeas_{t} \ran  - \int_0^\infty f(x) (1 - G(x)) \, dx \right|  = O (\ve).
\]
Sending $\ve \ra 0$ proves the theorem in this case.  \\

\noi
{\em Case 3.} For every $T < \infty$ there exists $s, t \geq T$ such that  $\fidle (s) > 0$ and
$\fidle (t) = 0$.  \\
In this case, for every $\ve > 0$ there exists $T_\ve > 0$ such that $\fidle({T_\ve}) = 0$ (or, equivalently,
$\lan \f1, \fmeas_{T_\ve}\ran = 1$) and $\fidle (t) < \ve$ for all $t \geq T_\ve$.  The latter assertion
follows from Lemma \ref{lem-idle}(3) and the fact that $\lan \f1, \fmeas_t \ran \ra 1$ as
$t \ra \infty$, as proved in Lemma \ref{lem-idle}(2).
For $u \in [0,\infty)$, define
\[ \fko (u) \doteq \fk (T_\ve + u) - \fk (T_\ve) \quad \quad \mbox{ and } \quad \quad \fkp (u) \doteq \fd (T_\ve + u) - \fd (T_\ve),
\]
and note that by relations (\ref{eq-fk}) and (\ref{def-fd}), it follows that
\be
\label{ve-bound}
  \sup_{u \in [0,\infty)} \nrm{\fko - \fkp}_u \leq \ve.
\ee
Now, for $i = 1, 2$, let $\fsmallk^{(i)}$ denote the derivatives of $\fk^{(i)}$, respectively,
and set $\fmeaso_u \doteq \fmeas_{T_\ve +u}$.
Then  equation (\ref{eq2-fmeas}), with $\newfk = \fk$,
 and Lemma \ref{lem-sgroup} imply that for $f \in \ocdb$,
\[
\ba{rcl}
\ds \lan f, \fmeaso_{u} \ran & = &  \ds \int_{[0, M)} f(x+u)\dfrac{1 - G(x+u)}{1 - G(x)}  \, \fmeas_{T_\ve} (dx) \\
 & & \ds + \int_0^u   f(t-s) (1 - G(u-s))  \fsmallk^{(1)} (T_\ve + s)\, ds.
\ea
\]
Next, we define $\fsmallk^{(2)} (\cdot)\doteq \lan h, \fmeas_\cdot \ran$ to be the departure
rate and let $\fmeasp$ be the measure on $[0, \infty)$ that satisfies the
last equation for every $f \in \cb$, but with  $\fsmallk^{(1)}$ replaced by
$\fsmallk^{(2)}$.
Also, note that  Corollary \ref{cor-pde}, relation (\ref{eq-fk}) and Lemma \ref{lem-sgroup}
imply that for $u \geq 0$
\[ \fk^{(2)} (T_\ve + u) - \fk^{(2)} (T_\ve) =
\int_{[0, M)} \dfrac{G(x+u) - G(x)}{1-G(x)} \, \fmeas_{T_\ve} (dx)
+ \int_0^u G(u-s) \fsmallk^{(2)} (T_\ve + s) \, ds.
\]
The same arguments as in Case 2  can then be used to show that
$\fsmallk^{(2)}$ satisfies the renewal equation
\[ \begin{array}{rcl}
\ds  \fsmallk^{(2)} (T_\ve + u) &  = & \ds \int_{[0, M)} \dfrac{g(x+u)}{1 - G(x)} \fmeas_{T_\ve} (dx)
 + \int_0^u g (u-s) \fsmallk^{(2)} (T_\ve + s) \, ds,
\end{array}
\]
and hence  that, for
every $f \in \cb$,
$\fsmallk^{(2)} (t) \ra 1$ and
$\lan f, \fmeas^{(2)}_t \ran  \ra \lan f, \fmeass \ran$ as $t \ra \infty$.
On the other hand,  comparing the equations for
$\lan f,\fmeas^{(i)} \ran$, $i = 1, 2$, given above, using Lemma \ref{lem-elem}
and the bound (\ref{ve-bound}), we have for every $f \in \ocdb$,
\[ \nrm{\lan f, \fmeaso_{\cdot} \ran - \lan f, \fmeasp_{\cdot} \ran}_\infty \leq
\left(2\nrm{f}_\infty + \nrm{f^\prime}_\infty \right) \ve. \]
Since $\ve \ra 0$, this proves the property for Case 3, and thus concludes the proof of the theorem.

\begin{appendix}

\renewcommand{\theequation}{A1.\arabic{equation}}

\beginsec

\section{Explicit Construction of the State Process}
\label{ap-markov}

In this section, we provide an explicit pathwise construction of the state
process
$(\ren, \xn, \measn)$ and auxiliary processes described in Sections \ref{subs-modyn} and \ref{subs-aux}.

Suppose that we are given the sequence of service requirements
$\{v_j, j \in \{-N+1, \ldots, 0\} \cup \N\}$ and initial conditions
 $\measn_0$ and $\xn (0)$ satisfying
\[  N - \lan \f1, \fmeasn_0 \ran = [N - \xn(0)]^+, \]
 as well as the  residual inter-arrival process
$\ren$ that takes values in $\mathcal{D}_{\R_+} [0,\infty)$. Here,
$\measn_0$ is a sum of  $\lan \f1, \measn_0\ran \leq N$ Dirac
masses.  Let $\agen_j (0)$, $j = -\lan \f1, \measn_0 \ran + 1$,
$-\lan \f1, \measn_0 \ran + 2, \ldots, 0$ be the positions of the
Dirac masses, ordered in increasing order. Then $\agen_j$
represents the age of the $j$th customer and so  we can  assume,
without loss of generality, that each $\agen_j < v_j$ and hence
write
\[  \measn_0 = \sum_{j=-\lan \f1, \measn_0 \ran + 1}^0 \delta_{\agen_j(0)}
= \sum_{j=-\lan \f1, \measn_0 \ran + 1}^0 \delta_{\agen_j(0)}  \ind_{\{\agen_j(0) < v_j\}}.
\]
The cumulative arrival process $\en$ is easily reconstructed from
$\ren$: $\en(t)$ is simply the number of jumps of $\ren$ in
$[0,t]$ and is thus adapted with respect to the filtration
generated by $\ren$. The evolution of the processes $\xn$ and
$\measn$ can then  be described recursively as follows. Define
$\ttime_0 \doteq \ttimen_0 \doteq 0, \dn(0) \doteq 0,  \kn(0)
\doteq 0.$
Since the ages of customers that have already entered service increase linearly with time while
in service,
 for $l \doteq 0$ and
 $j \in \{-\lan \measn_0, \f1\ran  + 1, -\lan \measn_0, \f1 \ran +2,
\ldots, \kn (\ttime_l) \}$,  we define
 \be
\label{def2-agen}
\agen_j(t) \doteq
\left\{
\begin{array}{rl}
\agen_j(\ttime_l) + t - \ttime_l & \mbox{ if } \agen_j(t) < v_j \\
v_j  & \mbox{ otherwise. }
\end{array}
\right.
\ee
Also, we define
\be
\label{def-tdn}
 \tdn_{l} (t) \doteq
 \sum\limits_{j=-\lan \measn_0, \f1 \ran + 1}^{ \kn (\ttime_l)} \ind_{\{\agen_j (t) \geq v_j\}}
 \quad \mbox{ for } t \in [\ttime_l,\infty).
\ee
The process $\tdn_{l}$, defined on $[\ttime_l, \infty)$,
represents the cumulative departure process, assuming there are no
new arrivals into the system after $\ttime_l$.
Now, let
$\tau_{l+1} \doteq \tau_{l+1}^{(N)}$ be the first time after $\tau_l$ when
there is either an arrival into or departure from the system.
Since $\tau_l + \resaren (\tau_l)$ is the first time after
$\tau_l$ that there is an arrival into the system, we have \be
\label{def-ttime} \ttime_{l+1} \doteq  [\tau_l + \resaren
(\tau_l)] \wedge \inf\left\{t > \ttime_l: \tdn_{l} (t) >
\dn(\ttime_l) \right\}. \ee
By  assumption, we have  $\agen_j (0) < v_j$ for $j \in \{-\lan \measn_0, \f1 \ran + 1, \ldots,
0\}$ and $\en(0) = 0$, and therefore it follows that $\tdn_{0} (0) = 0 =
\dn(0)$. Hence for $t \in [\ttime_l,\ttime_{l+1}]$, we set
\begin{eqnarray}
\label{def2-dn}
\dn(t) & \doteq & \tdn_l (t)  \\
\label{def2-agen2}
    \agen_j(t) & \doteq &  0 \quad\quad \quad \quad \mbox{ for }  j > \kn (\ttime_l) \\
\label{def2-xn} \xn(t) &  \doteq &\xn(0)+ \en(t) - \dn(t).
\end{eqnarray}
In addition, for $t \in (\ttime_l, \ttime_{l+1}]$,  the quantity
$\kn (t)$ is defined in terms of $\xn(t)$ and  $\dn(t)$ via the
relation \be \label{kn}
 \kn(t) \doteq  \dn(t) + \xn(t) \wedge N-\xn(0) \wedge N
\ee
and, in turn, $\measn(t)$ is defined in terms of $\kn$ and $\{\agen_j\}$ by
\be
\label{def2-nun}
 \nun (t) \doteq \sum\limits_{j = -\lan \f1, \measn_0 \ran + 1}^{\kn(t)} \delta_{\agen_j(t)}
\ind_{\{\agen_j (t) < v_j\}}.
\ee
This defines the processes $\xn$ and $\measn$ on $[0,\tau_{l+1}]$.

To proceed with the iteration, first note that
 since $\en(\ttime_{l+1})$, $\xn(\ttime_{l+1})$ and  $\agen_j, j \in
\{-\lan \f1, \measn_0\ran + 1, \ldots,  \kn (\ttime_{l+1})\} $
are now well-defined, $\tdn_{l+1}$ and
 $\ttime_{l+2}$ are also  well-defined by equations (\ref{def-tdn}) and
 (\ref{def-ttime}).  Therefore, replacing $l$ by $l+1$, the  recursive relations above can be used
to extend the definition of the processes $\xn$ and $\measn$ to
$[0, \ttime_{l+2}]$.
In order to show that this procedure can be used to
 construct the process $(\ren, \xn, \measn)$ on the whole of
$[0,\infty)$, it suffices  to show that $\ttime_{l} \ra \infty$
a.s.\ as $l \ra \infty$. To see why this is true, note that
arrival and departure events are generated by either the
departures of one of the $\xn(0) $ customers already present at
time $0$ or by the arrival and/or departure of a customer that
arrived after time $0$. Thus for  $l$ departure or arrival events
to have occurred,
 there must have been at least $(l - \xn(0))/2$ new customer arrivals
after time $0$. In other words,
 $\ttime_{l} \geq U_{[(l- \xn(0))/2]\vee 0}$,
where  $U_j$ is the time of arrival of the $j$th
customer after time $0$.
Since $\xn(0)$ is a.s.\ finite  and
 $U_l \ra \infty$ as $l \ra \infty$ due to the assumption that $\en(t)$ is finite for all $t$
 finite, this implies that $\ttime_l \ra \infty$ a.s.\ as $l \ra \infty$.
 Therefore the initial conditions, along with the recursive
equations (\ref{def2-agen})--(\ref{def2-nun}), define a.s. every
sample path of  $\en$, $\xn$ and $\nun$ on $[0,\infty)$.

\begin{remark}
\label{rem-aux} {\em
The above construction shows that the auxiliary processes
satisfy relations (\ref{def-dn})--(\ref{def-nonidling}).
Indeed, equations (\ref{def-dn}), (\ref{def-agejn}) and (\ref{def-nun})
are an immediate consequence of (\ref{def2-xn}),
(\ref{def2-agen}), (\ref{def2-agen2}) and (\ref{def2-nun}).
Moreover,  relations (\ref{def-tdn}), (\ref{def2-dn}) and
(\ref{def2-nun}), when combined, yield  relation (\ref{def-kn}).
In turn, equations (\ref{def-kn}) and (\ref{kn})
show that the idle server process $\idlen \doteq N - \lan \f1, \measn\ran$
satisfies the non-idling condition (\ref{def-nonidling}).
}
\end{remark}

\begin{remark}
\label{rem-assign}
{\em  From the above construction, it is also clear that the
value of $(\ren,\xn,\nun)$ does not depend on the method of assignment of
customers to servers, as long as  the non-idling requirement -- that a server
never be idle when there is a customer that needs service -- holds.
In order to implement any specific assignment policy, one may need to keep
track of additional variables.  For example,
if one were to
consider a random assignment of an arriving customer to idle servers,
then one would have to enlarge the probability space to accommodate an
infinite sequence of i.i.d.\ random variables, uniformly distributed in $(0,1)$,
that will be used to implement this random choice.   In this case,
the filtration ${\cal F}_{\tau_l}$ would have to be enlarged to include the
uniform random variable that makes this choice, resulting in a corresponding enlargement
of the filtration ${\cal F}_t^{(N)}$.
On the other hand, if one were to assign a customer to the server that has been idle
the longest since it last completed the processing of a customer, then one would
have to keep track of the identities of individual servers and the time elapsed
since completion of their last completion of service.
Since these details do not play a role in our analysis, we omit them.
 }
\end{remark}

\beginsec

\section{An Elementary Lemma}
\label{app-elemma}

\begin{lemma}
\label{lem-applem}
If, for every $L \in (0,\infty)$,
 the set of discontinuities of a function $h$  on
$[0,L]$ is a closed set of zero Lebesgue measure,
then for every $C < \infty$, the boundary of the set
\[  \left\{x \in [0,L]: h(x) < C \right\}  \]
has zero Lebesgue measure.
\end{lemma}
\begin{proof}
Let $J$ be the set of discontinuities of $h$ on $[0,L]$.
Since $J$ is closed, $[0,L] \sm J$ is open and  therefore
admits a representation as a countable union of open intervals $\cup_{i} (a_i, b_i)$.
Thus for any $C < \infty$, we can write
\[  \{x \in [0,L]: h(x) < C \} = [\cup_{i} {\cal O}_i] \cup \{ x \in J: h(x) < C \}\]
where
\[ {\cal O}_i \doteq \{ x \in (a_i,b_i): h (x) < C \}.
\]
This implies that
\[  \partial \{x \in [0,L]: h (x) < C \} \subseteq
\left[ \cup_{i} \partial {\cal O}_i \right]
\cup
\partial \{x \in J: h (x) < C \}  \subseteq \left[ \cup_{i} \partial {\cal O}_i \right]
\cup J,
\]
where the last inclusion follows due to the assumption that $J$ is closed.
Since $J$ has zero Lebesgue measure by assumption,
 to prove the lemma it suffices to show that each $\partial {\cal O}_i$
has zero Lebesgue measure.
For each $i$, since $h$ is continuous on $(a_i, b_i)$, the set
${\cal O}_i$  is open and can therefore be
written as a countable union of intervals.  Its boundary is
therefore at most countable and so has Lebesgue measure zero.
\end{proof}

\end{appendix}

\noi
{\bf Acknowledgments.}  We are grateful to Luc Tartar for several useful discussions on the
generalisation of Theorem 4.1 to general $h \in \lloc [0,M)$ (see Remark \ref{rem-gencont}).
In addition, we would  like to thank Avi Mandelbaum for bringing this open problem to our attention.
The second author would also like to thank Chris Burdzy and Z.-Q. Chen for their hospitality during her stay
at the University of Washington in the Fall of 2006,  when most of this work
was completed.

\bibliographystyle{amsplain}

\bibliography{refs4}
\providecommand{\bysame}{\leavevmode\hbox to3em{\hrulefill}\thinspace}
\providecommand{\MR}{\relax\ifhmode\unskip\space\fi MR }
\providecommand{\MRhref}[2]{%
  \href{http://www.ams.org/mathscinet-getitem?mr=#1}{#2}
}
\providecommand{\href}[2]{#2}

\end{document}